\documentclass[11pt]{article}

\usepackage{amssymb,amsmath,amsthm,enumitem,mathrsfs,mathabx,accents,xy}
\xyoption{all}
\usepackage{tikz}
\usetikzlibrary{cd}

\makeatletter
\renewenvironment{proof}[1][\proofname]{\par
  \pushQED{\qed}%
  \normalfont \topsep-5\p@\@plus6\p@\relax
  \trivlist
  \item[\hskip\labelsep
    #1\@addpunct{.}]\ignorespaces
}{%
  \popQED\endtrivlist\@endpefalse
}
\makeatother

\usepackage{titlesec}
\titlespacing\section{0pt}{-3pt plus 2pt minus 1pt}{-7pt plus 2pt minus 1pt}
\titlespacing\subsection{0pt}{-3pt plus 2pt minus 1pt}{-7pt plus 2pt minus 1pt}
\titlespacing\subsubsection{0pt}{10pt plus 4pt minus 2pt}{1pt plus 2pt minus 2pt}

\numberwithin{equation}{section}

\newtheoremstyle{reduced_space}{}{-0.5\baselineskip}{}{}{\bfseries}{}{.5em}{}
\theoremstyle{reduced_space}

\newtheorem{thm}{Theorem}[section]
\newtheorem{prp}[thm]{Proposition}
\newtheorem{lmm}[thm]{Lemma}
\newtheorem{crl}[thm]{Corollary}

\theoremstyle{definition}

\theoremstyle{remark}
\newtheorem{rmk}[thm]{Remark}

\def\BE#1{\begin{equation}\label{#1}}
\def\EE{\end{equation}}
\def\eref#1{(\ref{#1})}

\def\BEnum#1{\begin{enumerate}[label=#1,leftmargin=*,topsep=-10pt,itemsep=-3pt]}
\def\EEnum{\end{enumerate}}

\def\ov#1{\overline{#1}}
\def\sf#1{\textsf{#1}}
\def\wt#1{\widetilde{#1}}
\def\tn#1{\textnormal{#1}} 
\def\lr#1{\langle{#1}\rangle}
\def\blr#1{\big\langle{#1}\big\rangle}

\def\wch#1{\widecheck{#1}}
\def\wh#1{\widehat{#1}}

\def\sm#1{\begin{small}#1\end{small}}

\def\lra{\longrightarrow}
\def\xlra#1{\xrightarrow{{#1}}}
\def\lhra{\lhook\joinrel\xrightarrow}

\def\cA{\mathcal A}
\def\C{\mathbb C}
\def\cC{\mathcal C}
\def\bfC{\mathbf C}
\def\D{\mathbb D}

\def\cH{\mathcal H}

\def\cJ{\mathcal J}
\def\cL{\mathcal L}
\def\M{\mathfrak M}
\def\cM{\mathcal M}
\def\cN{\mathcal N}

\def\P{\mathbb P}
\def\R{\mathbb R}

\def\cS{\mathcal S}

\def\cZ{\mathcal Z}
\def\Z{\mathbb Z}

\def\al{\alpha}

\def\de{\delta}
\def\ep{\epsilon}
\def\ga{\gamma}
\def\io{\iota}

\def\la{\lambda}
\def\si{\sigma}
\def\om{\omega}
\def\th{\theta}
\def\ve{\varepsilon}
\def\vp{\varpi}
\def\vph{\varphi}
\def\ze{\zeta}

\def\De{\Delta}
\def\Ga{\Gamma}
\def\La{\Lambda}
\def\Om{\Omega}
\def\Si{\Sigma}
\def\Th{\Theta}
\def\Ups{\Upsilon}

\def\fc{\mathfrak c}
\def\fd{\mathfrak d}
\def\ff{\mathfrak f}
\def\fI{\mathfrak i}
\def\fj{\mathfrak j}
\def\fo{\mathfrak o}

\def\fp{\mathfrak p}
\def\fs{\mathfrak s}

\def\bh{\mathbf h}
\def\bp{\mathbf p}
\def\u{\mathbf u}

\def\codim{\tn{codim}}
\def\CR{\tn{CR}}
\def\nd{\tn{d}}
\def\tne{\tn{e}}
\def\dim{\tn{dim}}
\def\ev{\tn{ev}}
\def\id{\tn{id}}

\def\nod{\tn{nd}}
\def\PD{\tn{PD}}
\def\Pin{\tn{Pin}}
\def\pt{\tn{pt}}
\def\PSL{\tn{PSL}}

\def\top{\tn{top}}

\def\node{\tn{node}}

\def\i{\infty}

\def\eset{\emptyset}
\def\prt{\partial}
\def\dbar{\ov\partial}
\def\st{\bigstar}

\def\bu{\bullet}

\def\oGa{\Ga}
\def\OGa{\accentset{\circ}{\Ga}}
\def\oS{S}

\begin{document}

\title{Steenrod Pseudocycles, Lifted Cobordisms, and\\
Solomon's Relations for Welschinger's Invariants}
\author{Xujia Chen\thanks{Supported by NSF grant DMS 1500875}}
\date{\today}

\maketitle

\vspace{-.2in}

\begin{abstract}
\noindent
We establish two WDVV-style relations for the disk invariants
of real symplectic fourfolds by implementing Georgieva's suggestion
to lift homology relations from the Deligne-Mumford moduli spaces of stable real curves.
This is accomplished by lifting judiciously chosen cobordisms realizing
these relations.
The resulting lifted relations lead to the recursions for Welschinger's invariants
announced by Solomon in~2007 and have the same structure as his WDVV-style relations, 
but differ by signs from the latter.
Our topological approach provides a general framework for lifting relations via morphisms
between not necessarily orientable spaces.
\end{abstract}

\tableofcontents
\setlength{\parskip}{\baselineskip}

\section{Introduction}
\label{intro_sec}

The WDVV relation~\cite{KM,RT} for genus~0 Gromov-Witten invariants completely solves 
the classical problem of enumerating complex rational curves in the complex projective space~$\P^n$.
Invariant counts of real rational $J$-holomorphic curves in compact real symplectic fourfolds,
now known as \sf{Welschinger's invariants},
were defined in~\cite{Wel4} and interpreted in terms of counts of $J$-holomorphic maps
from the disk~$\D^2$ in~\cite{Jake}.
J.~Solomon announced two distinct WDVV-type relations for these counts 
in February~2007 and outlined an approach to their proof in the general spirit of
the original proof of the complex WDVV relation in~\cite{RT}.
However, the outline of the proof described in~\cite{Jake2} left 
a number of conceptual points mysterious and 
clearly required a major technical effort to implement.

The proof of the complex WDVV relation in~\cite{RT} involves defining a count of
$J$-holomorphic maps in a symplectic manifold~$(X,\om)$ for every cross-ratio~$\vp$
of four points on~$\P^1$ and showing that this count does not depend on~$\vp$.
This is also the strategy used in the alternative proof of a WDVV-type relation
for counts of real rational curves passing through conjugate pairs of points only
(no real point constraints) in~\cite{RealEnum}.
The complex WDVV relation can alternatively be viewed as a direct consequence
(at least conceptually) of two specific points, $\vp_1$ and~$\vp_2$, 
of the Deligne-Mumford moduli space $\ov\cM_{0,4}\!\approx\!\P^1$ of 
stable complex genus~0 curves with 4~marked points determining the same element of~$H_0(\ov\cM_{0,4})$.
This perspective is suitable for lifting homology relations in any dimension
from moduli spaces of curves to moduli spaces of $J$-holomorphic maps 
and has proved instrumental to studying the structure of complex Gromov-Witten invariants
as in~\cite{Gi01,Teleman}.
This is also the strategy used in the primary proof of a WDVV-type relation
for counts of real rational curves passing through conjugate pairs of points 
in~\cite{RealEnum}.
In all of these settings, the moduli spaces of curves and maps are closed and oriented.

\subsection{Lifting homology relations}
\label{liftrel_subs}

In Spring~2014, P.~Georgieva suggested that WDVV-type recursions for the real genus~0 invariants 
of~\cite{Wel4} might be obtainable by lifting
\BEnum{($\R\arabic*$)}

\item a zero-dimensional homology relation on the moduli space 
$\R\ov\cM_{0,1,2}\!\approx\!\R\P^2$ of stable real
genus~0 curves with 1~real marked point and 2~conjugate pairs of marked points and

\item the one-dimensional homology relation on the moduli space $\R\ov\cM_{0,0,3}$ of stable real
genus~0 curves with 3~conjugate pairs of marked points discovered in~\cite{RealEnum}

\EEnum
to the moduli spaces $\ov\M_{k,l}(B;J)$ of real rational $J$-holomorphic maps 
constructed in~\cite{Penka2}.
Unlike in the complex case and in the real case considered in~\cite{RealEnum}, 
major conceptual issues arise in lifting relations from 
$\R\ov\cM_{0,1,2}$ and $\R\ov\cM_{0,0,3}$ to $\ov\M_{k,l}(B;J)$
and in translating any lifted relations into invariant counts of~curves
because the moduli spaces $\ov\M_{k,l}(B;J)$ are generally not orientable.
The present paper deals with these issues by lifting homology relations {\it along with}
bounding cobordisms for them to cuts of $\ov\M_{k,l}(B;J)$ along certain hypersurfaces.

We first re-interpret the disk counts of~\cite{Jake} in the spirit of 
Steenrod homology~\cite{Steenrod} in terms of counts of real $J$-holomorphic maps
with marked points decorated by signs as in~\cite{Penka2}.
We then lift homology relations, along with {\it suitably chosen} bounding chains~$\Ups$, 
from  $\R\ov\cM_{0,1,2}$ and $\R\ov\cM_{0,0,3}$ to
the {\it bordered} moduli spaces $\wh\M_{k,l;l^*}(B;J)$ obtained by cutting $\ov\M_{k,l}(B;J)$
along hypersurfaces that obstruct the relative orientability of the forgetful morphisms
\BE{ffdfn_e0}\ff_{1,2}\!:\ov\M_{k,l}(B;J)\lra\R\ov\cM_{0,1,2} \quad\hbox{and}\quad  
\ff_{0,3}\!:\ov\M_{k,l}(B;J)\lra\R\ov\cM_{0,0,3}.\EE
The simple topological Lemma~\ref{InterOrient_lmm} expresses the wall-crossing effects 
on the lifted relations in $\ov\M_{k,l}(B;J)$ in terms of the intersections of
the boundary of $\wh\M_{k,l;l^*}(B;J)$ with~$\Ups$.
This allows us to obtain the two WDVV-type relations for the map counts 
depicted in Figure~\ref{LiftedRel_fig} on page~\pageref{LiftedRel_fig}, 
with the left-hand sides representing the initial relations in 
$\R\ov\cM_{0,1,2}$ and $\R\ov\cM_{0,0,3}$
and the right-hand sides representing the wall-crossing corrections.
The first two relations of Theorem~\ref{SolWDVV_thm} are obtained 
by using the two relations of Figure~\ref{LiftedRel_fig} with the divisors~$H_1$ and~$H_2$
as the first two non-real constraints and points as the remaining constraints.
The last relation of Theorem~\ref{SolWDVV_thm} is obtained by using 
the second relation of Figure~\ref{LiftedRel_fig} with the divisors~$H_1$, $H_2$,
and~$H_3$
as the first two non-real constraints and points as the remaining constraints.

By the comparison between the curve counts of~\cite{Wel4}
and the map counts of~\cite{Jake} established in~\cite{SpinPin},
the relations of Theorem~\ref{SolWDVV_thm} are equivalent to 
the relations for the former stated in \cite[Theorem~1]{RealWDVVapp}.
They completely determine Welschinger's invariants of~$\P^2$ and its blowups,
as shown in~\cite{Jake2} and~\cite{HorevSol}, respectively.
For the ease of use, these results are summarized in~\cite{RealWDVVapp};
the low-degree numbers obtained from the three relations for Welschinger's invariants
and listed in~\cite{RealWDVVapp} agree with \cite{ABL,B15,BP13,IKS09,IKS13,IKS13a,IKS13b,Wel07}.

The relations of \cite[Theorem~1]{RealWDVVapp} for Welschinger's invariants are 
the same as implied by the statements of Theorem~8, Proposition~10, and Theorem~11
in~\cite{Jake2}. 
The relations of Theorem~\ref{SolWDVV_thm} for the map invariants in the present paper
involve
the same terms as the difference between equations~(6) and~(7)
in~\cite{Jake2} and the symmetrization of equation~(5) in~\cite{Jake2}, 
but {\it different} signs.
The comparison between the curve counts of~\cite{Wel4} and the disk counts of~\cite{Jake} 
established in~\cite{SpinPin} likewise {\it differs} by sign from the claim of \cite[Thm.~11]{Jake2}.
The two sign discrepancies, which do not appear to be due to the formulations of 
the definitions of the disk invariants in the present paper and \cite{Jake,Jake2}, 
precisely cancel out to yield the same recursions for Welschinger's invariants.

This paper presents a general approach for pulling back a relation by a morphism
\hbox{$\ff\!:\M\!\lra\!\cM$}
between two spaces which is not necessarily relatively orientable.
The lifted relation then acquires a correction which doubly covers 
a ``non-orientability" hypersurface in~$\M$.
This approach should be applicable in many other settings.
It has already been used in~\cite{RealWDVV3} to
obtain WDVV-type relations for real symplectic sixfolds.

\subsection{Main theorem}
\label{SolWDVV_subs}

Let $(X,\om,\phi)$ be a compact \sf{real symplectic manifold}, i.e.~$\om$ is a symplectic form
on~$X$ so that $\phi^*\om\!=\!-\om$.
The fixed locus~$X^{\phi}$ of the anti-symplectic involution~$\phi$ on~$X$
is then a Lagrangian submanifold of~$(X,\om)$.
We denote by $H_2(X)$ the quotient~$H_2(X)$ of $H_2(X;\Z)$ modulo torsion,
by~$\cJ_{\om}$ the space of $\om$-compatible (or -tamed) 
almost complex structures~$J$ on~$X$, 
and by \hbox{$\cJ_{\om}^{\phi}\!\subset\!\cJ_{\om}$} the subspace of almost complex structures~$J$
such that \hbox{$\phi^*J\!=\!-J$}.
Let 
$$c_1(X,\om)\equiv c_1(TX,J)\in H^2(X)$$
be the first Chern class of $TX$ with respect to some $J\!\in\!\cJ_{\om}$;
it is independent of such a choice.
For \hbox{$B\!\in\!H_2(X)$}, define
$$\ell_{\om}(B)=\blr{c_1(X,\om),B}\!-\!1\in\Z, \quad
\lr{B}_l=\begin{cases}1,&\hbox{if}~2l\!=\!\ell_{\om}(B)\!-\!1;\\
0,&\hbox{otherwise}.\end{cases}$$

For $J\!\in\!\cJ_{\om}^{\phi}$ and $B\!\in\!H_2(X)$,
a subset $C\!\subset\!X$ is a \sf{genus~0} (or \sf{rational}) 
\sf{irreducible $J$-holomorphic degree~$B$ curve}
if there exists a simple (not multiply covered) \hbox{$J$-holomorphic} map 
\BE{Cudfn_e}u\!:\P^1\lra X \qquad\hbox{s.t.}\quad  C=u(\P^1),~~u_*[\P^1]=B.\EE
Such a subset~$C$ is called a \sf{real rational irreducible $J$-holomorphic degree~$B$ curve}
if in addition \hbox{$\phi(C)\!=\!C$}. 

From now on, suppose that the (real) dimension of~$X$ is~4.
The (tangent bundle of the) fixed locus~$X^{\phi}$ then admits a $\Pin^-$-structure~$\fp$.
Let $B\!\in\!H_2(X)$ and $l\!\in\!\Z^{\ge0}$ be such~that
\BE{dimcond_e} k\equiv \ell_{\om}(B)\!-\!2l\in\Z^{\ge0}\,.\EE
For a generic $J\!\in\!\cJ_{\om}^{\phi}$, 
there are then only finitely many real rational irreducible  $J$-holomorphic 
degree~$B$ curves $C\!\subset\!X$ intersecting a connected component~$\wch{X}^{\phi}$ of~$X^{\phi}$
and passing through  $k$ points in~$\wch{X}^{\phi}$
and $l$ points in $X\!-\!X^{\phi}$ in general position. 
According to \cite[Thm.~1.3]{Jake}, the number of such curves counted with appropriate signs
determined by~$\fp$ is independent of the choices of~$J$ and the points.
We denote this signed count of genus~0 curves by $N_{B,l}^{\phi,\fp}(\wch{X}^{\phi})$.
If the number~$k$ in~\eref{dimcond_e} is negative, 
we set $N_{B,l}^{\phi,\fp}(\wch{X}^{\phi})\!=\!0$.
We denote by~$N_{B,l}^{\phi,\fp}$ the sum of the numbers $N_{B,l}^{\phi,\fp}(\wch{X}^{\phi})$
over the connected components~$\wch{X}^{\phi}$ of~$X^{\phi}$.

Suppose $B\!\in\!H_2(X)$ and $\ell_{\om}(B)\!\ge\!0$.
For a generic $J\!\in\!\cJ_{\om}$, 
there are then only finitely many rational irreducible $J$-holomorphic 
degree~$B$ curves~$C$ passing through $\ell_{\om}(B)$ points in general position. 
The number of such curves counted with appropriate signs is independent of the choices of~$J$ 
and the points. 
This is the standard (complex) genus~0 degree~$B$ \sf{Gromov-Witten invariant} of~$(X,\om)$
with $\ell_{\om}(B)$ point insertions;
we denote it by~$N_B^X$.
If $\ell_{\om}(B)\!<\!0$, we set $N_B^X\!=\!0$.

For $B,B'\!\in\!H_2(X)$, we denote by $B\!\cdot_X\!\!B'\!\in\!\Z$  
the homology intersection product of~$B$ with~$B'$
and by $B^2\!\in\!\Z$ the self-intersection number of~$B$. 
Define
$$\fd\!:H_2(X)\lra H_2(X),~~\fd(B)=B\!-\!\phi_*(B),
\qquad H^2(X)^{\phi}_-=\big\{H\!\in\!H^2(X)\!:\phi^*H\!=\!-H\big\}.$$

\begin{thm}\label{SolWDVV_thm}
Suppose $(X,\om,\phi)$ is a compact real symplectic fourfold,
$\fp$ is a $\Pin^-$-structure on~$X^{\phi}$,
$\wch{X}^{\phi}$ is a connected component of~$X^{\phi}$, and 
$H_1,H_2,H_3\!\in\!H^2(X)^{\phi}_-$ are such that $\lr{H_1H_2,X}\!=\!1$ and $H_1H_3\!=\!0$.
Let $l\!\in\!\Z^{\ge0}$ and $B\!\in\!H_2(X)$. 
\BEnum{($\R$\!WDVV\arabic*)}

\item\label{Sol12rec_it} If $l\!\ge\!1$ and $\ell_{\om}(B)\!-\!2l\!\ge\!1$ 
(i.e.~$k\!\ge\!1$), then
\begin{equation*}\begin{split}
&\hspace{-1in}
N_{B,l}^{\phi;\fp}(\wch{X}^{\phi})= -2^{l-3}\lr{B}_l\lr{H_1,B}\lr{H_2,B}
\hspace{-.15in}\sum_{\begin{subarray}{c}B'\in H_2(X)\\ \fd(B')=B\end{subarray}}
\hspace{-.18in}N_{B'}^X\\
&\hspace{-.2in}
-\!\!\!\!\sum_{\begin{subarray}{c}B_0,B'\in H_2(X)-\{0\}\\ B_0+\fd(B')=B\end{subarray}}
\hspace{-.4in} 2^{l_{\om}(B')}\!\big(B_0\!\!\cdot_X\!\!B'\big)
\lr{H_1,B'}\lr{H_2,B'}\binom{l\!-\!1}{\ell_{\om}(B')}
N_{B'}^XN_{B_0,l-1-\ell_{\om}(B')}^{\phi;\fp}(\wch{X}^{\phi})\\
&\hspace{.1in}+\!\!\!\!\!
\sum_{\begin{subarray}{c}B_1,B_2\in H_2(X)-\{0\}\\B_1+B_2=B\\  
l_1+l_2=l-1,\,l_1,l_2\ge0\end{subarray}} \hspace{-.4in}
\lr{H_1,B_1}\binom{l\!-\!1}{l_1}\!\!
\Bigg(\!\lr{H_2,B_2}\binom{\ell_{\om}(B)\!-\!2l\!-\!1}{\ell_{\om}(B_1)\!-\!2l_1\!-\!1}\\
&\hspace{1.4in}
-\lr{H_2,B_1}\binom{\ell_{\om}(B)\!-\!2l\!-\!1}{\ell_{\om}(B_1)\!-\!2l_1}\!\!\!\Bigg)\!
N_{B_1,l_1}^{\phi;\fp}(\wch{X}^{\phi})N_{B_2,l_2}^{\phi;\fp}(\wch{X}^{\phi}).
\end{split}\end{equation*}

\item\label{Sol03rec_it} If $l\!\ge\!2$, then
\begin{equation*}\begin{split}
&\hspace{-.7in}
N_{B,l}^{\phi;\fp}(\wch{X}^{\phi})=
\sum_{\begin{subarray}{c}B_0,B'\in H_2(X)-\{0\}\\ B_0+\fd(B')=B\end{subarray}}
\hspace{-.4in} 2^{l_{\om}(B')-1}\!\big(B_0\!\!\cdot_X\!\!B'\big)\lr{H_1,B'}\\
&\hspace{.3in}\times\!\Bigg(\!\!\lr{H_2,B_0}\binom{l\!-\!2}{\ell_{\om}(B')\!-\!1}
-2\lr{H_2,B'}\binom{l\!-\!2}{\ell_{\om}(B')}\!\!\Bigg)
N_{B'}^XN_{B_0,l-1-\ell_{\om}(B')}^{\phi;\fp}(\wch{X}^{\phi})\\
&\hspace{.1in}+\!\!\!\!\!
\sum_{\begin{subarray}{c}B_1,B_2\in H_2(X)-\{0\}\\ B_1+B_2=B\\ 
l_1+l_2=l-2,\,l_1,l_2\ge0\end{subarray}} \hspace{-.4in}
\lr{H_2,B_1}\binom{l\!-\!2}{l_1}\!\!
\Bigg(\!\lr{H_1,B_2}\binom{\ell_{\om}(B)\!-\!2l}{\ell_{\om}(B_1)\!-\!2l_1\!-\!1}\\
&\hspace{1.4in}
-\lr{H_1,B_1}\binom{\ell_{\om}(B)\!-\!2l}{\ell_{\om}(B_1)\!-\!2l_1}\!\!\!\Bigg)\!
N_{B_1,l_1}^{\phi;\fp}(\wch{X}^{\phi})N_{B_2,l_2+1}^{\phi;\fp}(\wch{X}^{\phi}).
\end{split}\end{equation*}

\item\label{Sol03rec2_it} If $l\!\ge\!1$, then
\begin{equation*}\begin{split}
&\hspace{-.7in}
\lr{H_3,B}N_{B,l}^{\phi;\fp}(\wch{X}^{\phi})=
\sum_{\begin{subarray}{c}B_0,B'\in H_2(X)-\{0\}\\ B_0+\fd(B')=B\end{subarray}}
\hspace{-.4in} 2^{l_{\om}(B')}\!\big(B_0\!\!\cdot_X\!\!B'\big)
\lr{H_1,B'}\binom{l\!-\!1}{l_\om(B')}\\
&\hspace{.8in}\times\!\bigg(\!\!\lr{H_2,B_0}\lr{H_3,B'}\!-\!\lr{H_3,B_0}\lr{H_2,B'}\!\!\bigg)
N_{B'}^XN_{B_0,l-1-\ell_{\om}(B')}^{\phi;\fp}(\wch{X}^{\phi})\\
&\hspace{.4in}+\!\!\!\!\!
\sum_{\begin{subarray}{c}B_1,B_2\in H_2(X)-\{0\}\\ B_1+B_2=B\\ 
l_1+l_2=l-1,\,l_1,l_2\ge0\end{subarray}} \hspace{-.4in}
\lr{H_1,B_1}\big(\lr{H_3,B_1}\lr{H_2,B_2}\!-\!\lr{H_2,B_1}\lr{H_3,B_2}\big)\\
&\hspace{1.5in}
\times\!\binom{l\!-\!1}{l_1}\binom{l_\om(B)\!-\!2l}{l_\om(B_1)\!-\!2l_1}
N_{B_1,l_1}^{\phi;\fp}(\wch{X}^{\phi})N_{B_2,l_2}^{\phi;\fp}(\wch{X}^{\phi}).
\end{split}\end{equation*}

\EEnum
\end{thm}

\vspace{.1in}

Taking the difference between the relations of Theorem~\ref{SolWDVV_thm}
for $N_{B+B_{\bu},l}^{\phi}(X^{\phi})$ with $\ell_{\om}(B_{\bu})\!>\!0$ small 
yields relations involving 
the invariants $N_{B+B_{\bu},0}^{\phi}(X^{\phi})$ without conjugate pairs 
of marked points.
In some cases, the resulting relations determine these numbers; 
see~\cite{Jake2,HorevSol}.

\begin{rmk}\label{SolWDVV_rmk}
We define the invariants $N_{B,l}^{\phi;\fp}(X^{\phi})$ via the moduli spaces
of real maps constructed in~\cite{Penka2}.
This definition of $N_{B,l}^{\phi;\fp}(X^{\phi})$ differs by a power of~2
from the definitions in~\cite{Jake,Jake2}, but agrees in absolute value
with the invariants $N_{B,l}^{\phi}(X^{\phi})$ of~\cite{Wel4}.
\end{rmk}

\begin{rmk}\label{Erwan_rmk}
Welschinger's invariants of real symplectic fourfolds $(X,\om,\phi)$
with disconnected fixed loci~$X^{\phi}$ often vanish;
see Theorem~1.3 and Remark~1.1 in~\cite{B18}.
As suggested by E.~Brugall\'e, Theorem~\ref{SolWDVV_thm} and its proof readily adapt 
to Welschinger's invariants with finer notions of the curve degree~$B$ such as those taken~in
$$\wt{H}_2(X,\wch{X}^{\phi})\equiv H_2(X,\wch{X}^{\phi};\Z)\big/
\big\{b\!+\!\phi_*(b)\!:b\!\in\!H_2(X,\wch{X}^{\phi};\Z)\big\}\,.$$
As explained in \cite[Section~2.2]{Penka2}, there is a well-defined doubling map
$$\fd_{\wch{X}^{\phi}}\!:\wt{H}_2(X,\wch{X}^{\phi})\lra H_2(X;\Z)^{\phi}_-\,.$$
If the curve degrees~$B$ are taken in $\wt{H}_2(X,\wch{X}^{\phi})$,
the sums in Theorem~\ref{SolWDVV_thm} should then be taken over $B'$ and $B_i$
in $H_2(X;\Z)\!-\!\{0\}$ and $\wt{H}_2(X,\wch{X}^{\phi})$, respectively.
The doubling map~$\fd$ should be replaced by the composition
$$H_2(X;\Z)\lra H_2(X,\wch{X}^{\phi};\Z)\lra \wt{H}_2(X,\wch{X}^{\phi}).$$
All appearances of~$B$ and~$B_i$ (but not $B'$) inside~$\lr{\cdot}$ and~$(\cdot)$ 
should become~$\fd_{\wch{X}^{\phi}}(B)$ and~$\fd_{\wch{X}^{\phi}}(B_i)$, respectively. 
The resulting formulas yield relations in particular for the modification of 
Welschinger's invariants which originates in \cite[Prop.~1]{IKS13a}.
\end{rmk}

\vspace{-.1in}

The relations of Theorem~\ref{SolWDVV_thm} correspond to 
the partial differential equations~(4.82) and~(4.76) in~\cite{Adam}
for the generating functions given by~(4.68) and~(4.72), respectively,
if $H^2(X)^{\phi}_-$ is all of~$H^2(X)$.
These two PDEs are the same as the differential equations~(3) and~(4) in~\cite{Jake2}.
The first generating function in~\cite{Adam} is essentially the same as 
the generating function~$\Phi$ in~\cite{Jake2} if $H^2(X)$ is one-dimensional
(the latter does not distinguish between curve classes~$B$ with the same~$\ell_{\om}(B)$).
However, the coefficients in the second generating function differ 
from the coefficients in the generating function~$\Om$ in~\cite{Jake2} by factors of~2 and~$\fI$,
even if $H^2(X)$ is one-dimensional.
The former is due to the scaling discrepancy in the definitions of the real invariants
indicated in Remark~\ref{SolWDVV_rmk},
while the latter reflects the sign difference between the relations of Theorem~\ref{SolWDVV_thm}
and their analogues in~\cite{Jake2} mentioned in Section~\ref{liftrel_subs}.

\subsection{Outline of the proof}
\label{outline_subs}

Theorem~\ref{SolWDVV_thm} follows from 
the two relations for nodal map counts represented by
Figure~\ref{LiftedRel_fig} and from Propositions~\ref{Cdecomp_prp} and~\ref{Rdecomp_prp}.
In order to simplify the notation  for the remainder for the paper, 
we denote $\wch{X}^{\phi}$ by $X^{\phi}$,
$N_{B,l}^{\phi,\fp}(\wch{X}^{\phi})$ by $N_{B,l}^{\phi,\fp}$,
and the moduli space of real rational degree~$B$ $J$-holomorphic maps with 
$k$~real marked points and $l$~conjugate pairs of marked points 
that send the fixed locus of the domain to~$\wch{X}^{\phi}$
by $\ov\M_{k,l}(B;J)$.
If the fixed locus of~$\phi$ is connected, there is no clash with the notation~above.

We denote by $\tau$ the standard conjugation on~$\P^1$, i.e.
$$\tau\!:\P^1\lra\P^1, \qquad \tau\big([z_0,z_1]\big)=\big[\ov{z_1},\ov{z_0}].$$
For every real rational irreducible $J$-holomorphic degree~$B$ curve
contributing to~$N_{B,l}^{\phi,\fp}$, 
there exists a $J$-holomorphic map \hbox{$u\!:\P^1\!\lra\!X$} as in~\eref{Cudfn_e}
such that \hbox{$u\!\circ\!\tau\!=\!\phi\!\circ\!u$}.
Thus, the number~$N_{B,l}^{\phi,\fp}$ is a signed cardinality 
of the subset of the moduli space $\ov\M_{k,l}(B;J)$ of
real rational degree~$B$ $J$-holomorphic maps sending the $k$~real marked points
and the first points in the $l$~conjugate pairs of marked points 
to generic points in~$X^{\phi}$ and~$X$, respectively.

The domain and target of the evaluation morphism 
\BE{evtotmap_e0}\ev\!: \ov\M_{k,l}(B;J)\lra X_{k,l}\!\equiv\!(X^{\phi})^k\!\times\!X^l\EE
may not be relatively orientable, but it becomes relatively orientable 
after removing certain codimension~1 strata from the domain
(i.e.~the pull-back of 
the first Stiefel-Whitney class~$w_1$ of the target is the $w_1$ of the domain). 
We cut $\M_{k,l}(B;J)$ along these codimension~1 strata to obtain 
a bordered manifold $\wh\M_{k,l}(B;J)$ and~give 
$$\ev\!:\wh\M_{k,l}(B;J)\lra X_{k,l}$$
a relative orientation. 
The codimension~1 strata of $\M_{k,l}(B;J)$
consist of curves with two components and one real~node.

The forgetful morphisms~\eref{ffdfn_e0} we encounter take values
in the subspaces $\ov\cM_{k',l'}^{\tau}$ of $\R\ov\cM_{0,k',l'}$ of 
real curves with non-empty fixed locus;
$\ov\cM_{k',l'}^{\tau}$ is a proper subspace of $\R\ov\cM_{0,k',l'}$ if and only
if $k'\!=\!0$.
We choose a bordered hypersurface $\Ups$ in $\ov\cM_{k',l'}^{\tau}$ 
whose boundary consists of curves with three components and a conjugate pair of nodes
and a relative orientation on the inclusion of $\Ups$ into~$\ov\cM_{k',l'}^{\tau}$. 
Let $\bfC\!\subset\!X_{k,l}$ be a generic constraint consisting of divisors and points
so that the~maps
$$\ev\!\times\!\ff_{k',l'}:\wh\M_{k,l}(B;J)\lra X_{k,l}\!\times\!\ov\cM_{k',l'}^{\tau}
\quad\hbox{and}\quad
\io_{\bfC;\Ups}\!:\bfC\!\times\!\Ups\lhra{~~~} X_{k,l}\times\!\ov\cM_{k',l'}^{\tau}$$
are transverse and
$$\dim\,\wh\M_{k,l}(B;J)+\dim\big(\bfC\!\times\!\Ups\big)=
\dim\big(X_{k,l}\!\times\!\ov\cM_{k',l'}^{\tau}\big)+1\,.$$
With the relative orientations above, the signed counts of the intersection points~of 
\BEnum{(C$\arabic*$)}

\item\label{LHScnt_it} $\ev\!\times\!\ff_{k',l'}$ with the boundary of $\io_{\bfC;\Ups}$ and
\item\label{RHScnt_it} the boundary of $\ev\!\times\!\ff_{k',l'}$ with $\io_{\bfC;\Ups}$ 

\EEnum
are well-defined and equal. 

The first count above decomposes into curve-counting invariants similarly to the complex case. 
The second count can also be decomposed, based on the following observations:
\BEnum{$\bu$}

\item most boundary strata of~$\wh\M_{k,l}(B;J)$ get contracted by 
$\ev\!\times\!\ff_{k',l'}$ and thus do not contribute to~\ref{RHScnt_it};

\item some boundary strata that do not get contracted do not intersect~$\Ups$ via $\ff_{k',l'}$ 
due to our choice of $\Ups\!\subset\!\ov\cM_{k',l'}^{\tau}$
and thus do not contribute to~\ref{RHScnt_it} either;

\item intersecting the remaining boundary strata with~$\Ups$ via $\ff_{k',l'}$ 
has the effect of specifying the position of the node (relative to the marked points)
on the component of the curve carrying the first conjugate pair of marked points.

\EEnum
These statements are explained in the proof of Corollary~\ref{mainsetup_crl} at 
the end of Section~\ref{SolWDVVpf_subs} and 
in the proof of Proposition~\ref{Rdecomp_prp} in Section~\ref{Rdecomp_subs}.
The equality of the counts~\ref{LHScnt_it} and~\ref{RHScnt_it} then
translates into~\ref{Sol12rec_it} in the case $(k',l')\!=\!(1,2)$ 
and into~\ref{Sol03rec_it} and\ref{Sol03rec2_it} in the case \hbox{$(k',l')\!=\!(0,3)$}. 

The paper is organized as follows. 
Section~\ref{summary_sec} is a detailed version of the above outline of the proof 
of Theorem~\ref{SolWDVV_thm}. 
The notions of relative orientations, 
pseudocycles with relative orientations (called \sf{Steenrod pseudocycles}), and 
intersection signs between them are defined in Section~\ref{TopolPrelim_sec};
this section also contains all relevant observations concerning these notions.
Section~\ref{DM_sec} describes in detail the hypersurfaces~$\Ups$ 
in the Deligne-Mumford spaces $\ov\cM_{1,2}^{\tau}$ and $\ov\cM_{0,3}^{\tau}$
used in the proof of Theorem~\ref{SolWDVV_thm}.
Section~\ref{RealGWs_sec} sets up the notation relevant to the map spaces $\M_{k,l}(B;J)$, 
states the propositions that are among the main steps in the proof of Theorem~\ref{SolWDVV_thm},
and deduces this theorem from them and the lemmas of Section~\ref{cMbordism_subs}.
The (somewhat technical) proofs of these propositions are deferred to Section~\ref{proofs_sec}. 

\textit{Acknowledgments.} 
I would like to thank my thesis advisor Aleksey Zinger for introducing me to
this subject and background material, suggesting the topic, a lot of guidance and discussions throughout the process of the work, very detailed help with exposition, and constant encouragements.

\section{Summary of the proof of Theorem~\ref{SolWDVV_thm}}
\label{summary_sec}

The numbers $N_{B,l}^{\phi,\fp}\!\equiv\!N_{B,l}^{\phi,\fp}(\wch{X}^{\phi})$
appearing in Theorem~\ref{SolWDVV_thm} 
arise from the moduli space $\ov\M_{k,l}(B;J)$ of genus~0 real degree~$B$ $J$-holomorphic
maps to~$X$ that take the fixed locus of the domain to 
the chosen topological component~$X^{\phi}$ of the fixed locus of~$\phi$.
This moduli space has no boundary if 
\BE{BKcond_e} k\!+\!2\Z\neq \blr{w_2(X),B}\in\Z_2\,.\EE
By~\cite{Jake}, a $\Pin^-$-structure~$\fp$ on~$X^{\phi}$ can be used to specify
a \sf{relative orientation} of the restriction of the total evaluation morphism~\eref{evtotmap_e0}
to the main stratum $\M_{k,l}(B;J)$ of $\ov\M_{k,l}(B;J)$ if \eref{BKcond_e} holds.
Since $\M_{k,l}(B;J)$ is generally disconnected, there are a number of systematic ways
of doing so, some of which we index by $l^*\!\in\!\Z^{\ge0}$ with $l^*\!\le\!l$.
By~\cite{Jake} again, these orientations extend across some codimension~1 boundary
strata~$\cS$, but not others.
In our setup, the ``$l^*$-orientation" $\fo_{\fp;l^*}$ on the restriction of~\eref{evtotmap_e0} 
to $\M_{k,l}(B;J)$
extends over a such stratum~$\cS$ if and only if a certain $\Z$-valued invariant 
$\ep_{l^*}(\cS)$ of~$\cS$ is congruent to~0 or~1 mod~4; see Lemma~\ref{orient_lmm}.

If $k,l\!\in\!\Z^{\ge0}$ and $B\!\in\!H_2(X)$ satisfy~\eref{dimcond_e}, 
the path in $\ov\M_{k,l}(B;J)$ determined by a generic path of 
collections of $k$~points in~$X^{\phi}$ and $l$~points in~$X\!-\!X^{\phi}$ and
of almost complex structures $J_t\!\in\!\cJ_{\om}^{\phi}$ does not cross 
the codimension~1 boundary strata~$\cS$ with $\ep_0(\cS)$ congruent to~2 or~3 mod~4.
This fundamental insight, formulated in terms of moduli spaces of disk maps in~\cite{Jake},
along with the above orientation statements established
the invariance of the counts $N_{B,l}^{\phi,\fp}$ and has since been
used to construct numerical invariants in some other settings.

The image of each codimension~1 stratum~$\cS$ with $\ep_0(\cS)$ congruent to~2 or~3 mod~4
under~\eref{evtotmap_e0} is of smaller dimension than~$\cS$.
Along with the orientation statements above, this implies that the restriction of~\eref{evtotmap_e0}
to the complement $\ov\M_{k,l;0}^{\st}(B;J)$ of the closures~$\ov\cS$ of these strata is
a codimension~0 \sf{Steenrod pseudocycle} with respect to the orientation~$\fo_{\fp;0}$;
see Proposition~\ref{JakePseudo_prp}.
The number $N_{B,l}^{\phi,\fp}$ is the \sf{degree} $\deg(\ev,\fo_{\fp;0})$
of this pseudocycle.

The orientations on the restriction of~\eref{evtotmap_e0} to $\M_{k,l}(B;J)$
relevant to lifting relations from $\ov\cM_{1,2}^{\tau}$ and $\ov\cM_{0,3}^{\tau}$ 
to $\ov\M_{k,l}(B;J)$ are the orientations~$\fo_{\fp;l^*}$ of Lemma~\ref{orient_lmm}
with $l^*\!=\!2,3$, 
as we would like to apply the lifted relations with two and three divisor insertions.
The relevant restriction of~\eref{evtotmap_e0} shrinks the codimension~1 strata~$\cS$
with $\ep_{l^*}(\cS)$ congruent to~2 or~3 mod~4, but {\it not} with $\ep_{l^*}(\cS)\!=\!2$.
In order to deal with this issue, we cut $\ov\M_{k,l}(B;J)$ along 
the closures~$\ov\cS$ of the strata~$\cS$ with $\ep_{l^*}(\cS)$ congruent to~2 or~3 mod~4.
We obtain a moduli space $\wh\M_{k,l;l^*}(B;J)$ with boundary consisting of 
double covers~$\wh\cS$ of these strata.
The relative orientation~$\fo_{\fp;l^*}$ extends to 
a relative orientation~$\wh\fo_{\fp;l^*}$ of the total evaluation morphism
\BE{evtotmap_e0b}\ev\!:\wh\M_{k,l;l^*}(B;J)\lra X_{k,l}\EE
induced by~\eref{evtotmap_e0}.

Suppose $k,l\!\in\!\Z^{\ge0}$ and $B\!\in\!H_2(X)$ are as in~\eref{dimcond_e},
$k'\!\le\!k$, and $l^*\!\le\!l'\!\le\!l\!+\!l^*\!-\!1$ so that
there are well-defined forgetful morphisms
\BE{Rffdfn_e}\ff_{k',l'}\!:\ov\M_{k,l+l^*-1;l^*}(B;J)\lra\ov\cM_{k',l'}^{\tau}
\quad\hbox{and}\quad
\ff_{k',l'}\!:\wh\M_{k,l+l^*-1;l^*}(B;J)\lra\ov\cM_{k',l'}^{\tau}\,.\EE
An $l^*$-tuple $\bh\!\equiv\!(H_1,\ldots,H_{l^*})$ of divisors in~$X$ cuts 
out the subspace
$$\wh\cZ_{k,l+l^*-1;\bh}^{\st}(B;J)\subset \wh\M_{k,l+l^*-1;l^*}(B;J)\!\times\!
H_1\!\times\!\ldots\!\times\!H_{l^*}$$
of maps with the first $l^*$ non-real marked points lying on~$H_1,\ldots,H_{l^*}$.
The relative orientation~$\wh\fo_{\fp;l^*}$ of~\eref{evtotmap_e0b} and 
the orientation~$\fo_{\bh}$ on $H_1\!\times\!\ldots\!\times\!H_{l^*}$ 
induce a relative orientation~$\wh\fo_{\fp;\bh}$
of the evaluation morphism
$$\ev_{\bh}\!:\wh\cZ_{k,l+l^*-1;\bh}^{\st}(B;J)\lra X_{k,l-1}$$
at the remaining marked points.
A tuple~$\bp$ of points in~$X_{k,l-1}$ and a bordered compact real hypersurface 
$\Ups\!\subset\!\ov\cM_{k',l'}^{\tau}$ determine an embedding
$$f_{\bp;\Ups}\!:\Ups\lra X_{k,l-1}\!\times\!\ov\cM_{k',l'}^{\tau}\,.$$
Under appropriate regularity assumptions, the fiber product 
$M_{(\ev_{\bh},\ff_{k',l'}),f_{\bp;\Ups}}$ of 
$$\big(\ev_{\bh},\ff_{k',l'}\big)\!:\wh\cZ_{k,l+l^*-1;\bh}^{\st}(B;J)
\lra X_{k,l-1}\!\times\!\ov\cM_{k',l'}^{\tau}$$
with $f_{\bp;\Ups}$ is a compact one-dimensional manifold with the boundary
\BE{bndsplit_e0}\begin{split}
\prt M_{(\ev_{\bh},\ff_{k',l'}),f_{\bp;\Ups}}
&=\wh\cZ_{k,l+l^*-1;\bh}^{\st}(B;J)\,_{(\ev_{\bh},\ff_{k',l'})}\!\times_{f_{\bp;\Ups}}\!\prt \Ups\\
&\hspace{.5in}
\sqcup \big(\prt\wh\cZ_{k,l+l^*-1;\bh}^{\st}(B;J)\!\big)\,_{(\ev_{\bh},\ff_{k',l'})}\!\times_{f_{\bp;\Ups}}\!\Ups\,.
\end{split}\EE
The relative orientation~$\wh\fo_{\fp;\bh}$ and a co-orientation~$\fo_\Ups^c$ on~$\Ups$ determine 
signs of the points on the right-hand side of~\eref{bndsplit_e0} so~that 
\BE{bndsplit_e}\begin{split}
&\big|\wh\cZ_{k,l+l^*-1;\bh}^{\st}(B;J)\,_{(\ev_{\bh},\ff_{k',l'})}\!\!\times_{f_{\bp;\Ups}}\!
\prt \Ups\big|_{\wh\fo_{\fp;\bh},\prt\fo_\Ups^c}^{\pm}\\
&\hspace{1in}
=(-1)^{\dim\,\Ups}
\big|\big(\prt\wh\cZ_{k,l+l^*-1;\bh}^{\st}(B;J)\!\big)\,_{(\ev_{\bh},\ff_{k',l'})}
\!\!\times_{f_{\bp;\Ups}}\!\Ups\big|_{\prt\fo_{\fp;\bh},\fo_\Ups^c}^{\pm}\,,
\end{split}\EE
where $|\cdot|^{\pm}$ denotes the signed cardinality; see Lemma~\ref{InterOrient_lmm}.

Since only the strata~$\wh\cS$ of $\prt\wh\M_{k,l+l^*-1;l^*}(B;J)$
with $\ep_{l^*}(\wh\cS)\!=\!2$ are not shrunk by~\eref{evtotmap_e0b},
only the~strata
$$\wh\cS_{\bh}\equiv \big(\wh\cS\!\times\!H_1\!\times\!\ldots\!\times\!H_{l^*}\big)\cap
\wh\cZ_{k,l+1;\bh}^{\st}(B;J)$$
of $\wh\cZ_{k,l+l^*-1;\bh}^{\st}(B;J)$ with $\ep_{l^*}(\wh\cS)\!=\!2$
contribute to the right-hand side of~\eref{bndsplit_e}.
Since $\wh\cS_{\bh}$ is a double cover~of the subspace
$$\cS_{\bh}\subset \cS\!\times\!H_1\!\times\!\ldots\!\times\!H_{l^*}$$
of maps with the first $l^*$ non-real marked points lying on $H_1,\ldots,H_{l^*}$, 
we conclude~that
\begin{equation*}\begin{split}
&\big|\big(\prt\wh\cZ_{k,l+l^*-1;\bh}^{\st}(B;J)\!\big)\,_{(\ev_{\bh},\ff_{k',l'})}
\!\!\times_{f_{\bp;\Ups}}\!\Ups\big|_{\prt\wh\fo_{\fp;\bh},\fo_\Ups^c}^{\pm}
=2\!\!\!\!\!\sum_{\ep_{l^*}(\cS)=2}\!\!\!\!\!\!\big|(\cS_{\bh})\,_{(\ev_{\bh},\ff_{k',l'})}
\!\!\times_{f_{\bp;\Ups}}\!\Ups\big|_{\prt\wh\fo_{\fp;\bh},\fo_\Ups^c}^{\pm}\,.
\end{split}\end{equation*}

The moduli space $\ov\cM_{k',l'}^{\tau}$ contains codimension~1 strata~$\oS_i$
with $i\!\le\!l'$ parametrizing marked curves with two real components so that 
one of the components carries only the $i$-th conjugate pair of marked points.
We establish Theorem~\ref{SolWDVV_thm} by applying~\eref{bndsplit_e} 
with certain bordered compact hypersurfaces~$\Ups$ in \hbox{$\ov\cM_{1,2}^{\tau}\!\approx\!\R\P^2$}
and in the three-dimensional orientable manifold $\ov\cM_{0,3}^{\tau}$
so that~$\Ups$ is {\it disjoint} from the closure~$\ov{S}_1$ of~$\oS_1$. 

The moduli space $\ov\cM_{1,2}^{\tau}$  contains two points~$P^{\pm}$ 
corresponding to the two marked curves
consisting of one real component and one conjugate pair of components;
see the diagrams on the left-hand side of the first row in Figure~\ref{LiftedRel_fig}.
We take~$\Ups$ in $\ov\cM_{1,2}^{\tau}\!-\!\ov{S}_1$ to be a path from~$P^-$ to~$P^+$
as in Lemma~\ref{M12rel_lmm}.
In this case, \eref{bndsplit_e} is represented by 
the first row in Figure~\ref{LiftedRel_fig}.
The labels $\ep_{l^*}(\cS)\!=\!2$ and $\cap\!\Ups$
under the diagrams on the right-hand side indicate that
only ``intersections"
of some strata of two-component maps with~$\Ups$ contribute to this relation. 
These intersections arise from the last part of the boundary in~\eref{bndsplit_e0} and thus 
contribute twice each (with the same sign).
The strata of two-component maps whose contributions are described as being {\it insignificant due to 
sign cancellations} in \cite[p10]{Jake2} do not appear in our approach at~all.

\begin{figure}\begin{center}
\begin{tikzpicture}
\draw (0,0) circle [radius=0.5];
\draw (-0.5,0) arc [start angle=180, end angle=360, x radius=0.5, y radius=0.2];
\draw [dashed] (-0.5,0) arc [start angle=180, end angle=0, x radius=0.5, y radius=0.1];
\draw [fill] (-0.5,0) circle [radius=0.04];
\node [left] at (-0.5,0) {$x_1$};
\draw (0,1) circle [radius=0.5]; 
\draw (0,-1) circle [radius=0.5];
\draw [fill] (-0.2,1.25) circle [radius=0.02];
\node [above] at (-0.3,1.3) {$z_1^+$};
\draw [fill] (-0.2,-1.25) circle [radius=0.02];
\node [below] at (-0.3,-1.3) {$z_1^-$};
\draw [fill] (0.25,1.25) circle [radius=0.02];
\node [above] at (0.33,1.3) {$z_2^+$};
\draw [fill] (0.25,-1.25) circle [radius=0.02];
\node [below] at (0.33,-1.3) {$z_2^-$};
\node at (1,0) {$+$};
\end{tikzpicture}\hspace{-0.15cm}
\begin{tikzpicture}
\draw (0,0) circle [radius=0.5];
\draw (-0.5,0) arc [start angle=180, end angle=360, x radius=0.5, y radius=0.2];
\draw [dashed] (-0.5,0) arc [start angle=180, end angle=0, x radius=0.5, y radius=0.1];
\draw [fill] (-0.5,0) circle [radius=0.04];
\node [left] at (-0.5,0) {$x_1$};
\draw (0,1) circle [radius=0.5]; 
\draw (0,-1) circle [radius=0.5];
\draw [fill] (-0.2,1.25) circle [radius=0.02];
\node [above] at (-0.3,1.3) {$z_1^+$};
\draw [fill] (-0.2,-1.25) circle [radius=0.02];
\node [below] at (-0.3,-1.3) {$z_1^-$};
\draw [fill] (0.25,1.25) circle [radius=0.02];
\node [above] at (0.33,1.3) {$z_2^-$};
\draw [fill] (0.25,-1.25) circle [radius=0.02];
\node [below] at (0.33,-1.3) {$z_2^+$};
\node at (1.25,0) {=$~-2$};
\end{tikzpicture}\hspace{-0.3cm}
\begin{tikzpicture}
\draw (-0.5,0) circle [radius=0.5];
\draw (0.5,0) circle [radius=0.5];
\draw (0,0) arc [start angle=180, end angle=360, x radius=0.5, y radius=0.2];
\draw [dashed] (0,0) arc [start angle=180, end angle=0, x radius=0.5, y radius=0.1];
\draw (-1,0) arc [start angle=180, end angle=360, x radius=0.5, y radius=0.2];
\draw [dashed] (-1,0) arc [start angle=180, end angle=0, x radius=0.5, y radius=0.1];
\draw [fill] (-0.75,0.3) circle [radius=0.02];
\node [above left] at (-0.7,0.3) {$z_1^+$}; 
\draw [fill] (-0.75,-0.3) circle [radius=0.02];
\node [below left] at (-0.7,-0.3) {$z_1^-$}; 
\draw [fill] (-0.3,0.3) circle [radius=0.02];
\node [above] at (-0.3,0.4) {$z_2^\pm$}; 
\draw [fill] (-0.3,-0.3) circle [radius=0.02];
\node [below] at (-0.3,-0.4) {$z_2^\mp$};
\draw [fill] (1,0) circle [radius=0.04];
\node [right] at (1,0) {$x_1$};
\node [below] at (0.33,-1.7) { };
\node at (2,0) {$-$ 2};
\node at (-.2,-1.4){\sm{$\ep_{l^*}(\cS)\!=\!2,~\cap\!\Ups$}};
\end{tikzpicture}\hspace{-0.15cm}
\begin{tikzpicture}
\draw (-0.5,0) circle [radius=0.5];
\draw (0.5,0) circle [radius=0.5];
\draw (0,0) arc [start angle=180, end angle=360, x radius=0.5, y radius=0.2];
\draw [dashed] (0,0) arc [start angle=180, end angle=0, x radius=0.5, y radius=0.1];
\draw (-1,0) arc [start angle=180, end angle=360, x radius=0.5, y radius=0.2];
\draw [dashed] (-1,0) arc [start angle=180, end angle=0, x radius=0.5, y radius=0.1];
\draw [fill] (-1,0) circle [radius=0.04];
\node [left] at (-0.9,0) {$x_1$}; 
\draw [fill] (-0.3,0.3) circle [radius=0.02];
\node [above] at (-0.3,0.4) {$z_1^+$}; 
\draw [fill] (-0.3,-0.3) circle [radius=0.02];
\node [below] at (-0.3,-0.4) {$z_1^-$};
\draw [fill] (0.65,0.28) circle [radius=0.02];
\node [above right] at (0.6,0.28) {$z_2^+$};
\draw [fill] (0.65,-0.28) circle [radius=0.02];
\node [below right] at (0.6,-0.28) {$z_2^-$};
\node [below] at (0.33,-1.7) { };
\node at (-.2,-1.4){\sm{$\ep_{l^*}(\cS)\!=\!2,~\cap\!\Ups$}};
\end{tikzpicture}

\begin{tikzpicture}
\draw (0,0) circle [radius=0.5];
\draw (-0.5,0) arc [start angle=180, end angle=360, x radius=0.5, y radius=0.2];
\draw [dashed] (-0.5,0) arc [start angle=180, end angle=0, x radius=0.5, y radius=0.1];
\draw [fill] (-0.2,0.3) circle [radius=0.02];
\node [left] at (-0.25,0.3) {$z_2^\pm$};
\draw [fill] (-0.2,-0.3) circle [radius=0.02];
\node [left] at (-0.25,-0.4) {$z_2^\mp$};
\draw (0,1) circle [radius=0.5]; 
\draw (0,-1) circle [radius=0.5];
\draw [fill] (-0.2,1.25) circle [radius=0.02];
\node [above] at (-0.3,1.3) {$z_1^+$};
\draw [fill] (-0.2,-1.25) circle [radius=0.02];
\node [below] at (-0.3,-1.3) {$z_1^-$};
\draw [fill] (0.25,1.25) circle [radius=0.02];
\node [above] at (0.33,1.3) {$z_3^+$};
\draw [fill] (0.25,-1.25) circle [radius=0.02];
\node [below] at (0.33,-1.3) {$z_3^-$};
\node at (1,0) {+};
\end{tikzpicture}\hspace{-0.3cm}
\begin{tikzpicture}
\draw (0,0) circle [radius=0.5];
\draw (-0.5,0) arc [start angle=180, end angle=360, x radius=0.5, y radius=0.2];
\draw [dashed] (-0.5,0) arc [start angle=180, end angle=0, x radius=0.5, y radius=0.1];
\draw [fill] (-0.2,0.3) circle [radius=0.02];
\node [left] at (-0.25,0.3) {$z_2^\pm$};
\draw [fill] (-0.2,-0.3) circle [radius=0.02];
\node [left] at (-0.25,-0.4) {$z_2^\mp$};
\draw (0,1) circle [radius=0.5]; 
\draw (0,-1) circle [radius=0.5];
\draw [fill] (-0.2,1.25) circle [radius=0.02];
\node [above] at (-0.3,1.3) {$z_1^+$};
\draw [fill] (-0.2,-1.25) circle [radius=0.02];
\node [below] at (-0.3,-1.3) {$z_1^-$};
\draw [fill] (0.25,1.25) circle [radius=0.02];
\node [above] at (0.33,1.3) {$z_3^-$};
\draw [fill] (0.25,-1.25) circle [radius=0.02];
\node [below] at (0.33,-1.3) {$z_3^+$};
\node at (1,0) {$-$};
\end{tikzpicture}\hspace{-0.3cm}
\begin{tikzpicture}
\draw (0,0) circle [radius=0.5];
\draw (-0.5,0) arc [start angle=180, end angle=360, x radius=0.5, y radius=0.2];
\draw [dashed] (-0.5,0) arc [start angle=180, end angle=0, x radius=0.5, y radius=0.1];
\draw [fill] (-0.2,0.3) circle [radius=0.02];
\node [left] at (-0.25,0.3) {$z_3^\pm$};
\draw [fill] (-0.2,-0.3) circle [radius=0.02];
\node [left] at (-0.25,-0.4) {$z_3^\mp$};
\draw (0,1) circle [radius=0.5]; 
\draw (0,-1) circle [radius=0.5];
\draw [fill] (-0.2,1.25) circle [radius=0.02];
\node [above] at (-0.3,1.3) {$z_1^+$};
\draw [fill] (-0.2,-1.25) circle [radius=0.02];
\node [below] at (-0.3,-1.3) {$z_1^-$};
\draw [fill] (0.25,1.25) circle [radius=0.02];
\node [above] at (0.33,1.3) {$z_2^+$};
\draw [fill] (0.25,-1.25) circle [radius=0.02];
\node [below] at (0.33,-1.3) {$z_2^-$};
\node at (1,0) {$-$};
\end{tikzpicture}\hspace{-0.3cm}
\begin{tikzpicture}
\draw (0,0) circle [radius=0.5];
\draw (-0.5,0) arc [start angle=180, end angle=360, x radius=0.5, y radius=0.2];
\draw [dashed] (-0.5,0) arc [start angle=180, end angle=0, x radius=0.5, y radius=0.1];
\draw [fill] (-0.2,0.3) circle [radius=0.02];
\node [left] at (-0.25,0.3) {$z_3^\pm$};
\draw [fill] (-0.2,-0.3) circle [radius=0.02];
\node [left] at (-0.25,-0.4) {$z_3^\mp$};
\draw (0,1) circle [radius=0.5]; 
\draw (0,-1) circle [radius=0.5];
\draw [fill] (-0.2,1.25) circle [radius=0.02];
\node [above] at (-0.3,1.3) {$z_1^+$};
\draw [fill] (-0.2,-1.25) circle [radius=0.02];
\node [below] at (-0.3,-1.3) {$z_1^-$};
\draw [fill] (0.25,1.25) circle [radius=0.02];
\node [above] at (0.33,1.3) {$z_2^-$};
\draw [fill] (0.25,-1.25) circle [radius=0.02];
\node [below] at (0.33,-1.3) {$z_2^+$};
\node at (1.2,0) {=  2};
\end{tikzpicture}\hspace{-0.3cm}
\begin{tikzpicture}
\draw (-0.5,0) circle [radius=0.5];
\draw (0.5,0) circle [radius=0.5];
\draw (0,0) arc [start angle=180, end angle=360, x radius=0.5, y radius=0.2];
\draw [dashed] (0,0) arc [start angle=180, end angle=0, x radius=0.5, y radius=0.1];
\draw (-1,0) arc [start angle=180, end angle=360, x radius=0.5, y radius=0.2];
\draw [dashed] (-1,0) arc [start angle=180, end angle=0, x radius=0.5, y radius=0.1];
\draw [fill] (-0.75,0.3) circle [radius=0.02];
\node [above left] at (-0.7,0.3) {$z_1^+$}; 
\draw [fill] (-0.75,-0.3) circle [radius=0.02];
\node [below left] at (-0.7,-0.3) {$z_1^-$}; 
\draw [fill] (-0.3,0.3) circle [radius=0.02];
\node [above] at (-0.3,0.4) {$z_3^\pm$}; 
\draw [fill] (-0.3,-0.3) circle [radius=0.02];
\node [below] at (-0.3,-0.4) {$z_3^\mp$};
\draw [fill] (0.65,0.28) circle [radius=0.02];
\node [above right] at (0.6,0.28) {$z_2^+$};
\draw [fill] (0.65,-0.28) circle [radius=0.02];
\node [below right] at (0.6,-0.28) {$z_2^-$};
\node [below] at (0.33,-1.7) { };
\node at (1.5,0) {$+$ 2};
\node at (-.2,-1.4){\sm{$\ep_{l^*}(\cS)\!=\!2,~\cap\!\Ups$}};
\end{tikzpicture}\hspace{-0.3cm}
\begin{tikzpicture}
\draw (-0.5,0) circle [radius=0.5];
\draw (0.5,0) circle [radius=0.5];
\draw (0,0) arc [start angle=180, end angle=360, x radius=0.5, y radius=0.2];
\draw [dashed] (0,0) arc [start angle=180, end angle=0, x radius=0.5, y radius=0.1];
\draw (-1,0) arc [start angle=180, end angle=360, x radius=0.5, y radius=0.2];
\draw [dashed] (-1,0) arc [start angle=180, end angle=0, x radius=0.5, y radius=0.1];
\draw [fill] (-0.75,0.3) circle [radius=0.02];
\node [above left] at (-0.7,0.3) {$z_1^+$}; 
\draw [fill] (-0.75,-0.3) circle [radius=0.02];
\node [below left] at (-0.7,-0.3) {$z_1^-$}; 
\draw [fill] (-0.3,0.3) circle [radius=0.02];
\node [above] at (-0.3,0.4) {$z_2^\pm$}; 
\draw [fill] (-0.3,-0.3) circle [radius=0.02];
\node [below] at (-0.3,-0.4) {$z_2^\mp$};
\draw [fill] (0.65,0.28) circle [radius=0.02];
\node [above right] at (0.6,0.28) {$z_3^+$};
\draw [fill] (0.65,-0.28) circle [radius=0.02];
\node [below right] at (0.6,-0.28) {$z_3^-$};
\node [below] at (0.33,-1.7) { };
\node at (-.2,-1.4){\sm{$\ep_{l^*}(\cS)\!=\!2,~\cap\!\Ups$}};
\end{tikzpicture}
\end{center}
\caption{The relations on stable maps induced via~\eref{bndsplit_e} by lifting 
codimension~2 relations from $\ov\cM_{1,2}^{\tau}$ and~$\ov\cM_{0,3}^{\tau}$;
the curves on the right-hand sides of the two relations are constrained
by the hypersurfaces $\Ups$ in $\ov\cM_{1,2}^{\tau}$ and~$\ov\cM_{0,3}^{\tau}$.}
\label{LiftedRel_fig}\end{figure}

The one-dimensional strata of $\ov\cM_{0,3}^{\tau}$ that parametrize marked curves
consisting of one real component and one conjugate pair of components come
in three pairs $\oGa_i^{\pm}$ with $i\!=\!1,2,3$;
see the diagrams on the left-hand side of the second row in Figure~\ref{LiftedRel_fig}.
The closures $\ov\Ga_i^{\pm}$ of these strata with $i\!=\!2,3$
bound a compact oriented surface~$\Ups$ in $\ov\cM_{0,3}^{\tau}\!-\!\ov{S}_1$
as in Lemma~\ref{M03rel_lmm}.
In this case, \eref{bndsplit_e}  is represented by 
the second row in Figure~\ref{LiftedRel_fig}.
The curves represented by the diagrams on the right-hand side
in this relation again arise from the last part of the boundary in~\eref{bndsplit_e0}.

We apply the relations represented by Figure~\ref{LiftedRel_fig} 
with the divisors~$H_1,H_2$ as the first two non-real insertions
and points as the remaining insertions;
we also apply the second relation with the divisors $H_1,H_2,H_3$
as the first three non-real insertions.
The normal bundle to the strata of maps represented by the three-component curves
in this figure is canonically oriented.
Thus, the restriction of the total evaluation map~\eref{evtotmap_e0}
to these strata inherits a relative orientation from its restriction to $\M_{k,l}(B;J)$.
The proof of \cite[Prop.~4.2]{RealEnum} readily applies to express the associated
counts of nodal maps in terms of the real map counts~$N_{B,l}^{\phi,\fp}$
and the complex map counts~$N_B^X$; see Proposition~\ref{Cdecomp_prp}.

The map counts represented by the  two-component curves in Figure~\ref{LiftedRel_fig}
are more elaborate.
Each stratum~$\cS_{\bh}$ of such maps is the fiber product of the evaluation morphisms
$$\ev_{\nod}\!:\cZ_1\!\equiv\!\cZ_{k_1+1,l_1;\bh_1}(B_1;J)\lra X^{\phi}
\quad\hbox{and}\quad
\ev_{\nod}\!:\cZ_2\!\equiv\!\cZ_{k_2+1,l_2;\bh_2}(B_2;J)\lra X^{\phi}$$
 at the nodal points from moduli spaces associated with the two components,
for a split of~$\bh$ into an $l_1^*$-tuple~$\bh_1$ and an $l_2^*$-tuple~$\bh_2$.
The condition $\ep_{l^*}(\cS)\!=\!2$ implies~that each of the total evaluation morphisms
\begin{gather}\label{evbh1pr_e}
\ev_{\bh_1}'\!:\cZ_1'\!\equiv\!\cZ_{k_1,l_1;\bh_1}(B_1;J)\lra X_{k_1,l_1-l_1^*}
\qquad\hbox{and}\\
\notag
\ev_{\bh_2}\!\!\equiv\!\big(\ev_{\bh_2}',\ev_{\nod}\big)\!:
\cZ_2\lra X_{k_2+1,l_2-l_2^*}\!\equiv\!X_{k_2,l_2-l_2^*}\!\times\!X^{\phi}
\end{gather}
is a map between spaces of the same dimensions.
The latter implies that~\eref{BKcond_e} with $(k,B)$ replaced by either $(k_1,B_1)$
and $(k_2\!+\!1,B_2)$ holds.
Thus, the~maps $\ev_{\bh_1}'$ and $\ev_{\bh_2}$ have well-defined degrees
$\deg(\ev_{\bh_1}',\fo_{\fp;\bh_1})$ and $\deg(\ev_{\bh_2},\fo_{\fp;\bh_2})$
with respect to the relative orientations induced by 
the $\Pin^-$-structure~$\fp$ and the orientations of $H_1,\ldots,H_{l^*}$.
These degrees are related to the map counts $N^{\phi,\fp}_{B_1,l_1-l_1^*}$
and $N^{\phi,\fp}_{B_2,l_2-l_2^*}$ (with $k_1$ and $k_2\!+\!1$ real point insertions, 
respectively) via the divisor relation~\eref{RdivRel_e}.

A crucial consequence of our choices of the hypersurfaces $\Ups\!\subset\!\ov\cM_{0,3}^{\tau}$
is that the restriction of the first morphism in~\eref{Rffdfn_e} 
to~$\cS_{\bh}$ factors through a morphism
$$\ff_1\!: \cZ_1\lra\ov\cM_{k',l'}^{\tau}$$
if $\ep_{l^*}(\cS)\!=\!2$ and $\cS\!\cap\!\ff_{k',l'}^{-1}(\Ups)\!\neq\!\eset$;
see Corollary~\ref{mainsetup_crl}.
Thus,
\BE{dropfactor3_e}(\cS_{\bh})\,_{(\ev_{\bh},\ff_{k',l'})}\!\!\times_{f_{\bp;\Ups}}\!\Ups
= \big((\cZ_1)\,_{(\ev_{\bh_1}',\ff_1)}\!\!\times_{f_{\bp_1;\Ups}}\!\Ups\big)\,
_{\ev_{\nod}}\!\times_{\ev_{\nod}}\!\ev_{\bh_2}^{-1}(\bp_2),\EE
for a split of~$\bp$ into a $k_1$-tuple~$\bp_1$ and a $k_2$-tuple~$\bp_2$.
The equality above holds set-theoretically; 
Lemma~\ref{dropfactor_lmm3} compares the signs on the two sides.
The morphism~$\ev_{\bh_1}'$ on the right-hand side of this equality denotes the composition
of~\eref{evbh1pr_e} with the natural projection
$$\ff\!:\cZ_{\bh_1}\lra\cZ_{\bh_1}'$$
dropping the real marked point corresponding to the node.
Thus,
$$(\cZ_1)\,_{(\ev_{\bh_1}',\ff_1)}\!\!\times_{f_{\bp_1;\Ups}}\!\Ups 
=\big\{u_1\!\in\!\cZ_1|_{\ev_{\bh_1}'^{-1}(\bp_1)}\!:\ff_1(u_1)\!\in\!\Ups\big\};$$
Lemma~\ref{dropfactor_lmm} compares the signs on the two sides.
Since this set is finite, \eref{dropfactor3_e} implies~that 
$$\big|(\cS_{\bh})\,_{(\ev_{\bh},\ff_{k',l'})}
\!\!\times_{f_{\bp;\Ups}}\!\Ups\big|_{\prt\wh\fo_{\fp;\bh},\fo_\Ups^c}^{\pm}
=\al(\cS,\Ups)\deg(\ev_{\bh_1}',\fo_{\fp;\bh_1})\deg(\ev_{\bh_2},\fo_{\fp;\bh_2})$$
for some $\al(\cS,\Ups)\!\in\!\Z$ determined by~$\cS$ and~$\Ups$.
This leads to a decomposition of the nodal map counts associated with 
the two-component diagrams in Figure~\ref{LiftedRel_fig} into sums of pairwise products
of the real map counts~$N_{B,l}^{\phi,\fp}$; see Proposition~\ref{Rdecomp_prp}.

\section{Topological preliminaries}
\label{TopolPrelim_sec}

\subsection{Relative orientations}
\label{SignConv_subs}

For a real vector space or vector bundle~$V$, let
$\la(V)\!\equiv\!\La_{\R}^{\top}V$ be its top exterior power.
For a manifold $M$, possibly with nonempty boundary~$\prt M$, we denote by 
$$\la(M)\equiv\la(TM)\equiv\La^{\top}_{\R}TM\lra M$$
its \sf{orientation line bundle}.
An \sf{orientation} of~$M$ is a homotopy class of trivializations of~$\la(M)$.  
By definition, $\la(\pt)\!=\!\R$.
We identify the two orientations of any point with $\pm1$ in the obvious way.

For submanifolds $S'\!\subset\!S\!\subset\!M$, the short exact sequences
\begin{gather*}
0\lra TS\lra TM|_S\lra \cN S\!\equiv\!\frac{TM|_S}{TS}\lra 0  \qquad\hbox{and}\\
0\lra \cN_SS'\!\equiv\!\frac{TS|_{S'}}{TS'}\lra \cN S'\!\equiv\!\frac{TM|_{S'}}{TS'}
\lra \cN S|_{S'}\!\equiv\!\frac{TM|_{S'}}{TS|_{S'}} \lra0 
\end{gather*}
of vector spaces determine isomorphisms
\BE{lasplits_e} \la(M)\big|_S\approx \la(S)\!\otimes\!\la(\cN S)
\quad\hbox{and}\quad
\la(\cN S')\approx \la(\cN_SS')\!\otimes\!\la(\cN S)\big|_{S'}\EE
of line bundles over~$S$ and~$S'$, respectively.
A \sf{co-orientation} of $S$ in $M$ is an orientation of~$\cN S$.
We define the canonical co-orientation~$\fo_{\prt M}^c$ of~$\prt M$ in~$M$
to be given by the outer normal direction.

For a fiber bundle $\ff_{\cM}\!:\cM\!\lra\!\cM'$, we denote by 
$T\cM^v\!\equiv\!\ker\nd\ff_{\cM}$ 
its vertical tangent bundle.
The short exact sequence
\BE{finses_e}0\lra T\cM^v\lra 
T\cM \xlra{\nd\ff_{\cM}} \ff_{\cM}^*T\cM'\lra0\EE
of vector bundles determines an isomorphism 
\BE{lasplits_e2} \la(\cM)\approx \ff_{\cM}^*\la(\cM')\!\otimes\!\la(T\cM^v)\EE
of line bundles over~$\cM$.
The switch of the ordering of the factors in~\eref{lasplits_e2} from~\eref{finses_e}
is motivated by Lemma~\ref{fibrasign_lmm1a}\ref{fibisom_it} below and by the inductive construction
of the orientations~$\fo_{k,l}$ on the real Deligne-Mumford moduli spaces $\ov\cM_{k,l}^{\tau}$
in Section~\ref{cMstrata_subs}.

If $f\!:\cZ\!\lra\!Y$ is a continuous map between two smooth manifolds, 
possibly with boundary, let 
$$\la(f)\!\equiv f^*\la(Y)^*\!\otimes\!\la(\cZ)\lra \cZ\,.$$
A \sf{relative orientation} of~$f$ is an orientation on the line bundle~$\la(f)$.
For a relative orientation~$\fo$ of~$f$ and $u\!\in\!\cZ$, we denote by $\fo_u$ 
the associated homotopy class of trivializations of the fiber~$\la_u(f)$ over~$u$
and the associated homotopy class of isomorphisms \hbox{$\la_u(\cZ)\!\lra\!\la_{f(u)}(Y)$}.
If in addition $\fo'$ is a relative orientation of another continuous map 
\hbox{$g\!:Y\!\lra\!Z$},
we denote~by $\fo\fo'$ the relative orientation of $g\!\circ\!f$ corresponding to
the homotopy class of the compositions
$$\la_u(\cZ)\lra \la_{f(u)}(Y)\lra \la_{g(f(u))}(Z)$$
of isomorphisms in the homotopy classes~$\fo_u$ and~$\fo_{f(u)}'$ for each $u\!\in\!\cZ$.

We identify an orientation~$\fo$ on a manifold $\cZ$ with a relative orientation
of \hbox{$\cZ\!\lra\!\pt$} in the obvious way.
For a submanifold $\cS\!\subset\!\cZ$, we identify 
a co-orientation $\fo_{\cS}^c$ on~$\cS$
with a relative orientation of the inclusion \hbox{$\io_{\cS}\!:\cS\!\lra\!\cZ$}
via the first isomorphism in~\eref{lasplits_e}.
If $\cS'\!\subset\!\cS$ is also a submanifold with a co-orientation $\fo_{\cS'}^c$ in~$\cS$,
then the relative orientation $\fo_{\cS'}^c\fo_{\cS}^c$ of the inclusion
$$\io_{\cS'}\!: \cS'\lra \cS\lra \cZ$$
corresponds to the co-orientation of $\cS'$ in $\cZ$ induced by the co-orientations
 $\fo_{\cS}^c$ and  $\fo_{\cS'}^c$ via the second isomorphism in~\eref{lasplits_e}.
If  $\ff_{\cM}\!:\cM\!\lra\!\cM'$ is a fiber bundle, we similarly identify
an orientation~$\fo_{\cM}^v$ of~$T\cM^v$ with a relative orientation of~$\ff_{\cM}$
via~\eref{lasplits_e2}.
 
If $f$, $\fo$, $\cS$, and $\fo_{\cS}^c$ are as above, we denote by $\fo|_{\cS}$ 
the restriction of the trivialization of~$\la(f)$ determined by~$\fo$ to~$\cS$
and define 
\BE{prtcoorient_e0}\fo_{\cS}\equiv \fo_{\cS}^c\fo\EE
to be \sf{the relative orientation~of $\la(f|_{\cS})$ induced by $\fo$ and~$\fo_{\cS}^c$}.
If $\cZ$ is a manifold with boundary, let
\BE{prtcoorient_e}\prt\big(\cZ,\fo\big)\equiv \big(\prt \cZ,\prt\fo\big)
\equiv \big(\prt\cZ,\fo_{\prt\cZ}^c\fo\big)\,.\EE
If $Y$ is a point (and so $\fo$ and $\prt\fo$ are orientations
on~$\cZ$ and~$\prt\cZ$, respectively),
this convention agrees with \cite[p146]{Warner} if and only if the dimension of~$M$ is odd.
If $\cS\!=\!\{P\}$ is also a point, then the projection isomorphism
\hbox{$T_P\cZ\!\lra\!\cN\cS$} is orientation-preserving with respect to~$\fo$ and~$\fo_{\cS}^c$
if and only~if 
$$\fo_{\cS}^c\fo=+1\,;$$
this is the $\cM',\Ups\!=\!\{\pt\}$ case of Lemma~\ref{fibrasign_lmm1a}\ref{fibisom_it} below.

If $\fo$ is a relative orientation of $f\!:\cZ\!\lra\!Y$ and $u\!\in\!\cZ$ 
is such that $\nd_uf$ is an isomorphism, we define
$$\fs_u(\fo)=\begin{cases}+1,&\hbox{if}~\nd_uf\!\in\!\fo_u;\\
-1,&\hbox{if}~\nd_uf\!\not\in\!\fo_u.\end{cases}$$
If $g\!:Y\!\lra\!Z$ and $\fo'$ are also as above and $\nd_{f(u)}g$ is an isomorphism as well,
then
\BE{fsfoprod_e}\fs_u(\fo\fo')=\fs_u(\fo)\fs_{f(u)}(\fo')\,.\EE
If $y\!\in\!Y$ is a regular value of~$f$ and the set $f^{-1}(y)$ is finite, we define
$$\big|f^{-1}(y)\big|_{\fo}^{\pm}=\sum_{u\in f^{-1}(y)}\!\!\!\!\!\!\fs_u(\fo) \,.$$

Let $\ff_{\cM}\!:\cM\!\lra\!\cM'$ be a fiber bundle. 
If $\Ups\!\subset\!\cM$ is a submanifold and $P\!\in\!\Ups$, then
the differential $\nd_P(\ff_{\cM}|_\Ups)$ is an isomorphism if and only~if
the composition
\BE{fibisom_e} T_P\cM^v\!\equiv\!\ker\nd_P\ff_{\cM}\lra 
T_P\cM\lra \frac{T_P\cM}{T_P\Ups}\!\equiv\!\cN_P\Ups\EE
is.
If $\cM_2$ is another manifold, then 
$$\ff_{\cM}\!\times\!\id_{\cM_2}\!:\cM\!\times\!\cM_2\lra \cM'\!\times\!\cM_2$$
is also a fiber bundle and $\Ups\!\times\!\cM_2\!\subset\!\cM_1\!\times\!\cM_2$ is a submanifold;
see the first diagram in Figure~\ref{fibrasign_fig}.
The differential of 
\BE{pi1dfn_e}\pi_1\!:\cM\!\times\!\cM_2\lra \cM\EE
induces a commutative diagram 
$$\xymatrix{ \ker\!\big\{\ff_{\cM}\!\times\!\id_{\cM_2}\!\big\}^v 
\ar[rr]\ar[d]_{\nd\pi_1}&& \cN\big(\Ups\!\times\!\cM_2\big)\ar[d]^{\nd\pi_1}\\
\pi_1^*T\cM^v\ar[rr]&& \pi_1^*\cN\Ups}$$
of vector bundle homomorphisms.
Since the vertical arrows above are isomorphisms,
they pull back a vertical orientation~$\fo_{\cM}^v$ of~$\ff_{\cM}$
to a vertical orientation~$\pi_1^*\fo_{\cM}^v$ of~$\ff_{\cM}\!\times\!\id_{\cM_2}$
and a co-orientation~$\fo_{\Ups}^c$ of~$\Ups$ to 
a co-orientation~$\pi_1^*\fo_{\Ups}^c$ of~$\Ups\!\times\!\cM_2$.
We note the following.

\begin{lmm}\label{fibrasign_lmm1a}
Suppose $\ff_{\cM}\!:\cM\!\lra\!\cM'$ is a fiber bundle with an orientation~$\fo_{\cM}^v$ on~$T\cM^v$,
$\Ups\!\subset\!\cM$ is a submanifold with a co-orientation~$\fo_\Ups^c$, and 
$P\!\in\!\Ups$ is such that $\nd_P(\ff_{\cM}|_\Ups)$ is an isomorphism.
\BEnum{(\arabic*)}

\item\label{fibisom_it} 
The isomorphism~\eref{fibisom_e} is orientation-preserving with respect to~$\fo_{\cM}^v$
and~$\fo_\Ups^c$ if and only if \hbox{$\fs_P(\fo_\Ups^c\fo_{\cM}^v)\!=\!+1$}.

\item\label{fibisom_it2} 
If $\cM_2$ is another manifold, $\pi_1$ is as in~\eref{pi1dfn_e}, and $P_2\!\in\!\cM_2$, then
$$\fs_{(P,P_2)}\big((\pi_1^*\fo_\Ups^c)(\pi_1^*\fo_{\cM}^v)\big)
=\fs_P\big(\fo_\Ups^c\fo_{\cM}^v\big).$$

\EEnum
\end{lmm}

\begin{figure}
$$\xymatrix{\Ups\!\times\!\cM_2\ar@{^(->}[r] \ar[d]^{\pi_1}&
\cM\!\times\!\cM_2\ar[rr]^{\ff_{\cM}\times\id_{\cM_2}}\ar[d]^{\pi_1}&& 
\cM'\!\times\!\cM_2\ar[d]^{\pi_1}&&
f^{-1}(\Ups)\ar@{^(->}[r]\ar[d]
&\cZ\ar[r]^{\ff_{\cZ}} \ar[d]^f&  \cZ'\ar[d]^{f'}\\
\Ups\ar@{^(->}[r] &\cM\ar[rr]^{\ff_{\cM}}&& \cM'&&
\Ups\ar@{^(->}[r] &\cM\ar[r]^{\ff_{\cM}}& \cM' }$$
\caption{The maps of Lemmas~\ref{fibrasign_lmm1a} and~\ref{fibrasign_lmm1b}.}
\label{fibrasign_fig}
\end{figure}

\vspace{.15in}

Suppose that $\ff_{\cZ}\!:\cZ\!\lra\!\cZ'$ is another fiber bundle, 
$f,f'$ are maps as in the second diagram in Figure~\ref{fibrasign_fig} 
so that it commutes,
and $\fo_{\cZ}^v$ and $\fo_{\cM}^v$ are orientations on~$T\cZ^v$ and~$T\cM^v$,
respectively.
If $u\!\in\!\cZ$ is such that the restriction
\BE{fibisom_e2a}\nd_uf\!:T_u\cZ^v\!\equiv\!\ker\nd_u\ff_{\cZ} \lra T_{f(u)}\cM^v\EE
is an isomorphism, we define $\fs_u(f,\fo_{\cZ}^v,\fo_{\cM}^v)$ to be $+1$ 
if this isomorphism is orientation-preserving with respect to the orientations~$\fo_{\cZ}^v$
and~$\fo_{\cM}^v$ and to be $-1$ otherwise.
If $P\!\in\!\Ups$ are as above, $u\!\in\!f^{-1}(P)$, and
the homomorphisms~\eref{fibisom_e} and~\eref{fibisom_e2a} are isomorphisms, 
then $f$ is transverse to~$\Ups$ at~$u$, 
\hbox{$f^{-1}(\Ups)\!\subset\!\cZ$} is a smooth submanifold near~$u$,
the composition
$$T_u\cZ^v\!\equiv\!\ker\nd_u\ff_{\cZ}\lra 
T_u\cZ\lra \frac{T_u\cZ}{T_uf^{-1}(\Ups)}\!\equiv\!\cN_uf^{-1}(\Ups)$$
is an isomorphism,
and $\nd_ug$ descends to an isomorphism
$$\nd_ug\!: \cN_uf^{-1}(\Ups)\!\equiv\!\frac{T_u\cZ}{T_uf^{-1}(\Ups)} \lra 
\frac{T_P\cM}{T_P\Ups}\!\equiv\!\cN_P\Ups$$
and thus pulls back a co-orientation~$\fo_{\Ups}^c$ on $\Ups\!\subset\!\cM$ 
to a co-orientation $f^*\fo_{\Ups}^c$ on \hbox{$f^{-1}(\Ups)\!\subset\!\cZ$} near~$u$.

\begin{lmm}\label{fibrasign_lmm1b}
Let $\ff_{\cM}$, $\Ups$, $P$, $\fo_{\cM}^v$, and $\fo_\Ups^c$ be as in Lemma~\ref{fibrasign_lmm1a}.
Suppose in addition that $\ff_{\cZ}\!:\cZ\!\lra\!\cZ'$ is another fiber bundle with 
an orientation~$\fo_{\cZ}^v$ on~$T\cZ^v$, 
$f,f'$ are maps so that the second diagram in Figure~\ref{fibrasign_fig} commutes, and
$u\!\in\!f^{-1}(P)$. 
If the homomorphism~\eref{fibisom_e2a} is an isomorphism, then 
$$\fs_u\big((f^*\fo_\Ups^c)\fo_{\cZ}^v\big)=\fs_u(f,\fo_{\cZ}^v,\fo_{\cM}^v)
\fs_P\big(\fo_\Ups^c\fo_{\cM}^v\big)\,.$$
\end{lmm}

\subsection{Intersection signs}
\label{InterSigns_subs}

For continuous maps $f\!:\cZ\!\lra\!Y$ and $g\!:\Ups\!\lra\!Y$ between manifolds with boundary, define 
\begin{gather*}
M_{f,g}\equiv \cZ_f\!\!\times_g\!\Ups
=\big\{(u,P)\!\in\!\cZ\!\!\times\!\!\Ups\!-\!(\prt\cZ)\!\!\times\!\!(\prt \Ups)\!: 
f(u)\!=\!g(P)\big\}, \\
f\!\times_Y\!g\!:M_{f,g}\lra Y,\qquad f\!\times_Y\!g(u,P)=f(u)\!=\!g(P).
\end{gather*}
We call two such maps $f$ and $g$ \sf{strongly transverse} if  they are smooth and
the maps~$f$ and~$f|_{\prt\cZ}$ are transverse to the maps~$g$ and~$g|_{\prt \Ups}$.
The space $M_{f,g}$ is then a smooth manifold and 
\begin{gather}
\notag \dim\, M_{f,g}+\dim\,Y=\dim\,\cZ+\dim\,\Ups\,,\\
\label{prtMf1f2_e}
\prt M_{f,g}=\big(\cZ\!-\!\prt\cZ\big)\,_f\!\!\times_g\!(\prt \Ups)\sqcup 
(\prt \cZ)\,_f\!\!\times_g\!\big(\Ups\!-\!\prt \Ups\big)\,.
\end{gather}

Suppose in addition that 
\BE{finter_e} f\!=\!(f_1,f_2)\!:\cZ\lra Y\!\equiv\!X\!\times\!\cM,\quad
g\!=\!(g_1,g_2)\!:\Ups\lra X\!\times\!\cM,\EE
$\fo_1$ is a relative orientation of~$f_1$, and
$\fo_2$ is a relative orientation of~$g_2$;
see the second rows in the diagrams of Figure~\ref{dropfactor_fig}.
For $(u,P)\!\in\!M_{f,g}$ such~that the homomorphism
\BE{isominter_e}T_u\cZ\!\oplus\!T_P\Ups\lra T_{f(u)}Y\!=\!T_{g(P)}Y, \quad
(v,w)\lra \nd_uf(v)\!+\!\nd_Pg(w),\EE
is an isomorphism, we define $(f,\fo_1)_u\!\cdot_P\!\!(g,\fo_2)$ to be $+1$ 
if  the top exterior power~$\La^{\top}_{\R}$ of this  isomorphism lies in the homotopy class 
determined by $(\fo_1)_u$ and~$(\fo_2)_P$ and to be $-1$ otherwise.
If $M_{f,g}$ is also finite, let
$$\big|M_{f,g}\big|_{\fo_1,\fo_2}^{\pm}
=\sum_{(u,P)\in M_{f,g}}\!\!\!\!\!\!\!\!(f,\fo_1)\,_u\!\cdot_P\!(g,\fo_2)\,.$$

Suppose $\Ups\!\subset\!\cM$ is a bordered submanifold with co-orientation~$\fo_\Ups^c$, 
$g_1\!:\Ups\!\lra\!X$ is a constant map, and $g_2\!:\Ups\!\lra\!\cM$ is the inclusion.
Then,
$$M_{f,g}=\big\{\big(u,f_2(u)\big)\!:
u\!\in\!f_2^{-1}(\Ups)\!\cap\!f_1^{-1}(g_1(\Ups))\!-\!
(\prt\cZ)\!\cap\!f_2^{-1}(\prt \Ups)\big\}.$$
If in addition $f$ is strongly transverse to~$g$, then 
$f_2^{-1}(\Ups)\!\subset\!\cZ$ is a smooth submanifold with co-orientation~$f^*\fo_\Ups^c$
and the restriction
$$f_1\!: f_2^{-1}(\Ups)\lra X$$
is a submersion.
Along with the relative orientation~$\fo_1$ on~$f_1$, 
$f_2^*\fo_\Ups^c$ induces a relative orientation~$(f_2^*\fo_{\Ups}^c)\fo_1$ 
on this restriction as in~\eref{prtcoorient_e0}.
If in addition $S\!\subset\!\cM$ is another submanifold strongly transverse to~$\Ups$,
then the homomorphism
$$ \cN_S(\Ups\!\cap\!S)\!\equiv\!\frac{TS|_{\Ups\cap S}}{T(\Ups\!\cap\!S)}
\lra \frac{T\cM|_{\Ups\cap S}}{T\Ups|_{\Ups\cap S}}\!\equiv\!\cN\Ups|_{\Ups\cap S}$$
of vector bundles is an isomorphism.
The co-orientation $\fo_{\Ups}^c$ of~$\Ups$ in~$\cM$ then restricts 
to a co-orientation $\fo_{\Ups}^c|_{\Ups\cap S}$.
In such a case, 
$$M_{f,g}=M_{f,g|_{\Ups\cap S}}\,.$$
The following observations are straightforward.

\begin{figure}
$$\xymatrix{ f_2^{-1}(\Ups)\ar@{_(->}[d]\ar[r]^{f_1}& X &&
\cZ\ar@{=}[d]\ar[r]^>>>>>{\wt{f}}&  X\!\times\!\cM\!\times\!\cM_2\ar[d] && 
\Ups\!\times\!\cM_2 \ar@{_(->}[ll]_>>>>>>>>>>>{g\times\id_{\cM_2}} \ar[d]^{\pi_1}\\
\cZ\ar[r]^f\ar[dr]&  X\!\times\!\cM \ar[u]& \Ups \ar@{_(->}[l]_g &
\cZ\ar[r]^f&  X\!\times\!\cM && \Ups \ar@{_(->}[ll]_g\\
& X\!\times\!S \ar@{^(->}[u]& 
\Ups\!\cap\!S\ar@{_(->}[l]_g\ar@{^(->}[u]}$$
\caption{The maps of Lemma~\ref{dropfactor_lmm} with $g\!=\!(x,g_2)$ 
for some  $x\!\in\!X$.}
\label{dropfactor_fig}
\end{figure}
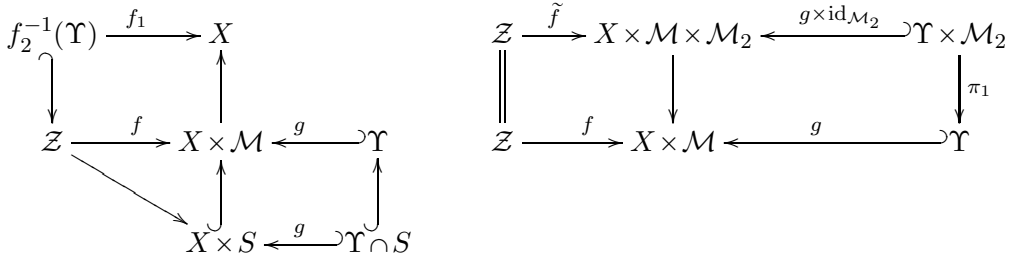

\begin{lmm}\label{dropfactor_lmm}
Suppose $\cZ,\Ups,X,\cM,f,g,f_i,g_i,\fo_1,\fo_\Ups^c$ are as above (with $\Ups\!\subset\!\cM$ and $g_1$ constant), 
$$\dim\,\cZ\!+\!\dim\,\Ups=\dim\,X\!+\!\dim\,\cM\,,$$
and $f$~is strongly transverse to $g$.
\BEnum{(\arabic*)}

\item\label{dropfactor_it1} For every $u\!\in\!f_2^{-1}(\Ups)\!\cap\!f_1^{-1}(g_1(\Ups))$,
$$(f,\fo_1)\,_u\!\cdot_{f_2(u)}\!\big(g,\fo_\Ups^c\big) =
(-1)^{(\dim\,\Ups)(\codim\,\Ups)}\fs_u\big((f_2^*\fo_{\Ups}^c)\fo_1\big).$$

\item\label{dropfactor_it2} 
Suppose $S\!\subset\!\cM$ is another submanifold strongly transverse to~$\Ups$
and $f_2(\cZ)\!\subset\!S$.
For every $(u,P)\!\in\!M_{f,g|_{\Ups\cap S}}$,
$$(f,\fo_1)\,_u\!\cdot_P\!\big(g,\fo_\Ups^c\big) 
=(-1)^{(\codim\,S)(\codim\,\Ups)}
(f,\fo_1)\,_u\!\cdot_P\!\big(g|_{\Ups\cap S},\fo_\Ups^c|_{\Ups\cap S}\big)$$
with the intersection on the right-hand side above taken in~$X\!\times\!S$.

\item\label{dropfactor_it3} 
Suppose $\cM_2$, $\pi_1$, and $\pi_1^*\fo_{\Ups}^c$ are
as in Lemma~\ref{fibrasign_lmm1a}\ref{fibisom_it2} and the second diagram 
in Figure~\ref{dropfactor_fig} commutes.
For every $(u,\wt{P})\!\in\!M_{\wt{f},g\times\id_{\cM_2}}$,
$$(\wt{f},\fo_1)\,_u\!\cdot_{\wt{P}}\!\big(g\!\times\!\id_{\cM_2},\pi_1^*\fo_\Ups^c\big) 
=(-1)^{(\dim\,\cM_2)(\codim\,\Ups)}
(f,\fo_1)\,_u\!\cdot_{\pi_1(\wt{P})}\!\big(g,\fo_\Ups^c\big).$$ 

\EEnum
\end{lmm}

\vspace{.1in}

Let $e_1\!:\cZ_1\!\lra\!X'$ and $e_2\!:\cZ_2\!\lra\!X'$ be strongly transverse maps so that
$$\cZ\equiv M_{e_1,e_2}\equiv\big\{
(u_1,u_2)\!\in\!\cZ_1\!\times\!\cZ_2\!-\!(\prt\cZ_1)\!\times\!(\prt \cZ_2)\!: e_1(u_1)\!=e_2(u_2)\big\}
\subset \cZ_1\!\times\!\cZ_2$$
is a smooth submanifold.
For each $u\!\equiv\!(u_1,u_2)\!\in\!\cZ$, the short exact sequence
\begin{gather*}
0\lra T_u\cZ\lra T_{u_1}\cZ_1\!\oplus\!T_{u_2}\cZ_2\lra T_{e_1(u_1)}X'\!=\!T_{e_2(u_2)}X'\lra0,\\
(v_1,v_2)\lra \nd_{u_2}e_2(v_2)\!-\!\nd_{u_1}e_1(v_1),
\end{gather*}
of vector spaces induces an isomorphism
$$\la_u(\cZ)\!\otimes\!\la\big(T_{e_2(u_2)}X'\big)\approx 
\la_{u_1}(\cZ_1)\!\otimes\!\la_{u_2}(\cZ_2)\,.$$
Combined with relative orientations~$\fo_{11}$ and $\fo_{12}$ of
$$f_{11}\!:\cZ_1\lra X_1 \qquad\hbox{and}\qquad (e_2,f_{12})\!:\cZ_2\lra X'\!\times\!X_2,$$
this isomorphism determines a homotopy class of isomorphisms
\BE{diagorient_e}\begin{split}
\la_u(\cZ)\!\otimes\la_{e_2(u_2)}\big(TX'\big) &\lra 
\la_{f_{11}(u_1)}(X_1)\!\otimes\!
\la_{(e_2(u_2),f_{12}(u_2))}(X'\!\times\!X_2)\\
&\lra
\la_{f_{11}(u_1)}(X_1)\!\otimes\!\la_{f_{12}(u_2)}(X_2)\!\otimes\!\la\big(T_{e_2(u_2)}X'\big)
\end{split}\EE
The homotopy class of trivializations in~\eref{diagorient_e} corresponds
to a relative orientation $(\fo_{11})\,_{e_1}\!\!\times_{e_2}\!\fo_{12}$ of the restriction
$$f_1\!\equiv\!(f_{11}\!\times\!f_{12})\big|_{\cZ}\!:\cZ\lra X_1\!\times\!X_2\,.$$
For a map $f_2\!:\cZ_1\!\lra\!\cM$, let $f_2\!:\cZ\!\lra\!\cM$ also denote 
the composition of~$f_2$ with the projection to~$\cZ_1$.
If in addition
\hbox{$g_2\!:\Ups\!\lra\!\cM$} is the embedding of a (possibly bordered) submanifold,
\hbox{$g_{1i}\!:\Ups\!\lra\!X_i$} are constant~maps with values $x_1$ and $x_2$, respectively,
$$f\!\equiv\!(f_{11},f_{12},f_2)\!:\cZ\lra X_1\!\times\!X_2\!\times\!\cM, \quad\hbox{and}\quad
g\!\equiv\!(g_{11},g_{12},g_2)\!:\Ups\lra X_1\!\times\!X_2\!\times\!\cM,$$
then
$$M_{f,g}=\big\{(u_1,u_2,P)\!:
((u_1,P),u_2)\!\in\!M_{e_1|_{M_{(f_{11},f_2),(g_{11},g_2)}},e_2|_{f_{12}^{-1}(x_2)}}
\big\};$$
see the diagram in Figure~\ref{dropfactor3_fig}.

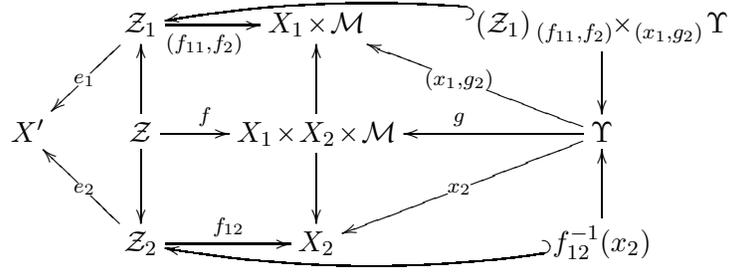
\begin{figure}
$$\xymatrix{ &\cZ_1 \ar[r]_>>>>>>>>{(f_{11},f_2)} \ar[ld]|{e_1}& X_1\!\times\!\cM & 
(\cZ_1)\,_{(f_{11},f_2)}\!\!\times_{(x_1,g_2)}\!\Ups
\ar@/_.7pc/@{^(->}[ll] \ar[d]\\
X' &\cZ\ar[r]^>>>>f \ar[u]\ar[d]& X_1\!\times\!X_2\!\times\!\cM\ar[d]\ar[u] &
\Ups \ar[l]_g \ar[lu]|{(x_1,g_2)} \ar[ld]|{x_2}\\
& \cZ_2 \ar[r]^{f_{12}}\ar[lu]|{e_2}  & X_2 & \ar@/^.7pc/@{_(->}[ll] \ar[u] f_{12}^{-1}(x_2)
}$$
\caption{The maps of Lemma~\ref{dropfactor_lmm3} with $x_i\!\in\!X_i$
corresponding to the constant map~$g_{1i}$.}
\label{dropfactor3_fig}
\end{figure}

\begin{lmm}\label{dropfactor_lmm3}
Suppose $\cZ_1,\cZ_2,\cZ,X',X_1,X_2,\cM,\Ups,e_i,f_{1i},f_2,f,g_{1i},g_2,g,\fo_{1i}$ are as above,
$\fo_\Ups^c$ is a co-orientation on~$\Ups$,
$$\dim\,\cZ_1\!+\!\dim\,\Ups=\dim\,X_1\!+\!\dim\,\cM, 
\qquad \dim\,\cZ_2=\dim\,X'\!+\!\dim\,X_2,$$
the maps $e_1$ and $e_2$ are strongly transverse, and
the maps~$f$ and~$g$ are strongly transverse.
If
$$\big((u_1,P),u_2\big)\in 
 \big((\cZ_1)\,_{(f_{11},f_2)}\!\!\times_{(g_{11},g_2)}\!\Ups\big)\,
_{e_1}\!\!\times_{e_2}\!f_{12}^{-1}\big(g_{12}(\Ups)\big),$$
then
\begin{equation*}\begin{split}
&\big(f,(\fo_{11})\,_{e_1}\!\!\times_{e_2}\!\fo_{12}\big)\,
_{(u_1,u_2)}\!\cdot_{f_2(u_1)}\!\!\big(g,\fo_\Ups^c\big)\\ 
&\hspace{.6in} =(-1)^{(\dim\,X_2)(\codim\,\Ups+\dim\,X')}
\big(((f_{11},f_2),\fo_{11})\,_{u_1}\!\!\cdot_{f_2(u_1)}\!((g_{11},g_2),\fo_\Ups^c)\big)
\fs_{u_2}\big(\fo_{12}\big).
\end{split}\end{equation*}
\end{lmm}

\subsection{Steenrod pseudocycles}
\label{Steenrod_subs}

Let $Y$ be a smooth manifold, possibly with boundary.
For a continuous map $f\!:\cZ\!\lra\!Y$,  let 
$$\Om(f)=\bigcap_{K\subset \cZ\text{~cmpt}}\!\!\!\!\!\!\!\!\ov{f(\cZ\!-\!K)}$$
be \sf{the limit set of~$f$}.
A \sf{$\Z_2$-pseudocycle} into~$Y$ is a continuous map \hbox{$f\!:\cZ\!\lra\!Y$} 
from a manifold, possibly with boundary, so that the closure of $f(\cZ)$ in~$Y$ is compact and
there exists a smooth map $h\!:\cZ'\!\lra\!Y$ such~that 
$$\dim\,\cZ'\le \dim\,\cZ\!-\!2, \qquad \Om(f)\subset h(\cZ'), \qquad 
f(\prt\cZ)\subset(\prt Y)\!\cup\!h(\cZ')\,.$$
The \sf{codimension} of such a $\Z_2$-pseudocycle is \hbox{$\dim\,Y\!-\!\dim\,\cZ$}.
A continuous map \hbox{$\wt{f}\!:\wt\cZ\!\lra\!Y$} from a manifold, possibly with boundary, is 
a \sf{bordered $\Z_2$-pseudocycle with boundary} $f\!:\cZ\!\lra\!Y$ if
the closure of $\wt{f}(\wt\cZ)$ in~$Y$ is compact, 
$$\cZ\subset\prt\wt\cZ, \qquad \wt{f}|_{\cZ}=f,$$
and there exists a smooth map $\wt{h}\!:\wt\cZ'\!\lra\!Y$ such~that 
$$\dim\,\wt\cZ'\le \dim\,\wt\cZ\!-\!2,\qquad
\Om(\wt{f})\subset\wt{h}(\wt\cZ'), \qquad 
\wt{f}\big(\prt\wt\cZ\!-\!\cZ\big)\subset(\prt Y)\!\cup\!\wt{h}(\wt\cZ') \,.$$
Given $\wt{f}$ as above, the choice of $\cZ\!\subset\!\prt\wt\cZ$ is generally not unique,
and the restriction of~$\wt{f}$ to any such~$\cZ$ need not be a $\Z_2$-pseudocycle.
If $\wt\cZ$ is one-dimensional, then $\wt\cZ$ is compact and 
$\wt{f}(\prt\wt\cZ\!-\!\cZ)\!\subset\!\prt Y$.

Two bordered $\Z_2$-pseudocycles 
\hbox{$\wt{f}_1\!:\wt\cZ_1\!\lra\!Y$} and \hbox{$\wt{f}_2\!:\wt\cZ_2\!\lra\!Y$} 
as above are \sf{transverse}~if 
\BEnum{$\bu$}

\item the maps $\wt{f}_1$ and $\wt{f}_2$ are strongly transverse and

\item there exist smooth maps $\wt{h}_1\!:\wt\cZ_1'\!\lra\!Y$ and 
$\wt{h}_2\!:\wt\cZ_2'\!\lra\!Y$
such that $\wt{h}_1$ is transverse to~$\wt{f}_2$ and~$\wt{f}_2|_{\prt\wt\cZ_2}$, 
$\wt{h}_2$ is transverse to~$\wt{f}_1$ and~$\wt{f}_1|_{\prt\wt\cZ_1}$, and
$$\dim\,\wt\cZ_1'\le\dim\,\wt\cZ_1\!-\!2, \quad 
\dim\,\wt\cZ_2'\le\dim\,\wt\cZ_2\!-\!2, \quad 
\Om(\wt{f}_1)\subset\wt{h}_1(\wt\cZ_1'), \quad 
\Om(\wt{f}_2)\!\subset\!\wt{h}_2(\wt\cZ_2').$$

\EEnum
In such a case, 
$$\wt{f}_1\!\times_Y\!\!\wt{f}_2\!:M_{\wt{f}_1,\wt{f}_2}\lra Y$$ 
is a bordered $\Z_2$-pseudocycle with boundary~\eref{prtMf1f2_e}.

A \sf{Steenrod pseudocycle} into~$Y$ is a $\Z_2$-pseudocycle \hbox{$f\!:\cZ\!\lra\!Y$} 
along with a relative orientation~$\fo$ of~$f$.
A bordered $\Z_2$-pseudocycle \hbox{$\wt{f}\!:\wt\cZ\!\lra\!Y$} with boundary~$f$
and a relative orientation~$\wt\fo$ of~$\wt{f}$ is
a \sf{bordered Steenrod pseudocycle} with boundary~$(f,\fo)$ if $\prt\wt\fo\!=\!\fo$.
If $(f,\fo)$ is a codimension~0 Steenrod pseudocycle, then the number
\BE{SteenDeg_e}\deg(f,\fo)\equiv \sum_{u\in f^{-1}(y)}\!\!\!\!\!\!\fs_u(\fo)\in\Z\EE
is well-defined for a generic choice of $y\!\in\!Y$ and is independent of such a choice.
We call this number the \sf{degree} of~$(f,\fo)$.
It vanishes if $(f,\fo)$ bounds a bordered Steenrod pseudocycle~$(\wt{f},\wt\fo)$. 

\begin{lmm}\label{InterOrient_lmm}
Suppose $\cZ,\Ups,X,\cM,Y,f,g,f_i,g_i,\fo_1,\fo_2$ are as in~\eref{finter_e} and just below
and such~that
$$\dim\,\cZ+\dim\,\Ups=\dim\,Y+1\,.$$
If $f$ and $g$ are transverse bordered $\Z_2$-pseudocycles, then
$\cZ_f\!\times_g\!(\prt \Ups)$ and \hbox{$(\prt\cZ)_f\!\times_g\!\Ups$} are finite sets and
\BE{InterOrient_e}
\big|\cZ_f\!\!\times_g\!(\prt \Ups)\big|_{\fo_1,\prt\fo_2}^{\pm}
=(-1)^{\dim\,\Ups}
\big|(\prt\cZ)_f\!\!\times_g\!\Ups\big|_{\prt\fo_1,\fo_2}^{\pm}\,.\EE
\end{lmm}

\begin{proof}
By the transversality and dimension assumptions,  $M_{f,g}$ is a compact one-dimensional manifold 
and
$$\prt M_{f,g}=\cZ\,_f\!\!\times_g\!(\prt \Ups)\sqcup (\prt \cZ)\,_f\!\!\times_g\!\Ups.$$
In particular, the two sets on the right-hand side above are finite.
By a direct computation, this equality respects the orientations with the orientation
on the last fiber product modified by $(-1)^{\codim\,\cZ}$ 
and for a suitably chosen orientation on the left-hand side.
Alternatively, \eref{InterOrient_e} is equivalent~to
$$\big|\cZ_f\!\!\times_g\!(\prt \Ups)\big|_{\fo_1,\prt\fo_2}^{\pm}
=(-1)^{(\dim\,\cZ)(\dim\,\Ups)}
\big|\Ups_g\!\!\times_f\!(\prt\cZ)\big|_{\fo_2,\prt\fo_1}^{\pm}\,.$$
The sign in this statement must be symmetric in $\dim\,\cZ$ and $\dim\,\Ups$, 
depend only on their parity, be $+1$ if both dimensions or codimensions are even,
and be $-1$ for linear maps from intervals to~$\R$.
\end{proof}

\section{Moduli spaces of stable curves}
\label{DM_sec}

\subsection{Main stratum and orientations}
\label{cMstrata_subs}

For $k\!\in\!\Z^{\ge0}$, let $[k]\!=\!\{1,\ldots,k\}$.
If in addition $k\!\ge\!3$, we denote by $\ov\cM_{0,k}$ 
the Deligne-Mumford moduli space of stable rational curves with $k$~marked points.
For $k,l\!\in\!\Z^{\ge0}$ with $k\!+\!2l\!\ge\!3$, 
we denote by $\ov\cM_{k,l}^{\tau}$ the Deligne-Mumford moduli space of 
stable real genus~0 curves
\BE{cCdfn0_e}\cC\equiv \big(\Si,(x_i)_{i\in[k]},(z_i^+,z_i^-)_{i\in[l]},\si\big)\EE
with $k$~real marked points, $l$ conjugate pairs of marked points, and 
an anti-holomorphic involution~$\si$ with separating fixed locus.
This space is a smooth manifold of dimension $k\!+\!2l\!-\!3$,
without boundary if $k\!\ge\!1$ and with boundary if $k\!=\!0$.
The boundary of $\ov\cM_{0,l}^{\tau}$ parametrizes the curves with no irreducible 
component fixed by the involution;
the fixed locus of the involution on a curve in $\prt\ov\cM_{0,l}^{\tau}$ is a single node.
The strata of $\ov\cM_{0,l}^{\tau}$ parametrizing curves with 
two invariant irreducible components sharing a real node are of codimension~1, 
but are not part of~$\prt\ov\cM_{0,l}^{\tau}$.
The moduli space $\ov\cM_{k,l}^{\tau}$ is orientable if and only if $k\!=\!0$
or $k\!+\!2l\!\le\!4$; see \cite[Prop.~1.5]{Penka2}.

The \sf{main stratum}~$\cM_{k,l}^{\tau}$ of~$\ov\cM_{k,l}^{\tau}$ is the quotient~of
\begin{equation*}\begin{split}
\big\{\!\big((x_i)_{i\in[k]},(z_i^+,z_i^-)_{i\in[l]}\big)\!:\,
x_i\!\in\!S^1,\,z_i^{\pm}\!\in\!\P^1\!-\!S^1,\,z_i^+\!=\!\tau(z_i^-),&\\
x_i\!\neq\!x_j,\,z_i^+\!\neq\!z_j^+,z_j^-~\forall\,i\!\neq\!j&\big\}
\end{split}\end{equation*}
by the natural action of the subgroup $\PSL_2^{\tau}\C\!\subset\!\PSL_2\C$ 
of automorphisms of~$\P^1$ commuting with~$\tau$.
The topological components of $\cM_{k,l}^{\tau}$ 
are indexed by the possible distributions of the points~$z_i^+$ between the interiors 
of the two disks cut out by the fixed locus~$S^1$ of 
the standard involution~$\tau$ on~$\P^1$ and by the orderings of 
the real marked points~$x_i$ on~$S^1$.

If $k\!+\!2l\!\ge\!4$ and $i\!\in\![k]$, let
\BE{ff1dfn_e}\ff_{k,l;i}^{\R}\!:\ov\cM_{k,l}^{\tau}\lra\ov\cM_{k-1,l}^{\tau}\EE
be the forgetful morphism dropping the $i$-th real marked point.
The restriction of~$\ff_{k,l;i}^{\R}$ to the preimage of $\cM_{k-1,l}^{\tau}$
is an $S^1$-fiber bundle.
The associated short exact sequence~\eref{finses_e} induces an isomorphism
\BE{cMorientR_e}
\la\big(\cM^\tau_{k,l}\big)\approx 
\ff_{k,l;i}^{\R\,*}\la\big(\cM^\tau_{k-1,l}\big)\big|_{\cM^\tau_{k,l}}
\!\otimes\!\big(\!\ker\nd\ff_{k,l;i}^{\R}\big)\big|_{\cM^\tau_{k,l}}\,.\EE
If $k\!+\!2l\!\ge\!5$ and $i\!\in\![l]$, we similarly denote~by
\BE{ff2dfn_e}\ff_{k,l;i}\!:\ov\cM_{k,l}^{\tau}\lra\ov\cM_{k,l-1}^{\tau}\EE
the forgetful morphism dropping the $i$-th conjugate pair of marked points.
The restriction of~$\ff_{k,l;i}$ to $\cM_{k,l}^{\tau}$
is a dense open subset of a $\P^1$-fiber bundle and thus induces an isomorphism
\BE{cMorientC_e}
\la\big(\cM^\tau_{k,l}\big)\approx 
\ff_{k,l;i}^{\,*}\la\big(\cM^\tau_{k,l-1}\big)\big|_{\cM^\tau_{k,l}}
\!\otimes\!\la\big(\!\ker\nd\ff_{k,l;i}\big)\big|_{\cM^\tau_{k,l}}\,.\EE
For each $\cC\!\in\!\cM^\tau_{k,l}$ as in~\eref{cCdfn0_e},
$$\ker\nd_{\cC}\ff_{k,l;i}\approx T_{z_i^+}\P^1 $$
is canonically oriented by the complex orientation of the fiber~$\P^1$ at~$z_i^+$.
We denote the resulting orientation of the last factor in~\eref{cMorientC_e} by~$\fo_i^+$.

Suppose $l\!\in\!\Z^+$ and $\cC\!\in\!\cM_{k,l}^{\tau}$ is as in~\eref{cCdfn0_e} with $\Si\!=\!\P^1$.
Let $\D^2_+\!\subset\!\C\!\subset\!\P^1$ be the disk cut out by 
the fixed locus~$S^1$ of~$\tau$ which contains~$z_1^+$. 
We orient \hbox{$S^1\!\subset\!\D_+^2\!\subset\!\C$} in the standard way
(this is the opposite of the boundary orientation of~$\D_+^2$ as defined in Section~\ref{SignConv_subs}).
If $k\!+\!2l\!\ge\!4$ and $i\!\in\![k]$,
this determines an orientation~$\fo_i^{\R}$ of the fiber
$$\ker\nd_{\cC}\ff_{k,l;i}^{\R}\approx T_{x_i}S^1 $$
of the last factor in~\eref{cMorientR_e} over~$\ff_{k,l;i}^{\R}(\cC)$.
This orientation extends over the subspace 
$$\ov\cM_{k,l;i}^{\tau;\st}\subset \ov\cM_{k,l}^{\tau}$$ 
consisting of curves~$\cC$ as in~\eref{cCdfn0_e} such that the real marked point $x_i$ of~$\cC$ 
lies on the same irreducible component of~$\Si$ as the marked point~$z_1^+$.

Let $(x_1,x_{j_2(\cC)},\ldots,x_{j_k(\cC)})$ be the ordering of the real marked points of~$\cC$
starting with~$x_1$ and going in the direction of the standard orientation of~$S^1$.
We denote by $\de_{\R}(\cC)\!\in\!\Z_2$ the sign of the permutation sending
$$\vp_{\cC}\!:\big\{2,\ldots,k\big\}\lra \big\{2,\ldots,k\big\}, \quad
\vp_{\cC}(i)=j_i(\cC)\,.$$
If $k\!=\!0$, we take $\de_{\R}(\cC)\!=\!0$.
For $l^*\!\in\![l]$, let 
$$\de_{l^*}^c(\cC)=
\big|\big\{i\!\in\![l]\!-\![l^*]\!:z_i^+\!\not\in\!\D_+^2\big\}\big|+2\Z\in\Z_2.$$
In particular, $\de_{\R}(\cC)\!=\!0$ if $k\!\le\!2$ and $\de_l^c(\cC)\!=\!0$.
The functions~$\de_{\R}$ and~$\de_{l^*}$ are locally constant on~$\cM_{k,l}^{\tau}$.

The space $\cM_{1,1}^{\tau}\!=\!\ov\cM_{1,1}^{\tau}$ is a single point;
we take $\fo_{1,1}\!\equiv\!+1$ to be its orientation as a plus point.
We identify the one-dimensional space 
$\ov\cM_{0,2}^{\tau}$ with $[0,\i]$ via the cross ratio 
\BE{cM02ident_e}
\vph_{0,2}\!:\ov\cM_{0,2}^{\tau}\lra[0,\i], \quad 
\vph\big([(z_1^+,z_1^-),(z_2^+,z_2^-)]\big)= 
\frac{z_2^+\!-\!z_1^-}{z_2^-\!-\!z_1^-}:\frac{z_2^+\!-\!z_1^+}{z_2^-\!-\!z_1^+}
=\frac{|1\!-\!z_1^+/z_2^-|^2}{|z_1^+\!-\!z_2^+|^2}\,;\EE
see Figure~\ref{fig_M_0,2}.
This identification, which is the {\it opposite} of \cite[(3.1)]{RealEnum} and 
\cite[(1.12)]{RealGWsII}, determines an orientation~$\fo_{0,2}$ on~$\ov\cM_{0,2}^{\tau}$.

\begin{figure}\begin{center}\begin{tikzpicture}
\draw (-4.5,0.5) circle [radius=0.5]; 
\draw (-4.5,-0.5) circle [radius=0.5];
\draw [fill] (-4.75,0.7) circle [radius=0.02];
\node [above left] at (-4.75,0.7) {$z_1^+$}; 
\draw [fill] (-4.36,0.7) circle [radius=0.02];
\node [above right] at (-4.36,0.7) {$z_2^-$};
\draw [fill] (-4.75,-0.7) circle [radius=0.02];
\node [below left] at (-4.75,-0.7) {$z_1^-$}; 
\draw [fill] (-4.36,-0.7) circle [radius=0.02];
\node [below right] at (-4.36,-0.7) {$z_2^+$};
\draw (-2.5,0) circle [radius=0.5];
\draw (-3,0) arc [start angle=180, end angle=360, x radius=0.5, y radius=0.2];
\draw [dashed] (-3,0) arc [start angle=180, end angle=0, x radius=0.5, y radius=0.1];
\draw [fill] (-2.7,0.3) circle [radius=0.02];
\node [above left] at (-2.7,0.3) {$z_1^+$}; 
\draw [fill] (-2.7,-0.3) circle [radius=0.02];
\node [below left] at (-2.7,-0.3) {$z_1^-$}; 
\draw [fill] (-2.3,0.3) circle [radius=0.02];
\node [above right] at (-2.3,0.3) {$z_2^-$}; 
\draw [fill] (-2.3,-0.3) circle [radius=0.02];
\node [below right] at (-2.3,-0.3) {$z_2^+$}; 
\draw (-0.5,0) circle [radius=0.5];
\draw (0.5,0) circle [radius=0.5];
\draw (0,0) arc [start angle=180, end angle=360, x radius=0.5, y radius=0.2];
\draw [dashed] (0,0) arc [start angle=180, end angle=0, x radius=0.5, y radius=0.1];
\draw (-1,0) arc [start angle=180, end angle=360, x radius=0.5, y radius=0.2];
\draw [dashed] (-1,0) arc [start angle=180, end angle=0, x radius=0.5, y radius=0.1];
\draw [fill] (-0.75,0.3) circle [radius=0.02];
\node [above left] at (-0.75,0.3) {$z_1^+$}; 
\draw [fill] (-0.75,-0.3) circle [radius=0.02];
\node [below left] at (-0.75,-0.3) {$z_1^-$}; 
\draw [fill] (0.6,0.28) circle [radius=0.02];
\node [above right] at (0.6,0.28) {$z_2^+$};
\draw [fill] (0.6,-0.28) circle [radius=0.02];
\node [below right] at (0.6,-0.28) {$z_2^-$};
\draw (2.5,0) circle [radius=0.5];
\draw (2,0) arc [start angle=180, end angle=360, x radius=0.5, y radius=0.2];
\draw [dashed] (2,0) arc [start angle=180, end angle=0, x radius=0.5, y radius=0.1];
\draw [fill] (2.3,0.3) circle [radius=0.02];
\node [above left] at (2.3,0.3) {$z_1^+$}; 
\draw [fill] (2.3,-0.3) circle [radius=0.02];
\node [below left] at (2.3,-0.3) {$z_1^-$}; 
\draw [fill] (2.7,0.3) circle [radius=0.02];
\node [above right] at (2.7,0.3) {$z_2^+$}; 
\draw [fill] (2.7,-0.3) circle [radius=0.02];
\node [below right] at (2.7,-0.3) {$z_2^-$}; 
\draw (4.5,0.5) circle [radius=0.5]; 
\draw (4.5,-0.5) circle [radius=0.5];
\draw [fill] (4.75,0.7) circle [radius=0.02];
\node [above right] at (4.75,0.7) {$z_2^+$}; 
\draw [fill] (4.36,0.7) circle [radius=0.02];
\node [above left] at (4.36,0.7) {$z_1^+$};
\draw [fill] (4.75,-0.7) circle [radius=0.02];
\node [below right] at (4.75,-0.7) {$z_2^-$}; 
\draw [fill] (4.36,-0.7) circle [radius=0.02];
\node [below left] at (4.36,-0.7) {$z_1^-$};
\draw [fill] (-4.5,-2) circle [radius=0.05];
\draw [fill] (4.5,-2) circle [radius=0.05];
\draw [fill] (0,-2) circle [radius=0.05];
\draw (-4.5,-2) to (4.5,-2); 
\node [below] at (-4.5,-2) {0}; 
\node [below] at (4.5,-2) {$\infty$};
\node [below] at (0,-2) {1};  
\node at (-3.5,0) {$\cdots$}; 
\node at (-1.5,0) {$\cdots$}; 
\node at (1.5,0) {$\cdots$}; 
\node at (3.5,0) {$\cdots$}; 
\end{tikzpicture}\end{center}
\caption{The structure of $\ov\cM_{0,2}^{\tau}$}
\label{fig_M_0,2}\end{figure}

We now define an orientation~$\fo_{k,l}$ on $\cM_{k,l}^{\tau}$ for $l\!\in\!\Z^+$ and
$k\!+\!l\!\ge\!3$ inductively.
If $k\!\ge\!1$, we take~$\fo_{k,l}$ to be so that the $i\!=\!k$ case 
of the isomorphism~\eref{cMorientR_e} is compatible with the orientations~$\fo_{k,l}$,
$\fo_{k-1,l}$, and~$\fo_k^{\R}$ on the three line bundles involved.
If $l\!\ge\!2$, we take~$\fo_{k,l}$ to be so that the $i\!=\!l$ case 
of the isomorphism~\eref{cMorientC_e} is compatible with the orientations~$\fo_{k,l}$,
$\fo_{k,l-1}$, and~$\fo_l^+$.
By a direct check, the orientations on~$\cM_{1,2}^{\tau}$ induced
from~$\cM_{0,2}^{\tau}$ via~\eref{cMorientR_e} and~$\cM_{1,1}^{\tau}$ via~\eref{cMorientC_e} are the same.
Since the fibers of $\ff_{k,l;l}|_{\cM_{k,l}^{\tau}}$ are even-dimensional,
it follows that the orientation~$\fo_{k,l}$ on~$\cM_{k,l}^{\tau}$ is well-defined
for all $l\!\in\!\Z^+$ and $k\!\in\!\Z^{\ge0}$ with $k\!+\!2l\!\ge\!3$.
This orientation is as above \cite[Lemma~5.4]{Penka2}.

For $l^*\!\in\![l]$, we denote by $\fo_{k,l;l^*}$ the orientation on 
$\cM_{k,l}^{\tau}$ which equals~$\fo_{k,l}$ at $\cC$ if and only if 
$\de_{\R}(\cC)\!=\!\de_{l^*}^c(\cC)$.
The next statement is straightforward.

\begin{lmm}\label{cMorient_lmm}
The orientations $\fo_{k,l;l^*}$ on $\cM_{k,l}^{\tau}$ with	
$k,l\!\in\!\Z^{\ge0}$ and $l^*\!\in\![l]$ such that \hbox{$k\!+\!2l\!\ge\!3$}
satisfy the following properties:
\BEnum{($\fo_{\cM}\arabic*$)}

\item\label{cMorientR_it} if $\cC\!\in\!\cM_{k+1,l}^{\tau}$,
the isomorphism~\eref{cMorientR_e} with $(k,i)$ replaced by $(k\!+\!1,j_{k+1}(\cC))$ 
respects the orientations $\fo_{k+1,l;l^*}$, $\fo_{k,l;l^*}$, and~$\fo_{k+1}^{\R}$
at~$\cC$; 

\item\label{cMorientC_it} 
the isomorphism~\eref{cMorientC_e} with $(l,i)$ replaced by $(l\!+\!1,l^*\!+\!1)$
respects the orientations $\fo_{k,l+1;l^*+1}$, $\fo_{k,l;l^*}$, and~$\fo_{l^*+1}^+$; 

\item the interchange of two real points $x_i$ and $x_j$ with $2\!\le\!i,j\!\le\!k$
preserves~$\fo_{k,l;l^*}$;

\item\label{cM1Rch_it} if $\cC\!\in\!\cM_{k,l}^{\tau}$,
the interchange of the real points $x_1$ and $x_{j_i(\cC)}$ with $2\!\le\!i\!\le\!k$
preserves~$\fo_{k,l;l^*}$ at~$\cC$ 
if and only if $(k\!-\!1)(i\!-\!1)\!\in\!2\Z$;

\item\label{Cijinter_it} if $\cC\!\in\!\cM_{k,l}^{\tau}$ and the marked points $z_i^+$ and $z_j^+$ 
are not separated by the fixed locus~$S^1$ of~$\cC$, then
the interchange of the conjugate pairs $(z_i^+,z_i^-)$ and $(z_j^+,z_j^-)$
preserves~$\fo_{k,l;l^*}$ at~$\cC$;  

\item the interchange of the points in a conjugate pair $(z_i^+,z_i^-)$
with $l^*\!<\!i\!\le\!l$ preserves~$\fo_{k,l;l^*}$;

\item the interchange of the points in a conjugate pair $(z_i^+,z_i^-)$ with $1^*\!<\!i\!\le\!l^*$
reverses~$\fo_{k,l;l^*}$;

\item  the interchange of the points in the conjugate pair $(z_1^+,z_1^-)$ preserves
$\fo_{k,l;l^*}$ if and only~if 
$$k\neq0~~\hbox{and}~~l\!-\!l^*\cong\binom{k}{2}~\tn{mod}~2 
\quad\hbox{or}\quad
k=0~~\hbox{and}~~l\!-\!l^*\cong1~\tn{mod}~2.$$ 

\EEnum
\end{lmm}

\subsection{Codimension 1 strata and degrees}
\label{cMorient_subs}

The (open) \sf{codimension~1 strata} of $\ov\cM_{k,l}^{\tau}\!-\!\prt\ov\cM_{k,l}^{\tau}$ correspond 
to the sets $\{(K_1,L_1),(K_2,L_2)\}$ such that 
$$[k]=K_1\!\sqcup\!K_2, \quad  [l]=L_1\!\sqcup\!L_2, \quad
|K_1|\!+\!2|L_1|\ge2, \quad |K_2|\!+\!2|L_2|\ge2\,.$$
The stratum $\oS$ corresponding to such a set  parametrizes marked curves~$\cC$ as in~\eref{cCdfn0_e} 
so that the underlying surface~$\Si$ consists of two real irreducible components
with one of them carrying the real marked points~$x_i$ with $i\!\in\!K_1$ and
the conjugate pairs of marked points $(z_i^+,z_i^-)$ with $i\!\in\!L_1$ and
the other component carrying the other marked points.
A \sf{closed codimension~1} stratum~$\ov{S}$ is the closure of such an open stratum~$\oS$.
Thus,
\BE{Ssplit_e0} \oS\approx\cM_{|K_1|+1,|L_1|}^{\tau}\!\times\!\cM_{|K_2|+1,|L_2|}^{\tau},
\quad \ov{S}\approx\ov\cM_{|K_1|+1,|L_1|}^{\tau}\!\times\!\ov\cM_{|K_2|+1,|L_2|}^{\tau}.\EE

Let $l\!\in\!\Z^+$. If $\oS$ is a codimension~1 stratum of $\ov\cM_{k,l}^{\tau}\!-\!\prt\ov\cM_{k,l}^{\tau}$ 
and 
$\cC\!\in\!\oS$, we denote by $\P^1_1$ the irreducible component of~$\cC$ containing
the marked points~$z_1^{\pm}$, by $\P^1_2$ the other irreducible component,
and by $S^1_1\!\subset\!\P^1_1$ and $S^1_2\!\subset\!\P^1_2$ the fixed loci of the involutions
on these components.
For $r\!=\!1,2$, we then take 
$K_r(S)$ and $L_r(S)$ to be the set of real marked points 
and the set of conjugate pairs of marked points, respectively, carried by~$\P^1_r$
and define 
$$k_r(S)= \big|K_r(S)\big|  \qquad\hbox{and}\qquad l_r(S)=\big|L_r(S)\big|.$$
For $i\!\in\![l]$, we denote by 
$$\oS_i\subset\ov\cM_{k,l}^{\tau} \qquad\hbox{and}\qquad \ov{S}_i\subset \ov\cM_{k,l}^{\tau} $$
the open codimension~1 stratum parametrizing marked curves consisting of two real spheres 
with the marked points~$z_i^{\pm}$ on one of them and all other marked points on the other sphere
and its closure, respectively.

If $\ov{S}\!\subset\!\ov\cM_{k,l}^{\tau}\!-\!\prt\ov\cM_{k,l}^{\tau}$ 
is a closed codimension~1 stratum different from~$\ov{S}_1$, let
\BE{ff1dfn_e2b}\ff_{S;1}\!:\ov{S}\lra\ov\cM_{k_1(S),l_1(S)}^{\tau}\!\times\!\ov\cM_{k_2(S)+1,l_2(S)}^{\tau}\EE
denote the composition of the second identification in~\eref{Ssplit_e0}
with the forgetful morphism
$$\ff_{\nod}^{\R}\!:\ov\cM_{k_1(S)+1,l_1(S)}^{\tau}\lra\ov\cM_{k_1(S),l_1(S)}^{\tau}$$  
as in~\eref{ff1dfn_e} dropping the marked point~$\nod$ corresponding to the node.
The vertical tangent bundle of $\ff_{S;1}|_{\oS}$ is a pullback
of the vertical tangent bundle of $\ff_{\nod}^{\R}|_{\cM_{k_1(S)+1,l_1(S)}^{\tau}}$ 
and thus inherits an orientation from
the orientation~$\fo_{\nod}^{\R}$ of the latter specified in Section~\ref{cMstrata_subs};
we denote the induced orientation also by~$\fo_{\nod}^{\R}$.
It extends over the subspace 
$$S^{\st}\subset \ov{S}\subset \ov\cM_{k,l}^{\tau}$$
of curves~$\cC$ so that the marked point $\nod$ of the first component of the image of~$\cC$
under~\eref{Ssplit_e0} lies on the same irreducible component of the domain
as the marked point corresponding to~$z_1^+$.

Let $\Ups\!\subset\!\ov\cM_{k,l}^{\tau}$ be a bordered hypersurface.
If $k\!+\!2l\!\ge\!4$ and $i\!\in\![k]$, we call~$\Ups$ 
\sf{regular with respect to~}$\ff_{k,l;i}^{\R}$ if 
$\Ups\!\subset\!\ov\cM_{k,l;i}^{\tau;\st}$,
$\ff_{k,l;i}^{\R}(\ov\Ups\!-\!\Ups)$ is contained in the strata of codimension at least~2, 
i.e.~the subspace of $\ov\cM_{k-1,l}^{\tau}$ parametrizing curves with at least two nodes, and 
$\ff_{k,l;i}^{\R}(\prt\Ups)$ is contained in the union of $\prt\ov\cM_{k-1,l}^{\tau}$ 
and the strata of codimension at least~2.
By the last two assumptions, $\ff_{k,l;i}^{\R}|_\Ups$ is a $\Z_2$-pseudocycle of codimension~0;
see Section~\ref{Steenrod_subs}.
By the first assumption, the orientation~$\fo_i^{\R}$ of the last factor in~\eref{cMorientR_e}
and a co-orientation~$\fo_\Ups^c$ on~$\Ups$ induce a relative orientation~$\fo_\Ups^c\fo_i^{\R}$
of~$\ff_{k,l;i}^{\R}|_\Ups$; see the paragraph above Lemma~\ref{fibrasign_lmm1a}. 
Let
$$\deg_i^{\R}\!\big(\Ups,\fo_\Ups^c\big)\equiv
\deg\!\big(\ff_{k,l;i}^{\R}|_\Ups,\fo_\Ups^c\fo_i^{\R}\big)$$
be the degree of the Steenrod pseudocycle $(\ff_{k,l;i}^{\R}|_\Ups,\fo_\Ups^c\fo_i^{\R})$;
see~\eref{SteenDeg_e}.

Suppose in addition that $S\!\subset\!\ov\cM_{k,l}^{\tau}\!-\!\prt\ov\cM_{k,l}^{\tau}$ 
is a codimension~1 stratum.
We call~$\Ups$ \sf{regular with respect to~$S$}
if $\Ups$ and $\prt\Ups$ are transverse to~$\ov{S}$ in $\ov\cM_{k,l}^{\tau}$,
$$\Ups\!\cap\!\ov{S}\approx\Ups_1\!\times\!\ov\cM_{k_2(S)+1,l_2(S)}^{\tau}$$
under the second identification in~\eref{Ssplit_e0} for some 
$\Ups_1\!\subset\!\ov\cM_{k_1(S)+1,l_1(S);\nod}^{\tau;\st}$,
$\ff_{S;1}((\ov\Ups\!-\!\Ups)\!\cap\!\ov{S})$ is contained in
the strata of codimension at least~2 of the target of~$\ff_{S;1}$, 
and $\ff_{S;1}(\prt\Ups\!\cap\!\ov{S})$ is contained in the union of 
the boundary and the strata of codimension at least~2 of the target of~$\ff_{S;1}$.
By the first and the last two assumptions, 
$\ff_{S;1}|_{\Ups\cap \ov{S}}$  is a $\Z_2$-pseudocycle of codimension~0. 
By the first assumption, a co-orientation~$\fo_{\Ups}^c$ on~$\Ups$ in~$\ov\cM_{k,l}^{\tau}$ 
determines a co-orientation
$$\fo_{\Ups\cap S}^c\equiv\fo_{\Ups}^c\big|_{\Ups\cap \ov{S}}$$
on $\Ups\!\cap\!\ov{S}$ in~$\ov{S}$.
By the second assumption, $\Ups\!\cap\!\ov{S}\!\subset\!S^{\st}$.
By the first two assumptions, $S\!\neq\!S_1$ if $\Ups\!\cap\!\ov{S}\!\neq\!\eset$
and that $\fo_{\Ups}^c$ and
the orientation~$\fo_{\nod}^{\R}$ of the fibers of the restriction of~\eref{ff1dfn_e2b} to~$\oS$
specified above induce a
relative orientation $\fo_{\Ups\cap S}^c\fo_{\nod}^{\R}$ of~$\ff_{S;1}|_{\Ups\cap \ov{S}}$.
Let
$$\deg_S\!\big(\Ups,\fo_{\Ups}^c\big)\equiv
\deg\!\big(\ff_{S;1}|_{\Ups\cap \ov{S}},\fo_{\Ups}^c\fo_{\nod}^{\R}\big)
\equiv\deg\!\big(\ff_{S;1}|_{\Ups\cap \ov{S}},\fo^c_{\Ups\cap S}\fo_{\nod}^{\R}\big).$$

We call a bordered hypersurface $\Ups\!\subset\!\ov\cM_{k,l}^{\tau}$ \sf{regular}
if $\ov\Ups\!-\!\Ups$ is contained in the strata of codimension at least~2 and
$\Ups$ is regular with respect to the forgetful morphism~$\ff_{k,l;i}^{\R}$ for every $i\!\in\![k]$
and with respect to every codimension~1 stratum 
\hbox{$S\!\subset\!\ov\cM_{k,l}^{\tau}\!-\!\prt\ov\cM_{k,l}^{\tau}$}.
For such a hypersurface, \hbox{$\Ups\!\cap\!\ov{S}_1\!=\!\eset$}.

\subsection{Strata orientations}
\label{NBstrata_subs}

Suppose $l\!\ge\!2$ and $k\!+\!2l\!\ge\!5$.
The moduli space~$\ov\cM_{k,l}^{\tau}$ contains codimension~2 strata~$\oGa$
that parametrize marked curves~$\cC$ as in~\eref{cCdfn0_e} so that 
the underlying surface~$\Si$ consists of one real component~$\P^1_0$ and 
one pair~$\P^1_{\pm}$ of conjugate components; see Figure~\ref{LiftedRel_fig}.
We do not distinguish these strata based on the ordering of the marked points on 
the fixed locus $S^1_1\!\subset\!\P^1_0$ of the involution.
For such a stratum~$\oGa$, let $l_0(\Ga),l_{\C}(\Ga)\!\in\!\Z^{\ge0}$ be the number 
of conjugate pairs of marked points carried by~$\P^1_0$ and $\P^1_-\!\cup\!\P^1_+$,
respectively.
In particular, 
$$l_{\C}(\Ga)\ge2 \qquad\hbox{and}\qquad l_0(\Ga)\!+\!l_{\C}(\Ga)=l.$$
The closure~$\ov\Ga$ of~$\oGa$ decomposes~as 
\BE{Gasplit_e} \ov\Ga\approx \ov\cM_{k,l_0(\Ga)+1}^{\tau}\!\times\!\ov\cM_{0,l_{\C}(\Ga)+1}\,.\EE
We call a codimension~2 stratum as above \sf{primary} if the marked point~$z_1^+$ of 
the curves~$\cC$ in~$\oGa$ is carried by $\P^1_-\!\cup\!\P^1_+$.

For a primary codimension~2 stratum $\oGa$ and $\cC\!\in\!\oGa$, 
we denote by $\P^1_+$ the irreducible component of~$\cC$ carrying the marked point~$z_1^+$.
In this case, we choose the identification~\eref{Gasplit_e} so~that
\BEnum{($\fo_{\Ga}\arabic*$)}

\item\label{GasplitC_it} 
the second factor on the right-hand side parametrizes the irreducible component~$\P^1_+$
with its marked points so that the node~$z_{\C}$ separating it from~$\P^1_0$
is the {\it first} marked point,

\item the node $z_{\R}^+$ separating $\P^1_0$ from~$\P^1_+$ 
is the {\it first} marked point in the {\it first} conjugate pair of marked points in
the corresponding element in the first factor on the right-hand side, and

\item\label{GasplitIndex_it} the remaining conjugate pairs of points and the real points
in the first factor 
on the right-hand side are numbered in the same order as on the left-hand side.

\EEnum
If in addition $l^*\!\in\![l]$, let $l_0^*(\Ga)$ (resp.~$l_-^*(\Ga)$)
be the number of marked points $z_i^-$ with $i\!\in\![l^*]$
carried by~$\P^1_0$ (resp.~$\P^1_+$).
The second factor in~\eref{Gasplit_e} is canonically oriented (being a complex manifold).
We denote by $\fo_{\Ga;l^*}$ the orientation on~$\oGa$ obtained via 
the identification~\eref{Gasplit_e} from 
the orientation $\fo_{k,l_0(\Ga)+1;l_0^*(\Ga)+1}$ on $\cM_{k,l_0(\Ga)+1}^{\tau}$
times $(-1)^{l_-^*(\Ga)}$.

With the identification as above, let 
$$\pi_1,\pi_2\!:\ov\Ga\lra \ov\cM_{k,l_0(\Ga)+1}^{\tau},\ov\cM_{0,l_{\C}(\Ga)+1}$$
be the projections to the two factors.
Denote~by
$$\cL_{\Ga}^{\R}\lra \ov\cM_{k,l_0(\Ga)+1}^{\tau} \qquad\hbox{and}\qquad
\cL_{\Ga}^{\C}\lra \ov\cM_{0,l_{\C}(\Ga)+1}$$
the universal tangent line bundles at the first point of the first conjugate pair
of marked points and at the first marked point, respectively.
The normal bundle~$\cN\Ga$ consists of conjugate smoothings of the two nodes 
of the curves in~$\oGa$.
Thus, it is canonically isomorphic to the complex line bundle
$$\cL_{\Ga}\equiv \pi_1^*\cL_{\Ga}^{\R}\!\otimes_{\C}\!\pi_2\cL_{\Ga}^{\C}\lra\Ga\,.$$
The next observation is straightforward.

\begin{lmm}\label{cNGa2_lmm}
Suppose $k,l\!\in\!\Z^{\ge0}$ and $l^*\!\in\![l]$ are 
such that \hbox{$k\!+\!2l\!\ge\!3$}.
Let \hbox{$\oGa\!\subset\!\ov\cM_{k,l}^{\tau}$} be a primary codimension~2 stratum.
The orientation~$\fo_{\Ga}^c$ on~$\cN\Ga$ induced by 
the orientations~$\fo_{k,l;l^*}$ on~$\cM_{k,l}^{\tau}$
and~$\fo_{\Ga;l^*}$ on~$\oGa$ agrees with the complex orientation of~$\cL_{\Ga}$.
\end{lmm}

Suppose now that $l\!\in\!\Z^+$ and
$\oS$ is a codimension~1 stratum of $\ov\cM_{k,l}^{\tau}\!-\!\prt\ov\cM_{k,l}^{\tau}$. 
For $r\!=\!1,2$, let 
$$K_r(S)\subset[k], \qquad L_r(S)\subset[l], \qquad\hbox{and}\quad 
k_r(S),l_r(S)\!\in\!\Z^{\ge0}$$
be as in Section~\ref{cMorient_subs}. Define
$$r(S)=\begin{cases}1,&\hbox{if}~k\!=\!0~\hbox{or}~1\!\in\!K_1(S);\\
2,&\hbox{if}~1\!\in\!K_2(S);\end{cases}$$
thus, the real marked point~$x_1$ lies on~$S^1_{r(S)}$ if $k\!\ge\!1$.

For $l^*\!\in\![l]$ and $r\!=\!1,2$, define
$$L_r^*(S)=L_r(S)\!\cap\![l^*], \qquad l_r^*(S)\equiv \big|L_r^*(S)\big|.$$ 
An orientation~$\fo_{S;\cC}^c$ of the normal bundle~$\cN_{\cC}S$ of~$S$ 
in $\ov\cM_{k,l}^{\tau}$ at $\cC\!\in\!S$ determines a direction of degeneration of elements of 
$\cM_{k,l}^{\tau}$ to~$\cC$.
The orientation~$\fo_{k,l;l^*}$ on~$\cM_{k,l}^{\tau}$ limits to 
an orientation~$\fo_{k,l;l^*;\cC}$ of $\la_{\cC}(\ov\cM_{k,l}^{\tau})$
obtained by approaching~$\cC$ from this direction.
Along with~$\fo_{S;\cC}^c$, $\fo_{k,l;l^*;\cC}$ determines 
an orientation $\prt_{\fo_{S;\cC}^c}\fo_{k,l;l^*;\cC}$ of~$\la_{\cC}(S)$
via the first isomorphism in~\eref{lasplits_e}.
If in addition $l_2^*(S)\!\ge\!1$,  let $i^*\!\in\!L_2^*(S)$ be the smallest element.
The two directions of degeneration of elements of $\cM_{k,l}^{\tau}$ to~$\cC$
are then distinguished by whether the marked points $z_1^+$ and $z_{i^*}^+$ of
the degenerating elements lie on the same disk~$\D^2_+$ or not.
We denote by~$\fo_{S;\cC}^{c;+}$ the orientation of~$\cN_{\cC}S$ 
which corresponds to the direction of degeneration for which $z_1^+,z_{i^*}^+\!\in\!\D^2_+$
and by~$\fo_{S;\cC}^{c;-}$ the opposite orientation.
Let $\fo_{k,l;l^*;\cC}^{\pm}$ and $\fo_{S;l^*;\cC}^{\pm}$ 
be the orientations of $\la_{\cC}(\ov\cM_{k,l}^{\tau})$ and~$\la_{\cC}(S)$,
respectively, induced by~$\fo_{S;\cC}^{c;\pm}$ as above. 

A topological component~$\oS_*$ of~$\oS$ is characterized by the distribution of 
the points~$z_i^+$ with $i\!\in\!L_r(S)$ between the interiors of the two disks 
cut out by the fixed locus~$S^1_r$ in each component~$\P^1_r$ of the domain of 
the curves in~$S$ and by the orderings of 
the real marked points~$x_i$ with $i\!\in\!K_r(S)$ on~$S^1_r$.
Thus, 
\BE{Ssplit_e} \oS_*\approx \cM_1\!\times\!\cM_2
\subset\cM_{k_1(S)+1,l_1(S)}^{\tau}\!\times\!
\cM_{k_2(S)+1,l_2(S)}^{\tau}\EE
for some topological components $\cM_1$ and $\cM_2$ of the moduli spaces
on the right-hand side above.
We choose this identification so~that
\BEnum{($\fo_S\arabic*$)}

\item\label{SsplitC_it} 
the orderings of the conjugate pairs of marked points on the two sides are consistent,

\item\label{SsplitR_it}  
the nodal point on each of the irreducible components on the left-hand side
corresponds to the {\it first} real marked point in the associated factor on
the right-hand side.

\EEnum

\vspace{.15in}

If in addition $l_2^*(S)\!\ge\!1$ and  $i^*\!\in\!L_2^*(S)$ is the smallest element as before,
we denote by~$\fo_{S;l^*}$ the orientation on~$\oS$ obtained via 
the identification~\eref{Ssplit_e} from the orientations 
$\fo_{k_1(S)+1,l_1(S);l_1^*(S)}$ on $\cM_{k_1(S)+1,l_1(S)}^{\tau}$
and $\fo_{k_2(S)+1,l_2(S);l_2^*(S)}$ on $\cM_{k_2(S)+1,l_2(S)}^{\tau}$.
The orientation~$\fo_{S;l^*}$ does not depend on the orderings of the real points
on~$S^1_1$ and~$S^1_2$.
In this case, both fixed loci $S^1_r\!\subset\!\P_r^1$ are canonically oriented.
For a topological component~$\oS_*$ of~$\oS$, 
let $j_1'(\oS_*)\!\in\!\Z^{\ge0}$ be the number of real marked points that lie 
on the oriented arc of~$S_{r(S)}^1$ between the nodal point of~$\P^1_{r(S)}$ 
and the real marked point~$x_1$ of any $\cC\!\in\!\oS_*$; if $k\!=\!0$, we take $j_1'(\oS_*)\!=\!0$.
Define
\begin{alignat*}{2}
\de_{\C;l^*}^+(S)&=1, &\quad 
\de_{\R}^+(\oS_*)&=(k\!-\!1)j_1'(\oS_*)\!+\!\big(r(S)\!-\!1\big)k_1(S)k_2(S),\\
\de_{\C;l^*}^-(S)&=l_2(S)\!-\!l_2^*(S),&\quad
\de_{\R}^-(\oS_*)&=(k\!-\!1)j_1'(\oS_*)\!+\!
\binom{k_2(S)\!+\!1}{2}\!+\!\big(r(S)\!-\!1\big)(k\!-\!1).
\end{alignat*}

\begin{lmm}\label{DMboundary_lmm} 
Suppose $k,l\!\in\!\Z^{\ge0}$ and $l^*\!\in\![l]$ are 
such that \hbox{$k\!+\!2l\!\ge\!3$}.
Let \hbox{$\oS_*\!\subset\!\ov\cM_{k,l}^{\tau}\!-\!\prt\ov\cM_{k,l}^{\tau}$} be a topological component
of codimension~1 stratum such that $l_2^*(S)\!\ge\!1$.
The orientations~$\fo_{S;l^*}^{\pm}$ and~$\fo_{S;l^*}$ on $\la(S)|_{\oS_*}$  are the same 
if and only if $\de_{\C;l^*}^{\pm}(S)\!\cong\!k\!+\!\de_{\R}^{\pm}(\oS_*)$ mod~2.
\end{lmm}

\begin{proof} For $r\!=\!1,2$, let 
$$l_r=l_r(S), \qquad l_r^*=l_r^*(S), \qquad  k_r=k_r(S), \qquad j_1'=j_1'(\oS_*).$$
If $l^*\!=\!l\!=\!2$ and $k\!=\!0$, 
$S\!=\!S_1\!=\!S_2$ is a point and~$\fo_{S;l^*}\!=\!+1$.
The claim in this case thus holds by the definition of the orientations 
$\fo_{0,2;2}\!=\!\fo_{0,2}$ on~$\cM_{0,2}^{\tau}$ and $\fo_{S;\cC}^{c;\pm}$ on~$\cN S$.
Since the orientation \hbox{$\fo_{0,l;l}\!\equiv\!\fo_{0,l}$} with $l\!\ge\!3$ 
(resp.~$\fo_{1,l;l}\!\equiv\!\fo_{1,l}$ with $l\!\ge\!2$) is obtained
from the orientations~$\fo_{0,l-1;l-1}$ (resp.~$\fo_{1,l-1;l-1}$) and~$\fo_l^+$,
it follows that the claim holds whenever $l^*\!=\!l$ and $k\!=\!0$.

Let $\cC\!\in\!\oS_*$ be as in~\eref{cCdfn0_e}.
Suppose $l^*\!<\!l$ and $k\!=\!0$.  
Let $l_1^c$ and~$l_2^c$ be the numbers of the marked points $z_i^-$ of~$\cC$ with 
\hbox{$i\!\in\![l]\!-\![l^*]$}  on the same disk as~$z_1^+$ and on the same disk as~$z_{i^*}^+$, 
respectively.
By definition,
\begin{alignat*}{2}
\fo_{1,l_1;l_1^*}\big|_{\cM_1}
&=(-1)^{l_1^c}\fo_{1,l_1;l_1}\big|_{\cM_1}\,, &\quad
\fo_{S;l^*}^+&=(-1)^{l_1^c+l_2^c}\fo_{S;l}^+,\\
\fo_{1,l_2;l_2^*}\big|_{\cM_2}
&=(-1)^{l_2^c}\fo_{1,l_2;l_2}\big|_{\cM_2}\,,&\quad
\fo_{S;l^*}^-&=(-1)^{l_1^c+(l_2-l_2^*-l_2^c)}\fo_{S;l}^-\,.
\end{alignat*}
Thus, the claim in this case follows from the $l^*\!=\!l$ case above.

Suppose $k\!>\!0$, $S'\!\subset\!\cM_{0,l}^{\tau}$ is the image
of~$S$ under the forgetful morphism
$$\ff\!:\cM_{k,l}^{\tau}\lra\cM_{0,l}^{\tau}$$ 
dropping all real marked points, $\cC'\!=\!\ff(\cC)$, and 
$(\cC_1',\cC_2')\!\in\!\cM_1'\!\times\!\cM_2'$
is the corresponding pair of marked irreducible components
(with 1~real marked point each).
Let $(x_{i_1},\ldots,x_{i_{k_1}})$ be the ordering of the real marked points 
on~$S^1_1$ along its canonical direction starting from the first point after the node
and $(x_{j_1},\ldots,x_{j_{k_2}})$ be the analogous ordering of the real marked points
on~$S^1_2$.
The orientation~$\fo_{S;l^*}$ on~$T_{\cC}\oS$ is obtained via isomorphisms
\BE{DMboundary_e5}\begin{split}
\big(T_{\cC}\oS,\fo_{S;l^*}\big)
&\approx \big(T_{\cC_1'}\cM_1',\fo_{1,l_1;l_1^*}\big)
\!\oplus\!\bigoplus_{m=1}^{k_1}\!\!T_{x_{i_m}}\!S^1_1
\oplus \big(T_{\cC_2'}\cM_2',\fo_{1,l_2;l_2^*}\big)
\!\oplus\!\bigoplus_{m=1}^{k_2}\!\!T_{x_{j_m}}\!S^1_2\\
&\approx \big(T_{\cC_1'}\cM_1',\fo_{1,l_1;l_1^*}\big)\!\oplus\!
\big(T_{\cC_2'}\cM_2',\fo_{1,l_2;l_2^*}\big)\oplus
\bigoplus_{m=1}^{k_1}\!\!T_{x_{i_m}}\!S^1_1\!\oplus\!\bigoplus_{m=1}^{k_2}\!\!T_{x_{j_m}}\!S^1_2\\
&\approx \big(T_{\cC'}S',\fo_{S';l^*}\big)\oplus
\bigoplus_{m=1}^{k_1}\!\!T_{x_{i_m}}\!S^1_1\!\oplus\!\bigoplus_{m=1}^{k_2}\!\!T_{x_{j_m}}\!S^1_2
\end{split}\EE
from  the standard orientations on $S^1_1$ and $S^1_2$
determined by the marked points~$z_1^+$ and~$z_{i^*}^+$.
The second isomorphism above is orientation-preserving because 
the dimension of $T_{\cC_2'}\cM_2'$ is even.

Let $\wt\cC\!\in\!\cM_{k,l}^{\tau}$ be a smooth marked curve 
close to~$\cC$ from the direction of degeneration
determined by~$\fo_S^{c;\pm}$ and $\wt\cC'\!=\!\ff(\wt\cC)$.
Let $(x_1,x_{i_2^{\pm}},\ldots,x_{i_k^{\pm}})$ be the ordering 
of the real marked points of~$\wt\cC$ along the standard direction of~$S^1$ determined 
by~$z_1^+(\wt\cC)$.
The orientation~$\fo_{S;l^*}^{\pm}$ at~$\cC$ is obtained via isomorphisms
\BE{DMboundary_e9}\begin{split}
\big(T_{\cC}\oS,\fo_{S;l^*}^{\pm}\big)
\!\oplus\!\big(\cN_{\cC}S,\fo_{S}^{c;\pm}\big)
&\approx\big(T_{\wt\cC}\cM_{k,l}^{\tau},\fo_{k,l;l^*}\big)
\approx  \big(T_{\wt\cC'}\cM_{0,l}^{\tau},\fo_{0,l;l^*}\big)\!\oplus\!
\bigoplus_{m=1}^k\!\!T_{x_{i_m^{\pm}}}\!S^1\\
&\approx \big(T_{\cC'}\oS',\fo_{S';l^*}^{\pm}\big)\!\oplus\!
\big(\cN_{\cC'}S',\fo_{S'}^{c;\pm}\big)\!\oplus\!
\bigoplus_{m=1}^k\!\!T_{x_{i_m^{\pm}}}\!S^1\\
&\approx(-1)^k\big(T_{\cC'}\oS',\fo_{S';l^*}^{\pm}\big)\!\oplus\!
\bigoplus_{m=1}^k\!\!T_{x_{i_m^{\pm}}}\!S^1\!\oplus\!
\big(\cN_{\cC}S,\fo_{S}^{c;\pm}\big).
\end{split}\EE
By \eref{DMboundary_e5}, \eref{DMboundary_e9}, and the $k\!=\!0$ case above,
the claim in the general case holds if $\de_{\R}^{\pm}(\oS_*)$ has the same parity
as the parity of the permutation
\BE{DMboundary_e11}
\big(i_1,\ldots,i_{k_1},j_1,\ldots,j_{k_2}\big) \lra 
\big(i_1^{\pm}\!=\!1,i_2^{\pm},\ldots,i_k^{\pm}\big)\EE
plus the parity of $k_2$ in the minus case, since the tangent spaces $T_{x_{j_m}}S^1_2$
then enter with the reversed orientations.

Suppose $r(S)\!=\!1$. 
The plus case of~\eref{DMboundary_e11} then moves the indices $(i_1,\ldots,i_{j_1'})$
to the end preserving their order.
The parity of this permutation~is 
$$\de_{\R}^+(\oS_*)=j_1'\big(k\!-\!j_1'\big)\cong (k\!-\!1)j_1' \mod2\,.$$
The minus case of~\eref{DMboundary_e11} is the composition of the permutation
\BE{DMboundary_e15}\big(j_1,\ldots,j_{k_2}\big)\lra\big(j_{k_2},\ldots,j_1\big) \EE
with the transposition in the plus case.
This adds an extra $k_2(k_2\!-\!1)/2$ to the parity.

Suppose $r(S)\!=\!2$. 
The plus case of~\eref{DMboundary_e11} then moves $(j_{j_1'+1}\!=\!1,\ldots,j_{k_2})$
to the front preserving their order.
The parity of this permutation~is 
$$\de_{\R}^+(\oS_*)=\big(k_2\!-\!j_1'\big)\big(k_1\!+\!j_1'\big)
\cong (k\!-\!1)j_1'\!+\!k_1k_2 \mod2\,.$$
The minus case of~\eref{DMboundary_e11} consists of the permutation~\eref{DMboundary_e15}
followed by moving $(j_{j_1'+1},\ldots,j_1)$ to the front of the entire $k$~tuple.
The parity of this permutation plus $k_2$ is 
$$\de_{\R}^-(\oS_*)=\binom{k_2\!+\!1}{2}+\big(j_1'\!+\!1\big)\big(k\!-\!1\!-\!j_1'\big)
\cong (k\!-\!1)j_1'+k\!-\!1+\binom{k_2\!+\!1}{2}\,.$$
This establishes the claim.
\end{proof}

\subsection{Bordisms in $\ov\cM_{1,2}^{\tau}$ and $\ov\cM_{0,3}^{\tau}$}
\label{cMbordism_subs}

The two relations of Theorem~\ref{SolWDVV_thm} are proved by applying~\eref{bndsplit_e} with 
the hypersurfaces $\Ups\!\subset\!\ov\cM_{1,2}^\tau$ and $\Ups\!\subset\!\ov\cM_{0,3}^{\tau}$
of Lemmas~\ref{M12rel_lmm} and~\ref{M03rel_lmm} below. 
These hypersurfaces are \sf{regular}, in the sense defined at the end of Section~\ref{cMorient_subs},
and in particular are disjoint from the codimension~1 stratum~$S_1$ of the moduli space.
We determine the degrees of these hypersurfaces with respect to the other non-boundary codimension~1 strata
and with respect to the forgetful morphism~$\ff_{1,2;1}^{\R}$ in the first case.
These degrees are essential for computing the right-hand side of~\eref{bndsplit_e};
see Proposition~\ref{Rdecomp_prp}.

Orientations are interpreted below as relative orientations of maps to a point;
see Section~\ref{SignConv_subs}.
All notation for the codimension~1 strata and the degrees is as in Section~\ref{cMorient_subs}.
For a primary codimension~2 stratum~$\Ga$ of $\ov\cM^\tau_{k,l}$, we denote by~$\fo_{\Ga}^c$ 
the canonical orientation on~$\cN\Ga$ as in Lemma~\ref{cNGa2_lmm} and 
by~$\fo_{\Ga;l}$ the orientation on~$\Ga$ as in the first half of Section~\ref{NBstrata_subs}.
Since $\fo_{k,l}\!=\!\fo_{k,l;l}$ for $k\!=\!0,1$,
\BE{M12rel_e0}\fo_{\Ga;l}=\fo_{\Ga}^c\fo_{k,l}\EE
in the cases of Lemmas~\ref{M12rel_lmm} and~\ref{M03rel_lmm}.
Let $P^{\pm}\!\in\!\ov\cM_{1,2}^{\tau}$ be 
the three-component curve so that $z_1^+$ and~$z_2^{\pm}$ lie on the same irreducible
component.

\begin{lmm}\label{M12rel_lmm}
There exists an embedded closed path $\Ups\!\subset\!\ov\cM_{1,2}^\tau$ 
with a co-orientation~$\fo_\Ups^c$ so that $\Ups$ 
is a regular hypersurface and
\BE{M12rel_e}\prt\big(\Ups,\fo_\Ups^c\big)=\big(P^+,\fo_{P^+}^c\big)\!\sqcup\!\big(P^-,\fo_{P^-}^c\big),
\quad \deg_1^{\R}\!\big(\Ups,\fo_{\Ups}^c\big)=1, 
\quad \deg_{S_2}\!\!\big(\Ups,\fo_{\Ups}^c\big)=-1\,.\EE
\end{lmm}

\begin{proof} Since $(P^+,\fo_{P^{\pm};2})$ is a $\pm$-point, \eref{M12rel_e0} gives 
\BE{M12rel_e1}\fo_{P^{\pm}}^c\fo_{1,2}=\pm1.\EE
Let $\wch\cM_{1,2}^{\tau}\!\approx\!S^2$ be the space obtained by contracting $S_1$
to a point~$P_0$.
By \cite[Lemma~5.4]{Penka2}, the orientation~$\fo_{1,2}$ on~$\cM_{1,2}^{\tau}$ extends
over $\wch\cM_{1,2}^{\tau}$; this can also be readily seen from the definitions.
The morphisms~$\ff_{1,2;1}^{\R}$ and~$\ff_{1,2;2}$ descend to smooth maps
$$\ff_{1,2;1}^{\R}\!: \wch\cM_{1,2}^{\tau}\lra\ov\cM_{0,2}^{\tau} \qquad\hbox{and}\qquad
\ff_{1,2;2}\!: \wch\cM_{1,2}^{\tau}\lra\ov\cM_{1,1}^{\tau}\,.$$
We can identify $\wch\cM_{1,2}^{\tau}$ with $S^2\!\subset\!\R^3$ 
and $\ov\cM_{0,2}^{\tau}$ with $[-1,1]$ so that 
$P^{\pm}\!=\!(\pm1,0,0)$ and $\ff_{1,2;1}^{\R}$ is the height function.
The fibers of~$\ff_{1,2;1}^{\R}$ over $\cM_{0,2}^{\tau}$ are then the circles of 
constant latitude.
The orientation~$\fo_1^{\R}$ of the fibers of $\ff_{1,2;1}^{\R}|_{\cM_{1,2}^{\tau}}$
specified in Section~\ref{cMstrata_subs} extends over the equator $\ov{S}_2\!\approx\!S^1$.
By Lemma~\ref{cMorient_lmm}\ref{cMorientR_it}, 
\BE{M12rel_e5}\fo_{1,2}\big|_{\cM_{1,2}^{\tau}}= 
\big(\fo_1^{\R}\fo_{0,2}\big)\big|_{\cM_{1,2}^{\tau}} \,.\EE

Let $\Ups'\!\subset\!\wch\cM_{1,2}^{\tau}$ 
be a meridian running from~$P^-$ to~$P^+$ disjoint from~$P_0$ and
$\fo_{\Ups'}$ be its canonical orientation.
Thus, the restriction
$$\ff_{1,2;1}^{\R}\!: \big(\Ups',\fo_{\Ups'}\big)\lra \big(\ov\cM_{0,2}^{\tau},\fo_{0,2}\big)$$
is an orientation-preserving diffeomorphism.
We take~$\fo_{\Ups'}^c$ to be the orientation of~$\cN \Ups'$ so that the projection
$$\big(\!\ker\nd\ff_{1,2;1}^{\R},\fo_1^{\R}\big)\lra \big(\cN \Ups',\fo_{\Ups'}^c\big)$$
is an orientation-preserving isomorphism.
By~\eref{M12rel_e5} and Lemma~\ref{fibrasign_lmm1a}\ref{fibisom_it},
\BE{M12rel_e7}\fo_{\Ups'}=\fo_{\Ups'}^c\fo_{1,2} \qquad\hbox{and}\qquad
\deg_1^{\R}\!\big(\Ups',\fo_{\Ups'}^c\big)\!\equiv\!
\deg\!\big(\ff_{1,2;1}^{\R}|_{\Ups'},\fo_{\Ups'}^c\fo_1^{\R}\big)=1.\EE
By Lemma~\ref{cMorient_lmm}\ref{cMorientC_it}, 
the orientation~$\fo_{1,2}$ corresponds to the natural orientation of
the complex coordinate~$z_2^+$ with $z_1^+\!=\!0$ and $x_1\!=\!1$ fixed.
Thus, $\fo_{\Ups'}^c$ is the negative rotation in the $z_2^+$-coordinate.
Along~$\ov{S}_2$, it corresponds to the negative rotation of the node.
Thus,  
$$\deg_{S_2}\!\!\big(\Ups',\fo_{\Ups'}^c\big)\!\equiv\!
\deg\!\big(\ff_{0,2;2}|_{\Ups'\cap \ov{S}_2},\fo_{\Ups'}^c\fo_{\nod}^{\R}\big)=-1.$$

Since the outer normal co-orientation $\fo_{\prt \Ups'}^c$ of~$\prt \Ups'$ 
agrees with the restriction of~$\pm\fo_{\Ups'}$ at~$P^{\pm}$,  i.e.
$$\big(\fo_{\prt \Ups'}^c\fo_{\Ups'}\big)\big|_{P^{\pm}}=\pm1,$$
the first statement in~\eref{M12rel_e7} gives 
$$\big(\fo_{\prt \Ups'}^c\fo_{\Ups'}^c\big)\big|_{P^{\pm}}\fo_{1,2}\big|_{P^{\pm}}
=\fo_{\prt \Ups'}^c\big|_{P^{\pm}}\fo_{\Ups'}\big|_{P^{\pm}}=\pm1.$$
Comparing with~\eref{M12rel_e5}, we conclude that 
$$\big(\fo_{\prt \Ups'}^c\fo_{\Ups'}\big)\big|_{P^{\pm}}=\fo_{P^{\pm}}^c\,,$$
i.e.~the first equality in~\eref{M12rel_e} with $\Ups$ replaced by~$\Ups'$ holds as well.

We take $\Ups\!\subset\!\ov\cM_{1,2}^{\tau}$ to be the preimage of~$\Ups'$ under the blowdown map
(which is a diffeomorphism on a neighborhood of~$\Ups$) 
and $\fo_\Ups^c$ to be the pullback of~$\fo_{\Ups'}^c$.
\end{proof}

The moduli space $\ov\cM_{0,3}^{\tau}$ is a 3-manifold with the~boundary
$$\prt\ov\cM_{0,3}^{\tau} =\ov{S}_{23}^{++}\sqcup \ov{S}_{23}^{+-}\sqcup \ov{S}_{23}^{-+}\sqcup \ov{S}_{23}^{--}\,,$$
where 
$$S_{ij}^{\pm\pm}\approx\ov\cM_{0,4}\approx S^2$$ 
is the closure of the open codimension~1 stratum  $\oS_{ij}^{\pm\pm}$ 
of curves consisting of a pair of conjugate spheres with the marked points
$z_i^{\pm}$ and $z_j^{\pm}$  on the same sphere as~$z_1^+$;
see \cite[Fig.~4]{RealEnum}
and the first diagram in Figure~\ref{fig_curveshapes}.
There are four primary codimension~2 strata $\Ga_i^{\pm}$, with $i\!=\!2,3$,
in $\ov\cM_{0,3}^{\tau}$.
The closed interval~$\ov\Ga_i^+$ (resp.~$\ov\Ga_i^-$) is the closure of the open codimension~2 
stratum~$\oGa_i^+$ (resp.~$\oGa_i^-$) of curves consisting of one real sphere and 
a conjugate pair of spheres so that the real sphere carries the marked points~$z_i^{\pm}$
and the decorations~$^{\pm}$ of the marked points on each of the conjugate spheres
are the same (resp.~different); 
see the last pair of diagrams in Figure~\ref{fig_curveshapes}.
Let 
$$\OGa_i^+\!=\!\Ga_i^+\!\cup\!\big(\ov\Ga_i^+\!\cap\!\ov{S}_i) \subset \ov\Ga_i^+$$
be the complement of the endpoints of~$\ov\Ga_i^+$.

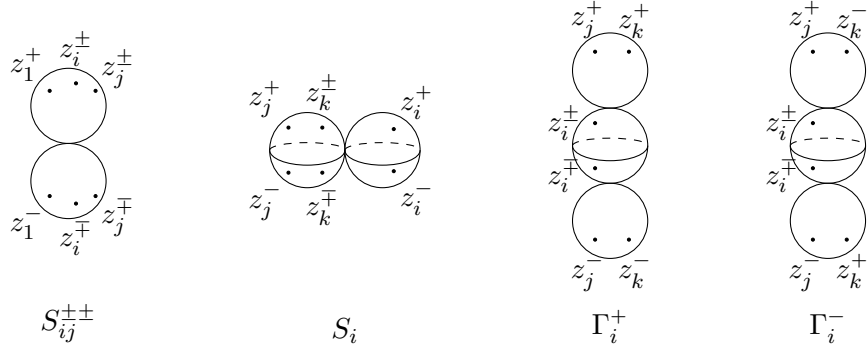
\begin{figure}
\begin{center}
\begin{tikzpicture}
\draw (0,0.5) circle [radius=0.5]; 
\draw (0,-0.5) circle [radius=0.5];
\draw [fill] (-0.25,0.7) circle [radius=0.02];
\node [above left] at (-0.2,0.7) {$z_1^+$}; 
\draw [fill] (0.1,0.8) circle [radius=0.02];
\node [above] at (0.1,0.9) {$z_i^\pm$}; 
\draw [fill] (0.36,0.7) circle [radius=0.02];
\node [above right] at (0.3,0.6) {$z_j^\pm$};
\draw [fill] (-0.25,-0.7) circle [radius=0.02];
\node [below left] at (-0.2,-0.7) {$z_1^-$}; 
\draw [fill] (0.1,-0.8) circle [radius=0.02];
\node [below] at (0.1,-0.9) {$z_i^\mp$}; 
\draw [fill] (0.36,-0.7) circle [radius=0.02];
\node [below right] at (0.3,-0.6) {$z_j^\mp$};
\node at (0,-2.4) {$\oS_{ij}^{\pm\pm}$}; 
\end{tikzpicture} \hspace{1cm}
\begin{tikzpicture}
\draw (-0.5,0) circle [radius=0.5];
\draw (0.5,0) circle [radius=0.5];
\draw (0,0) arc [start angle=180, end angle=360, x radius=0.5, y radius=0.2];
\draw [dashed] (0,0) arc [start angle=180, end angle=0, x radius=0.5, y radius=0.1];
\draw (-1,0) arc [start angle=180, end angle=360, x radius=0.5, y radius=0.2];
\draw [dashed] (-1,0) arc [start angle=180, end angle=0, x radius=0.5, y radius=0.1];
\draw [fill] (-0.75,0.3) circle [radius=0.02];
\node [above left] at (-0.7,0.3) {$z_j^+$}; 
\draw [fill] (-0.75,-0.3) circle [radius=0.02];
\node [below left] at (-0.7,-0.3) {$z_j^-$}; 
\draw [fill] (-0.3,0.3) circle [radius=0.02];
\node [above] at (-0.3,0.4) {$z_k^\pm$}; 
\draw [fill] (-0.3,-0.3) circle [radius=0.02];
\node [below] at (-0.3,-0.4) {$z_k^\mp$};
\draw [fill] (0.65,0.28) circle [radius=0.02];
\node [above right] at (0.6,0.28) {$z_i^+$};
\draw [fill] (0.65,-0.28) circle [radius=0.02];
\node [below right] at (0.6,-0.28) {$z_i^-$};
\node at (0,-2.4) {$\oS_i$};
\end{tikzpicture} \hspace{1cm}
\begin{tikzpicture}
\draw (0,0) circle [radius=0.5];
\draw (-0.5,0) arc [start angle=180, end angle=360, x radius=0.5, y radius=0.2];
\draw [dashed] (-0.5,0) arc [start angle=180, end angle=0, x radius=0.5, y radius=0.1];
\draw [fill] (-0.2,0.3) circle [radius=0.02];
\node [left] at (-0.25,0.3) {$z_i^\pm$};
\draw [fill] (-0.2,-0.3) circle [radius=0.02];
\node [left] at (-0.25,-0.4) {$z_i^\mp$};
\draw (0,1) circle [radius=0.5]; 
\draw (0,-1) circle [radius=0.5];
\draw [fill] (-0.2,1.25) circle [radius=0.02];
\node [above] at (-0.3,1.3) {$z_j^+$};
\draw [fill] (-0.2,-1.25) circle [radius=0.02];
\node [below] at (-0.3,-1.3) {$z_j^-$};
\draw [fill] (0.25,1.25) circle [radius=0.02];
\node [above] at (0.33,1.3) {$z_k^+$};
\draw [fill] (0.25,-1.25) circle [radius=0.02];
\node [below] at (0.33,-1.3) {$z_k^-$};
\node at (0, -2.4) {$\oGa_i^+$};
\end{tikzpicture} \hspace{1cm}
\begin{tikzpicture}
\draw (0,0) circle [radius=0.5];
\draw (-0.5,0) arc [start angle=180, end angle=360, x radius=0.5, y radius=0.2];
\draw [dashed] (-0.5,0) arc [start angle=180, end angle=0, x radius=0.5, y radius=0.1];
\draw [fill] (-0.2,0.3) circle [radius=0.02];
\node [left] at (-0.25,0.3) {$z_i^\pm$};
\draw [fill] (-0.2,-0.3) circle [radius=0.02];
\node [left] at (-0.25,-0.4) {$z_i^\mp$};
\draw (0,1) circle [radius=0.5]; 
\draw (0,-1) circle [radius=0.5];
\draw [fill] (-0.2,1.25) circle [radius=0.02];
\node [above] at (-0.3,1.3) {$z_j^+$};
\draw [fill] (-0.2,-1.25) circle [radius=0.02];
\node [below] at (-0.3,-1.3) {$z_j^-$};
\draw [fill] (0.25,1.25) circle [radius=0.02];
\node [above] at (0.33,1.3) {$z_k^-$};
\draw [fill] (0.25,-1.25) circle [radius=0.02];
\node [below] at (0.33,-1.3) {$z_k^+$};
\node at (0, -2.4) {$\oGa_i^-$};
\end{tikzpicture}
\end{center}
\caption{Elements of open codimension 1 and 2 strata of $\ov\cM_{0,3}^\tau$,
with $\{i,j\}\!=\!\{2,3\}$ in the first diagram and 
$\{i,j,k\}\!=\!\{1,2,3\}$ in the other four.}
\label{fig_curveshapes}
\end{figure}

\begin{lmm}\label{M03rel_lmm}
There exist a bordered surface $\Ups\!\subset\!\ov\cM_{0,3}^{\tau}$ 
with a co-orientation~$\fo_\Ups^c$ and a one-dimensional manifold 
$\ga'\!\subset\!\ov\cM_{0,3}^{\tau}$ with a co-orientation~$\fo_{\ga'}^c$
so~that $\Ups$ is transverse to all open strata of~$\ov\cM_{0,3}^{\tau}$
not contained in any~$\ov\Ga_i^{\pm}$ with $i\!=\!2,3$,
$\Ups$ is a regular hypersurface, and
\begin{gather}\label{M03rel_e0}
\prt\big(\Ups,\fo_\Ups^c\big)=\big(\OGa_2^+,\fo_{\Ga_2^+}^c\big)\!\cup\!
\big(\OGa_3^+,-\fo_{\Ga_3^+}^c\big)\!\cup\!
\big(\OGa_2^-,\fo_{\Ga_2^-}^c\big)\!\cup\!
\big(\OGa_3^-,-\fo_{\Ga_3^-}^c\big)\!\cup\!
\big(\ga',\fo_{\ga'}^c\big),\\
\notag
\ga'\subset\prt\ov\cM_{0,3}^{\tau},  
\qquad \deg_{S_2}\!\!\big(\Ups,\fo_{\Ups}^c\big)=1, \qquad 
\deg_{S_3}\!\!\big(\Ups,\fo_{\Ups}^c\big)=-1.
\end{gather}
\end{lmm}

\begin{proof}
For $i\!=\!2,3$, $z_i^+$ moves in $(\oGa_i^+,\fo_{\Ga_i^+;3})$ (resp.~$(\oGa_i^-,\fo_{\Ga_i^-;3})$) 
from the node separating the sphere carrying~$z_1^-$ (resp.~$z_1^+$) to the other node. 
Each closed interval~$\ov\Ga_i^{\pm}$ intersects~$\ov{S}_i$ transversally 
at one~point~$P_i^{\pm}$ and does not intersect~$\ov{S}_j$ for $j\!=\!1,2,3$ with $j\!\neq\!i$.
It intersects $\prt\ov\cM_{0,3}^\tau$ at its endpoints;
we denote the starting point by~$P_i^{\pm-}$ and the ending point by~$P_i^{\pm+}$.
By \cite[Section~3]{RealEnum}, the orientation $\fo_{0,3}\!=\!\fo_{0,3;3}$ 
on $\cM_{0,3}^{\tau}$ extends over~$\ov\cM_{0,3}^{\tau}$.

By \cite[Remark~3.5]{RealEnum}, $\ov\cM_{0,3}^{\tau}$ is the blowup of
a bordered manifold $\wch\cM_{0,3}^{\tau}$ at a point~$P_0$ with the exceptional 
divisor~$\ov{S}_1$. 
Denote~by 
$$p:\ov\cM_{0,3}^{\tau}\lra \wch\cM_{0,3}^{\tau}$$ 
the blowdown map. 
The morphisms~$\ff_{0,3;2}$ and~$\ff_{0,3;3}$ descend to smooth maps
\BE{M03rel_e3}\ff_{0,3;2}\!: \wch\cM_{0,3}^{\tau}\lra\ov\cM_{0,2}^{\tau} 
\qquad\hbox{and}\qquad
\ff_{0,3;3}\!: \wch\cM_{0,3}^{\tau}\lra\ov\cM_{0,2}^{\tau}\,.\EE
Since $\ov{S}_1$ is disjoint from the four spheres of~$\prt\ov\cM_{0,3}^{\tau}$ and
the four intervals~$\ov\Ga_i^{\pm}$ with $i\!=\!2,3$,  $p$ is a diffeomorphism on 
neighborhoods of these spaces.
We denote the images of these intervals and the twelve points $P_i^{\pm},P_i^{\pm\pm}$
on them under~$p$ in the same~way.
The spaces 
$$\wch{S}_2\!\equiv\!p(\ov{S}_2)\approx\wch\cM_{1,2}^{\tau}\!\approx\!S^2
\qquad\hbox{and}\qquad
\wch{S}_3\!\equiv\!p(\ov{S}_3)\approx\wch\cM_{1,2}^{\tau}$$
are the fibers of~$\ff_{0,3;3}$ and~$\ff_{0,3;2}$, respectively, 
over the curve consisting of two real
components, which corresponds to $1\!\in\![0,\i]$ 
under the identification~$\vph_{0,2}$ in~\eref{cM02ident_e}.

Setting $(z_1^+,z_1^-)\!=\!(0,\i)$, we obtain a natural identification
\begin{gather*}
\wch\cM_{0,3}^{\tau}\!-\!\prt\wch\cM_{0,3}^{\tau}
\approx\big\{\big((z_2^+,z_2^-),(z_3^+,z_3^-)\big)\!\in\!(\P^1)^4\!:
z_i^+\!=\!\tau(z_i^-),\,(z_2^{\pm},z_3^{\pm})\!\neq\!(0,0)\big\}\!\big/\!\!\sim,\\
\quad \big((z_2^+,z_2^-),(z_3^+,z_3^-)\big)\!\sim\!\big((zz_2^+,zz_2^-),(zz_3^+,zz_3^-)\big)
\quad\forall\,z\!\in\!S^1.
\end{gather*}
The condition $z_i^+\!=\!\tau(z_i^-)$ implies that the points~$z_i^+$ and~$z_i^-$
lie on a great arc through the poles $z_1^+\!\equiv\!0$ and $z_1^-\!\equiv\!\i$
(or~lie at~$z_1^{\pm}$).
The blowup point~$P_0$ in this identification corresponds to the point $[(1,1),(1,1)]$.
The projections~\eref{M03rel_e3} in this identification  are given~by 
$$\ff_{0,3;i}\big([(z_2^+,z_2^-),(z_3^+,z_3^-)]\big)=\big[(0,\i),(z_{5-i}^+,z_{5-i}^-)\big].$$
In particular, the fiber of~$\ff_{0,3;i}$ over a point of 
$[(0,\i),(z_{5-i}^+,z_{5-i}^-)]$ of $\ov\cM_{0,2}^{\tau}\!-\!\prt\ov\cM_{0,2}^{\tau}$
can be identified via $z_i^+$ with~$\P^1$ by choosing $z_{5-i}^+\!\in\!\R^+$.

The space $\wch\cM_{0,3}^{\tau}\!-\!\prt\wch\cM_{0,3}^{\tau}$ is covered by two charts
\BE{M03rel_e7}\begin{aligned}
\R^+\!\times\!\P^1&\lra
\wch\cM_{0,3}^{\tau}\!-\!\ff_{0,3;2}^{-1}(\prt\ov\cM_{0,2}^{\tau}), &
(r_2,z_3)&\lra \big[(r_2,1/r_2),(z_3,1/\ov{z_3})\big],\\
\P^1\!\times\!\R^+&\lra
\wch\cM_{0,3}^{\tau}\!-\!\ff_{0,3;3}^{-1}(\prt\ov\cM_{0,2}^{\tau}), &
(z_2,r_3)&\lra \big[(z_2,1/\ov{z_2}),(r_3,1/r_3)\big].
\end{aligned}\EE
In these charts,
\BE{M03rel_e8}\begin{aligned}
\vph_{0,2}\big(\ff_{0,3;3}(r_2,z_3)\big)&=1/r_2^2, &\qquad 
\wch{S}_2&=\big\{(r_2,z_3)\!\in\!\R^+\!\times\!\P^1\!:r_2\!=\!1\big\},\\
\vph_{0,2}\big(\ff_{0,3;2}(z_2,r_3)\big)&=1/r_3^2, &\qquad
\wch{S}_3&=\big\{(z_2,r_3)\!\in\!\P^1\!\times\!\R^+\!:r_3\!=\!1\big\}.
\end{aligned}\EE
The overlap map between the two charts,
$$\R^+\!\times\!\C^*\lra \C^*\!\times\!\R^+, \qquad
\big(r_2,r_3\tne^{\fI\th}\big)\lra \big(r_2\tne^{-\fI\th},r_3\big),$$
is orientation-preserving with respect to the standard orientations 
$\fo_{\R^+}$ on~$\R^+$ and~$\fo_{\P^1}$ on~$\P^1$. 
We take $\wch\fo_{0,3}$ to be the orientation on~$\wch\cM_{0,3}^{\tau}$ 
{\it opposite} to the orientation determined by~$\fo_{\R^+}$ and~$\fo_{\P^1}$
via the two charts in~\eref{M03rel_e7}.
Since the~map
\BE{M03rel_e8a}\big(\R^+,\fo_{\R^+}\big)\lra\big(\R^+,-\fo_{\R^+}\big), \qquad 
r_2\lra \vph_{0,2}\big(\ff_{0,3;3}(r_2,z_3)\big),\EE
is orientation-preserving for each $z_3^+\!\in\!\P^1$ fixed,
Lemma~\ref{cMorient_lmm}\ref{cMorientC_it} and \eref{M12rel_e0} give
\BE{M03rel_e8b} 
\fo_{0,3}\big|_{\ov\cM_{0,3}^{\tau}-\ov{S}_1}=p^*\wch\fo_{0,3}\big|_{\ov\cM_{0,3}^{\tau}-\ov{S}_1},
\qquad
\fo_{\Ga_i^{\pm};3}=\fo_{\Ga_i^{\pm}}^c\wch\fo_{0,3}\,.\EE

The moduli space $\ov\cM_{0,3}^{\tau}$ is a submanifold of $\ov\cM_{0,6}$.
By \cite[Appendix~D.4,5]{MS}, the four cross-ratios
$$\CR_{\pm\pm}^{\tau}\!:\cM_{0,3}^{\tau}\lra\P^1, \quad
\CR_{\pm\pm}^{\tau}\big(\big[(z_i^+,z_i^-)_{i\in[3]}\big]\big)= 
\frac{z_2^{\pm}\!-\!z_1^-}{z_3^{\pm}\!-\!z_1^-}:
\frac{z_2^{\pm}\!-\!z_1^+}{z_3^{\pm}\!-\!z_1^+}\,,$$
extend over $\ov\cM_{0,3}^{\tau}$ and descend to smooth maps from $\wch\cM_{0,3}^{\tau}$.
The subspace 
$$\wt\Ups'\subset\wch\cM_{0,3}^{\tau}\!-\!\big\{P_2^{\pm\pm},P_3^{\pm\pm}\big\}$$ 
where all four cross-ratios  take values~in
\BE{M03rel_e8d} \R\P^1\equiv [-\i,\i]/-\!\i\!\sim\!\i\EE
is an orientable surface, as explained in the next paragraph.
The boundary of~$\wt\Ups'$ consists of the complement of two points
in a circle on each boundary sphere of~$\prt\wch\cM_{0,3}^{\tau}$.

The intersections of~$\wt{\Ups}'$ with the charts~\eref{M03rel_e7} are given~by
\BE{M03rel_e9}\begin{aligned}
\R^+\!\times\!\R\P^1&\lra
\wt{\Ups}'\!-\!\ff_{0,3;2}^{-1}(\prt\ov\cM_{0,2}^{\tau}), &
(r_2,r_3)&\lra \big[(r_2,1/r_2),(r_3,1/r_3)\big],\\
\R\P^1\!\times\!\R^+&\lra
\wt{\Ups}'\!-\!\ff_{0,3;3}^{-1}(\prt\ov\cM_{0,2}^{\tau}), &
(r_2,r_3)&\lra \big[(r_2,1/r_2),(r_3,1/r_3)\big].
\end{aligned}\EE
An element $[(z_2^+,z_2^-),(z_3^+,z_3^-)]$ of
$\wch\cM_{0,3}^{\tau}\!-\!\prt\wch\cM_{0,3}^{\tau}$
belongs to~$\wt{\Ups}'$ if and only if all four points \hbox{$z_i^{\pm}\!\in\!\P^1$} with $i\!=\!2,3$
lie on a great circle through~$z_1^-$ and~$z_1^+$.
The structure of~$\wt{\Ups}'$ along $\prt\wch\cM_{0,3}^{\tau}$ is described by 
the local coordinates of \cite[Remark~3.5]{RealEnum} with $z\!\in\!\R$.
The overlap map between the charts~\eref{M03rel_e9},
$$\R^+\!\times\!\R^*\lra \R^*\!\times\!\R^+, \qquad
(r_2,r_3)\lra \begin{cases} (r_2,r_3),&\hbox{if}~r_3\!\in\!\R^+;\\
(-r_2,-r_3),&\hbox{if}~r_3\!\in\!\R^-;\end{cases}$$
is orientation-preserving with respect to the orientation $\fo_{\R^+}$ on~$\R^+$
and the orientation~$\fo_{\R\P^1}$ on~$\R\P^1$ induced by the standard orientation
of~$[-\i,\i]$ via~\eref{M03rel_e8d}.
We take $\fo_{\wt{\Ups}'}$  to be the orientation on~$\wt{\Ups}'$ 
determined by~$\fo_{\R^+}$ and~$\fo_{\R\P^1}$ via the two charts in~\eref{M03rel_e9}.

The surface~$\wt\Ups'$ contains the four open intervals 
$$\OGa_i^{\pm} \equiv \ov\Ga_i^{\pm}\!-\!\big\{P_i^{\pm+},P_i^{\pm-}\big\}$$ 
with $i\!=\!2,3$;
the closures of these intervals connect the components of the closure of~$\prt\wt{\Ups}'$.
In the two charts~\eref{M03rel_e7}, 
\begin{alignat*}{2}
\OGa_2^+&=\big\{(r_2,z_3)\!\in\!\R^+\!\times\!\P^1\!:z_3\!=\!0\big\},
&\quad 
\OGa_2^-&=\big\{(r_2,z_3)\!\in\!\R^+\!\times\!\P^1\!:z_3\!=\!\i\big\},\\
\OGa_3^+&=\big\{(z_2,r_3)\!\in\!\P^1\!\times\!\R^+\!:z_2\!=\!0\big\},
&\quad 
\OGa_3^-&=\big\{(z_2,r_3)\!\in\!\P^1\!\times\!\R^+\!:z_2\!=\!\i\big\}.
\end{alignat*}
The cut~$\wh{\Ups}'$ of~$\wt{\Ups}'$ along the four open intervals has two components,
$\wh{\Ups}^+$ and~$\wh{\Ups}^-$.
They are distinguished by whether $\CR_{23}^{++}(\cC)$ lies in~$\R^+$ or~$\R^-$
for the elements~$\cC$ of $\wh{\Ups}'\!-\!\prt\wh{\Ups}'$, i.e.~whether 
the points~$z_2^{\pm}$ lie on the same great arc through~$z_1^-$ and~$z_1^+$
as the points~$z_3^{\pm}$ or on the opposite arc.

Let $\Ups'\!=\!\wh{\Ups}^+$ and $\ga'\!=\!\Ups'\!\cap\!\prt\wt{\Ups}'$.
The former is a surface with boundary
$$\prt \Ups'=\OGa_2^+\!\cup\!\OGa_3^+\!\cup\!\OGa_2^-
\!\cup\!\OGa_3^-\!\cup\!\ga'\,.$$
By~\eref{M03rel_e8} and~\eref{M03rel_e9}, 
this surface intersects~$\wch{S}_2$ and~$\wch{S}_3$ transversely along
the closed line segments given~by
\begin{equation*}\begin{split}
\wch\Ga_2&\equiv\!\Ups'\!\cap\!\wch{S}_2
=\big\{(r_2,z_3)\!\in\!\R^+\!\times\!\P^1\!:r_2\!=\!1,\,z_3\!\in\![0,\i]\big\},\\
\wch\Ga_3&\equiv\!\Ups'\!\cap\!\wch{S}_3
=\big\{(z_2,r_3)\!\in\!\P^1\!\times\!\R^+\!:r_3\!=\!1,\,z_2\!\in\![0,\i]\big\}
\end{split}\end{equation*}
in the charts~\eref{M03rel_e7}.
For $i\!=\!2,3$, let $\fo_{\wch\Ga_i}$ be the {\it opposite} of
the orientation on~$\wch\Ga_i$ given by the $r_{5-i}\!\equiv\!|z_{5-i}|$ coordinate.
The restriction 
\BE{M03rel_e15}\ff_{0,3;i}\!:\big(\wch\Ga_i,\fo_{\wch\Ga_i}\big)\lra  
\big(\ov\cM_{0,2}^{\tau},\fo_{0,2}\big)\EE
is then an orientation-preserving diffeomorphism
(because \eref{M03rel_e8a} is orientation-preserving).

\begin{figure}
\begin{center}
\begin{tikzpicture}
\draw [fill] (4,5) circle [radius=0.03];
\draw (4,5) -- (2,5);
\draw [->] (-4,5) -- (2,5);
\draw [fill] (-4,5) circle [radius=0.03];
\draw [dashed] (-4,5) to [out=-50, in=70] (-4,4);
\node at (-4.6,4.5) {$S_{23}^{-+}$};
\draw [fill] (-4,4) circle [radius=0.03];
\draw (-4,4) -- (-4,2);
\draw [->] (-4,0.5) -- (-4,2);
\draw [fill] (-4,0.5) circle [radius=0.03];
\draw [dashed] (-4,0.5) to [out=-50,in=70] (-4,-0.5);
\node at (-4.6,0) {$S_{23}^{--}$};
\draw [fill] (-4,-0.5) circle [radius=0.03];
\draw (-4,-0.5) -- (2,-0.5);
\draw [->] (4,-0.5) -- (2,-0.5); 
\draw [fill] (4,-0.5) circle [radius=0.03];
\draw [dashed] (4,-0.5) to [out=70,in=-50] (4,0.5);
\node at (4.8,0) {$S_{23}^{+-}$};
\draw [fill] (4,0.5) circle [radius=0.03];
\draw (4,0.5) -- (4,2);
\draw [->] (4,4) -- (4,2);
\draw [fill] (4,4) circle [radius=0.03];
\draw [dashed] (4,4) to [out=70,in=-50] (4,5);
\node at (4.8,4.5) {$S_{23}^{++}$};

\draw [fill] (0,5) circle [radius=0.03];
\draw (-1,7) -- (1,5.5) -- (1,-2.5) -- (-1,-1) -- (-1,7);
\draw [fill] (0,-0.5) circle [radius=0.03];
\draw [dashed,->] (0,-0.5) -- (0,2);
\draw [dashed] (0,2) -- (0,5);

\draw [->] (0,5) -- (0,5.5);
\draw [->] (0,5) -- (0.4,4.7);

\node at (-1.9,5.3) {$\Ga_2^+$};
\node at (2,5.3) {$\fo_{\Ga_2^+;3}$};
\node at (-4.3,3.5) {$\Ga_3^-$};
\node at (-4.7,2) {$-\fo_{\Ga_3^-;3}$};
\node at (-2.2,-0.8) {$\Ga_2^-$};
\node at (2,-0.8) {$\fo_{\Ga_2^-;3}$};
\node at (4.4,3.2) {$\Ga_3^+$};
\node at (4.68,2) {$-\fo_{\Ga_3^+;3}$};

\node at (-0.3,1.2) {$\wch\Ga_2$};
\node at (-0.3,5.5) {$\fo_{\wch\Ga_2}$};
\node at (0.5,4.5) {$\fo_{\Ups'}^c$};
\node at (-0.3,4.7) {$P_2^+$};
\node at (-0.3,-0.8) {$P_2^-$};

\node at (-2.5,3) {$\Ups'$};
\node at (0.5,-1.5) {$\wch{S}_2$};

\draw [fill] (4,-3) circle [radius=0.03];
\node [below] at (4,-3) {$\infty$};
\draw [fill] (0,-3) circle [radius=0.03];
\node [below] at (0,-3) {1};
\draw [fill] (-4,-3) circle [radius=0.03];
\node [below] at (-4,-3) {0};
\draw (4,-3) -- (-4,-3);

\node at (-6.5,2.5) {$\wch\cM^\tau_{0,3}$};
\node at (-6.5,-3) {$\ov\cM^\tau_{0,2}$};
\draw [->] (-6.5,1.7) -- (-6.5,-2.2);
\node [left] at (-6.5,0) {$\ff_{0,3;3}$};
\end{tikzpicture}
\end{center}
\caption{The surfaces $\Ups'$ and $\wch{S}_2$ in $\wch\cM^\tau_{0,3}$;
the dotted arcs indicate the four components of $\ga'\!\subset\!\prt\Ups'$.}
\label{UpsS_fig}
\end{figure}

We denote by $\fo_{\Ups'}$ and $\fo_{\Ups'}^c$ the restrictions of the orientation~$\fo_{\wt{\Ups}'}$ 
and the co-orientation~$\fo_{\wt{\Ups}'}^c$ to~$\Ups'$,
by~$\fo_{\ga'}$ the boundary orientation on~$\ga'$ induced by~$\fo_{\Ups'}$,
and by $\fo_{\ga'}^c$ the orientation on $\cN\ga'$ determined by~$\wch\fo_{0,3}$ 
and~$\fo_{\ga'}$.
Thus,
\BE{M03rel_e16} \fo_{\Ups'}=\fo_{\Ups'}^c\wch\fo_{0,3}, \qquad
\fo_{\ga'}=\big(\fo_{\prt\Ups'}^c\fo_{\Ups'}\big)\big|_{\ga'}
=\fo_{\ga'}^c\wch\fo_{0,3}\,.\EE
At the point $P_2^+\!\in\!\ov\Ga_2^+,\wch\Ga_2$, the orientation~$\fo_{\Ga_2^+;3}$ on~$\ov\Ga_2^+$
is the opposite of the orientation given by the $r_2$-coordinate
(because $z_2^+$ moves from $z_1^-\!=\!\i$ to $z_1^+\!=\!0$);
see Figure~\ref{UpsS_fig}.
Since the natural isomorphisms 
$$\big(T_{P_2^+}\wch\Ga_2,\fo_{\wch\Ga_2}\big)\lra 
\big(\cN_{\Ups'}\!\prt\Ups'|_{P_2^+},\fo_{\prt\Ups'}^c\big)
\quad\hbox{and}\quad
\big(T_{P_2^+}\ov\Ga_2^+,\fo_{\Ga_2^+;3}\big)\!\oplus\!\big(T_{P_2^+}\wch\Ga_2,\fo_{\wch\Ga_2}\big)
\lra \big(T_{P_2^+}\Ups',\fo_{\Ups'}\big)$$
are orientation-preserving, 
$$\prt\fo_{\Ups'}|_{P_2^+}\!\equiv\!\big(\fo_{\prt\Ups'}^c\fo_{\Ups'}\big)\big|_{P_2^+}
=\fo_{\Ga_2^+;3}|_{P_2^+}.$$
Since the right-hand side of 
$$\prt\big(\Ups',\fo_{\Ups'}\big)=\big(\OGa_2^+,\fo_{\Ga_2^+;3}\big)\!\cup\!
\big(\OGa_3^+,-\fo_{\Ga_3^+;3}\big)\!\cup\!
\big(\OGa_2^-,\fo_{\Ga_2^-;3}\big)\!\cup\!
\big(\OGa_3^-,-\fo_{\Ga_3^-;3}\big)\!\cup\!\big(\ga',\fo_{\ga'}\big)$$
is an oriented loop and the equality above respects the orientations at~$P_2^+$,
it follows that this equality respects the orientations everywhere.
Combining it with the second equality in~\eref{M03rel_e8b} and 
the first and last equalities in~\eref{M03rel_e16}, 
we obtain~\eref{M03rel_e0} with~$\Ups$ replaced by~$\Ups'$.

We now compute the degree 
\BE{M03rel_e17}  
\deg\!\big(\ff_{0,3;i}|_{\wch\Ga_i},\fo_{\Ups'}^c|_{\wch\Ga_i}\fo_{\nod}^{\R}\big)
\in\Z\EE
of $\ff_{0,3;i}|_{\wch\Ga_i}$ with respect to the co-orientation
$\fo_{\Ups'}^c|_{\wch\Ga_i}$ on~$\wch\Ga_i$ in~$\wch{S}_i$
and the natural orientation~$\fo_{\nod}^{\R}$ of the fibers~of 
$$\ff_{0,3;i}|_{\wch{S}_i}\!\approx\!\ff_{1,2;1}^{\R}\!: 
\wch{S}_i\!\approx\!\wch\cM_{1,2}^{\tau}\lra\ov\cM_{0,2}^{\tau}$$
over~$\cM_{0,2}^{\tau}$ as in the proof of Lemma~\ref{M12rel_lmm}.
By Lemma~\ref{cMorient_lmm}\ref{cMorientC_it},
the orientation~$\fo_{1,2}$ on \hbox{$\wch{S}_i\!\cap\!\ff_{0,3;i}^{-1}(\cM_{0,2}^{\tau})$}
is given by the $z_{5-i}$-coordinate under the corresponding identification in~\eref{M03rel_e8}.
Since the diffeomorphism~\eref{M03rel_e15} is orientation-preserving,
it follows that the vertical orientation~$\fo_{\nod}^{\R}$ on~$\wch{S}_i$ is 
given by the {\it negative} rotation in the $z_{5-i}$-coordinate. 
Since the charts~\eref{M03rel_e7} are orientation-reversing with respect to~$\wch\fo_{0,3}$
and the charts~\eref{M03rel_e9} are orientation-preserving with respect to~$\fo_{\wt{\Ups}'}$,
the orientation~$\fo_{\Ups'}^c$ on $\cN \Ups'$ is given by 
the {\it negative} rotation in the $z_3$-coordinate in the first chart in~\eref{M03rel_e7} and
the {\it positive} rotation in the $z_2$-coordinate in the second chart in~\eref{M03rel_e7}. 
Thus, the projection
$$\big(\!\ker\nd\big\{\ff_{0,3;i}|_{\wch{S}_i}\big\},
\fo_{\nod}^{\R}\big)\big|_{\wch\Ga_i-\{P_i^{\pm}\}} \lra 
\big(\cN \Ups',(-1)^i\fo_{\Ups'}^c\big)\big|_{\wch\Ga_i-\{P_i^{\pm}\}}$$
is an orientation-preserving isomorphism and
the number in~\eref{M03rel_e17} is~$(-1)^i$; see Lemma~\ref{fibrasign_lmm1a}\ref{fibisom_it}.

The surface $\Ups'$ is transverse to $\wch{S}_1\!\cap\!\wch{S}_2$, but passes through $P_0$.
Let $(\Ups'',\fo_{\Ups''}^c)$ be a co-oriented surface in $\wch\cM_{0,3}^{\tau}$
obtained from $(\Ups',\fo_{\Ups'}^c)$ by a small deformation around~$P_0$ so that 
$\Ups''$ is still transverse to  $\wch{S}_1\!\cap\!\wch{S}_2$ and $P_0\!\not\in\!\Ups''$. 
We take $\Ups\!\subset\!\ov\cM_{0,3}^{\tau}$ to be the preimage of~$\Ups''$ under 
the blowdown map~$p$ (which is a diffeomorphism on a neighborhood of~$\Ups$) 
and $\fo_\Ups^c\!=\!p^*\fo_{\Ups''}^c$.
\end{proof}

\begin{rmk}\label{M03rel_rmk1}
We could have taken $\Ups'\!=\!\wh{\Ups}^-$ in the proof of Lemma~\ref{M03rel_lmm}
to avoid~$P_0$, but with $\fo_{\Ups'}\!=\!-\fo_{\wt{\Ups}'}|_{\Ups'}$.
\end{rmk}

\section{Real GW-invariants}
\label{RealGWs_sec}

\subsection{Moduli spaces of stable maps}
\label{MapSpaces_subs}

Let $(X,\om,\phi)$ be a real symplectic manifold and $k,l\!\in\!\Z^{\ge0}$ with $k\!+\!2l\!\ge\!3$.
We denote by $\cH_{k,l}^{\om,\phi}$ the space of pairs $(J,\nu)$ 
consisting of $J\!\in\!\cJ_{\om}^{\phi}$
and a real perturbation~$\nu$ of the $\dbar_J$-equation as in \cite[Section~2]{Penka2}.
For $(J,\nu)\!\in\!\cH_{k,l}^{\om,\phi}$,
a \sf{real genus~0 $(J,\nu)$-map with $k$~real marked points 
and $l$~conjugate pairs of marked points} is a tuple
\BE{udfn_e} \u=\big(u\!:\Si\!\lra\!X,(x_i)_{i\in[k]},(z_i^+,z_i^-)_{i\in[l]},\si\big)\EE
such that 
\BE{cCdfn_e}\cC\equiv \big(\Si,(x_i)_{i\in[k]},(z_i^+,z_i^-)_{i\in[l]},\si\big)\EE
is a real genus~0 nodal curve with complex structure~$\fj$, 
$k$~real marked points, and $l$~conjugate pairs of marked points
and $u$ is a smooth map satisfying
$$u\!\circ\!\si=\phi\!\circ\!u, \qquad
\dbar_Ju|_z\!\equiv\!\frac12\big(\nd_zu\!+\!J\!\circ\!\nd_zu\!\circ\!\fj\big)
=\nu\big(z,u(z)\big)~~\forall~z\!\in\!\Si.$$
Such a map is called \sf{simple} if  the restriction 
of~$u$ to each unstable irreducible component of the domain is simple (i.e.~not multiply covered)
and no two such restrictions have the same image.

For $B\!\in\!H_2(X)$ and $(J,\nu)\!\in\!\cH_{k,l}^{\om,\phi}$,
we denote by $\ov\M_{k,l}(B;J,\nu)$ the moduli space of the equivalence classes
of stable real genus~0 degree~$B$ $(J,\nu)$-maps 
with $k$~real marked points and $l$~conjugate pairs of marked points that
take the fixed locus of the domain to 
the chosen topological component~$X^{\phi}$ of the fixed locus of~$\phi$.
modulo the reparametrizations. 
Let 
$$\ov\M_{k,l}^*(B;J,\nu)\subset\ov\M_{k,l}(B;J,\nu) \quad\hbox{and}\quad 
\M_{k,l}(B;J,\nu)\subset\ov\M_{k,l}^*(B;J,\nu)$$
be the subspace of simple maps and the (\sf{virtually}) \sf{main stratum}, 
i.e.~the subspace consisting of maps as in~\eref{udfn_e} from smooth domains~$\Si$, respectively.

The forgetful morphisms 
$$\ff_{k+1,l;i}^{\R}\!:\ov\cM_{k+1,l}^{\tau}\lra\ov\cM_{k,l}^{\tau},~~i\!\in\![k\!+\!1],
\quad\hbox{and}\quad
\ff_{k,l+1;i}\!:\ov\cM_{k,l+1}^{\tau}\lra\ov\cM_{k,l}^{\tau}, ~~i\!\in\![l\!+\!1],$$
induce maps
$$\ff_{k+1,l;i}^{\R\,*}\!:\cH_{k,l}^{\om,\phi}\!\lra\!\cH_{k+1,l}^{\om,\phi}
\qquad\hbox{and}\qquad
\ff_{k,l+1;i}^{\,*}\!:\cH_{k,l}^{\om,\phi}\!\lra\!\cH_{k,l+1}^{\om,\phi},$$
respectively.
For each $\nu\!\in\!\cH_{k,l}^{\om,\phi}$, we also denote by 
\BE{ffMdfn_e}\begin{split}
\ff_{k+1,l;i}^{\R}\!:\ov\M_{k+1,l}(B;J,\ff_{k+1,l;i}^{\R\,*}\nu)&\lra\ov\M_{k,l}(B;J,\nu),\\
\ff_{k,l+1;i}\!:\ov\M_{k,l+1}(B;J,\ff_{k,l+1;i}^{\,*}\nu)&\lra\ov\M_{k,l}(B;J,\nu)
\end{split}\EE
the forgetful morphisms dropping the $i$-th real marked point and
the $i$-th conjugate pair of marked points, respectively.
The restriction of the second morphism in~\eref{ffMdfn_e} to 
$\M_{k,l+1}(B;J,\ff_{k,l+1;i}^{\,*}\nu)$ is a dense open subset of a $\P^1$-fiber bundle.
We denote by~$\fo_i^+$ the relative orientation of this restriction induced by
the position of the marked point~$z_i^+$.

For $\fc\!\in\!\Z^+$, a (\sf{virtually}) \sf{codimension~$\fc$ stratum}~$\cS$ 
of $\ov\M_{k,l}(B;J,\nu)$ is a subspace of maps from domains~$\Si$ with precisely~$\fc$ 
nodes and thus with $\fc\!+\!1$ irreducible components isomorphic to~$\P^1$. 
It is characterized~by  the distributions~of
\BEnum{$\bu$}

\item the degree~$B$ of the map components~$u$ of its elements~$\u$ 
as in~\eref{udfn_e},

\item the $k$ real marked points, and

\item the $l$ conjugate pairs of marked points 
\EEnum
between the irreducible components of~$\Si$.
There are two types of codimension~1 strata distinguished by whether
the fixed locus~$\Si^{\si}$ of~$(\Si,\si)$ consists of a single point or a wedge of two circles.
These two types are known as \sf{sphere bubbling} and \sf{disk bubbling},
respectively.
If $k$ and $B$ satisfy~\eref{BKcond_e}, as is the case if~\eref{dimcond_e} holds,
then the fixed locus~$\Si^{\si}$ of the domain~$(\Si,\si)$ of every element~\eref{udfn_e} of 
$\ov\M_{k,l}(B;J,\nu)$ is a circle or a tree of two or more circles.
In this case, sphere bubbling does not~occur.

Suppose $l\!\in\!\Z^+$ and $\cS$ is a codimension~1 disk bubbling stratum of 
$\ov\M_{k,l}(B;J,\nu)$.
We~define 
$$K_1(\cS),K_2(\cS)\subset[k], \quad L_1(\cS),L_2(\cS)\subset[l],
\quad k_1(\cS),k_2(\cS),l_1(\cS),l_2(\cS)\in\Z^{\ge0}$$
analogously to $K_r(S),L_r(S),k_r(S),l_r(S)$ in Section~\ref{cMorient_subs}.
We denote by $B_1(\cS)\!\in\!H_2(X)$ the degree of the restriction  
of the map components~$u$ of the elements~$\u$ of~$\cS$ to
the irreducible component~$\P^1_1$ of the domain carrying the marked points~$z_1^{\pm}$
and by $B_2(\cS)\!\in\!H_2(X)$ the degree of the restriction of~$u$ 
to the other irreducible component~$\P^1_2$ of the domain.
Let
$$\ov\cS\subset \ov\M_{k,l}(B;J,\nu)$$
be the \sf{virtual closure} of~$\cS$, i.e.~the subspace of maps~$\u$ as in~\eref{udfn_e}
so that the domain~$\Si$ can be split at a node into two connected (possibly reducible) surfaces,
$\Si_1$ and~$\!\Si_2$, so that
the degree of the restriction of the map component~$u$ of~$\u$ to~$\Si_1$ is~$B_1(\cS)$,
the real marked points $x_i$ with $i\!\in\!K_1(\cS)$ lie on~$\Si_1$,
and so do the conjugate pairs of marked points~$z_i^{\pm}$ with $i\!\in\!L_1(\cS)$.

If in addition $l^*\!\in\![l]$, let 
\begin{gather*}
L_1^*(\cS)=L_1(\cS)\!\cap\![l^*], \quad L_2^*(\cS)=L_2(\cS)\!\cap\![l^*], \quad
l_1^*(\cS)=\big|L_1^*(\cS)\big|, \quad l_2^*(\cS)=\big|L_2^*(\cS)\big|,\\
\ve_{l^*}(\cS)=\blr{c_1(X,\om),B_2(\cS)}
-\big(k_2(\cS)\!+\!2\big(l_2(\cS)\!-\!l_2^*(\cS)\!\big)\big)\,.
\end{gather*} 
In particular,
\begin{gather*}
\ell_{\om}\big(B_1(\cS)\big)\!+\!\ell_{\om}\big(B_2(\cS)\big)=
\ell_{\om}(B)\!-\!1, \quad 
1\le l_1^*(\cS)\le l_1(\cS), \quad l_2^*(\cS)\le l_2(\cS),\\
k_1(\cS)\!+\!k_2(\cS)=k, \quad
l_1(\cS)\!+\!l_2(\cS)=l, \quad l_1^*(\cS)\!+\!l_2^*(\cS)=l^*.
\end{gather*}

We denote by
$$\M_{k,l;l^*}^{\st}(B;J,\nu)\subset\ov\M_{k,l}^*(B;J,\nu)$$
the subspace of simple maps that have no nodes, or
lie in a codimension~1 stratum~$\cS$ with \hbox{$\ve_{l^*}(\cS)\!\cong\!0,1$ mod~4}, or
have only one conjugate pair of nodes.
Let $\wh\M_{k,l;l^*}(B;J,\nu)$ be the space obtained by cutting $\ov\M_{k,l}(B;J,\nu)$ 
along the closures $\ov\cS$ of the codimension~1 strata~$\cS$ with 
\hbox{$\ve_{l^*}(\cS)\!\cong\!2,3$} mod~4.
Thus, $\wh\M_{k,l;l^*}(B;J,\nu)$ contains a double cover of~$\ov\cS$
for each codimension~1 stratum~$\cS$ of $\ov\M_{k,l}(B;J,\nu)$ with $\ve_{l^*}(\cS)\!\cong\!2,3$
mod~4; the union of these covers forms the (virtual) boundary of $\wh\M_{k,l;l^*}(B;J,\nu)$.
Let
\BE{qdfn_e}q\!:\wh\M_{k,l;l^*}(B;J,\nu)\lra \ov\M_{k,l}(B;J,\nu)\EE
be the quotient map.
We denote by
\BE{whMstdfn_e}\wh\M_{k,l;l^*}^{\st}(B;J,\nu)\subset\wh\M_{k,l;l^*}(B;J,\nu)\EE
the subspace of simple maps that 
\BEnum{$\bu$}

\item have no nodes, or

\item have only one real node, or

\item have only one conjugate pair of nodes.

\EEnum
The boundary $\prt\wh\M_{k,l;l^*}^{\st}(B;J,\nu)$ of this subspace consists of 
double covers~$\wh\cS^*$ of the subspaces~$\cS^*$ of simple maps
of the codimension~1 strata~$\cS$ of $\ov\M_{k,l}(B;J,\nu)$
with $\ve_{l^*}(\cS)\!\cong\!2,3$ mod~4.

For each $i\!\in\![k]$, let 
$$\ev_i^{\R}\!:\ov\M_{k,l}(B;J,\nu)\lra X^{\phi},\quad
\ev_i^{\R}\big([u,(x_j)_{j\in[k]},(z_j^+,z_j^-)_{j\in[l]},\si]\big)=u(x_i),$$
be the evaluation morphism for the $i$-th real marked point.
For each $i\!\in\![l]$, let 
$$\ev_i^+\!:\ov\M_{k,l}(B;J,\nu)\lra X,\quad
\ev_i^+\big([u,(x_j)_{j\in[k]},(z_j^+,z_j^-)_{j\in[l]},\si]\big)=u(z_i^+),$$
be the evaluation morphism for the positive point of the $i$-th conjugate pair of marked points.
Let
\BE{fMevdfn_e}\ev\!\equiv\!\prod_{i=1}^k\!\ev_i^{\R}\times\!\prod_{i=1}^l\!\ev_i^+\!:
\ov\M_{k,l}(B;J,\nu)\lra X_{k,l}\!\equiv\!(X^{\phi})^k\!\times\!X^l\EE
be the \sf{total evaluation} map.
We also denote~by
\begin{gather}\notag
\ev_i^{\R}\!:\wh\M_{k,l;l^*}(B;J,\nu)\lra X^{\phi},\qquad
\ev_i^+\!:\wh\M_{k,l;l^*}(B;J,\nu)\lra X, \\
\label{whfMevdfn_e}
\ev\!:\wh\M_{k,l;l^*}(B;J,\nu)\lra X_{k,l}
\end{gather}
the compositions of the evaluation maps above with the quotient map~$q$ in~\eref{qdfn_e}.
We will use the same notation for the compositions of the first three evaluation maps
with all obvious maps to~$\ov\M_{k,l}(B;J,\nu)$.

For $l^*\!\in\![l]$ and a tuple 
\BE{bhdfn_e}\bh\!\equiv\!(h_i\!:H_i\!\lra\!X)_{i\in[l^*]}\EE
of maps, define
\begin{gather}
\label{Mhdfn_e}f_{\bh}\!:M_{\bh}\equiv \prod_{i=1}^{l^*}\!H_i\lra X^{l^*},\qquad
f_{\bh}\big((y_i)_{i\in[l^*]}\big)=\big(h_i(y_i)\!\big)_{i\in[l^*]},\\
\notag
\cZ_{k,l;\bh}^{\st}(B;J,\nu)=\big\{\big(\u,(y_i)_{i\in[l^*]}\big)\!\in\!
\M_{k,l;l^*}^{\st}(B;J,\nu)\!\times\!M_{\bh}\!:
\ev_i^+(\u)\!=\!h_i(y_i)\,\forall\,i\!\in\![l^*]\big\}.
\end{gather}
We denote by
\BE{JakePseudo_e} \ev_{k,l;\bh}\!: \cZ_{k,l;\bh}^{\st}\big(B;J,\nu\big)\lra X_{k,l-l^*}\EE
the map induced by~\eref{fMevdfn_e}.
Orientations on~$H_i$ determine an orientation~$\fo_{\bh}$ on~$M_{\bh}$.
Along with the symplectic orientation~$\fo_{\om}$ of~$X$ and a relative orientation $\fo_{\ev}$ of
\BE{whevdfn_e}\ev\!:\M_{k,l;l^*}^{\st}(B;J,\nu)\lra X_{k,l}\,,\EE
the orientation~$\fo_{\bh}$ determines a relative orientation~$\fo_{\ev}\fo_{\bh}$
of~\eref{JakePseudo_e}.

A dimension~$n$ pseudocycle $h\!:H\!\lra\!X$ in the usual sense
determines an element~$[h]$ of $H_n(X;\Z)$; see~\cite{pseudo}.
If in addition $B$ is a homology class in~$X$ in the complementary dimension, let
$$h\!\cdot_X\!\!B\equiv \blr{\PD_X([h]),B}\in\Z$$
denote the homology intersection product of~$[h]$ with~$B$. 
If $h$ and $B$ are not of complementary dimensions, we set $h\!\cdot_X\!\!B\!=\!0$.
The next two statements follow readily from~\cite{Jake}; see Section~\ref{orient_subs}.

\begin{lmm}\label{orient_lmm}
Suppose $(X,\om,\phi)$ is a real symplectic fourfold, $k,l\!\in\!\Z^{\ge0}$
with \hbox{$k\!+\!2l\!\ge\!3$}, $l^*\!\in\![l]$, $B\!\in\!H_2(X)$, and 
$(J,\nu)\!\in\!\cH_{k,l}^{\om,\phi}$ is generic.
If~$k$ and~$B$ satisfy~\eref{BKcond_e}, then a $\Pin^-$-structure~$\fp$ on~$X^{\phi}$
determines relative orientations~$\fo_{\fp;l^*}$ and~$\wh\fo_{\fp;l^*}$ of the~maps  
\BE{orienlmm_e}\ev\!:\M_{k,l;l^*}^{\st}(B;J,\nu)\lra X_{k,l} \quad\hbox{and}\quad
\ev\!:\wh\M_{k,l;l^*}^{\st}(B;J,\nu)\lra X_{k,l},\EE
respectively, with the following properties:
\BEnum{($\fo_{\fp}\arabic*$)}

\item the restrictions of~$\fo_{\fp;l^*}$ and $\wh\fo_{\fp;l^*}$ to
$\M_{k,l}(B;J,\nu)$ are the same;

\item\label{fforient_it} the restrictions of $\fo_{\fp;l^*+1}\fo_{\om}$ and $\fo_{l^*+1}^+\fo_{\fp;l^*}$ 
to $\M_{k,l+1}(B;J,\ff_{k,l+1;l^*+1}^{\,*}\nu)$ are the same;

\item the interchange of two real points $x_i$ and $x_j$ preserves~$\fo_{\fp;l^*}$;

\item\label{Cijinter2_it} if $\u\!\in\!\M_{k,l}(B;J,\nu;\wch X^\phi)$ and 
the marked points $z_i^+$ and $z_j^+$ are not separated by the fixed locus~$S^1$
of the domain of~$\u$, then
the interchange of the conjugate pairs $(z_i^+,z_i^-)$ and $(z_j^+,z_j^-)$
preserves~$\fo_{\fp;l^*}$  at~$\u$; 

\item the interchange of the points in a conjugate pair $(z_i^+,z_i^-)$
with $l^*\!<\!i\!\le\!l$ preserves~$\fo_{\fp;l^*}$;

\item the interchange of the points in a conjugate pair $(z_i^+,z_i^-)$
with $1^*\!<\!i\!\le\!l^*$ reverses~$\fo_{\fp;l^*}$;

\item\label{orient1pm_it} the interchange of the points in the conjugate pair $(z_1^+,z_1^-)$
reverses~$\fo_{\fp;l^*}$ if and only~if 
$$\ell_{\om}(B)\cong k\!+\!2(l\!-\!l^*) \quad\tn{mod}~4;$$

\item\label{orient0_it} if $k,l,l^*\!=\!1$ and $B\!=\!0$, 
$(\ev_1^{\R}\!\times\!\id_X,\fo_{\fp;l^*}\fo_{\om})$ is a Steenrod pseudocycle of degree~1.

\EEnum

\end{lmm}

\begin{prp}\label{JakePseudo_prp}
Suppose $(X,\om,\phi)$ is a real symplectic fourfold,
$\fp$ is a $\Pin^-$-structure on~$X^{\phi}$,
$l\!\in\!\Z^+$, $l^*\!\in\![l]$, and $B\!\in\!H_2(X)$ are such~that
\BE{JakePseudo_e0} k\!\equiv\!\ell_{\om}(B)-2(l\!-\!l^*)\ge\max(0,3\!-\!2l)\,.\EE
Let $\bh\!\equiv\!(h_i)_{i\in[l^*]}$ be a tuple of pseudocycles of codimension~2
in general position.
For a generic choice of \hbox{$(J,\nu)\!\in\!\cH_{k,l}^{\om,\phi}$},  the map~\eref{JakePseudo_e}
with the relative orientation~$\fo_{\fp;l^*}\fo_{\bh}$
is a codimension~0 Steenrod pseudocycle.
If $B\!\neq\!0$, then
\BE{RdivRel_e}\deg\!\big(\ev_{k,l;\bh},\fo_{\fp;l^*}\fo_{\bh}\big)=
\bigg(\prod_{i=1}^{l^*}\!h_i\!\cdot_X\!\!B\!\!\bigg)N_{B,l-l^*}^{\phi;\fp}\,.\EE
\end{prp}

\subsection{Decomposition formulas}
\label{DecompForm_subs}

Let $(X,\om,\phi)$ be a real symplectic fourfold, $\fp$ be a $\Pin^-$-structure on~$X^{\phi}$,
$k,k',l\!\in\!\Z^{\ge0}$, \hbox{$l',l^*\!\in\![l]$} with 
\BE{klcond_e} k'\le k \qquad\hbox{and}\quad  k'\!+\!2l'\ge3,\EE
$B\!\in\!H_2(X)$, and $(J,\nu)\!\in\!\cH_{k,l}^{\om,\phi}$.
If $k$ and~$B$ satisfy~\eref{BKcond_e}, there is a well-defined \sf{forgetful morphism}
\BE{ffdfn_e}\ff_{k',l'}\!:\ov\M_{k,l}(B;J,\nu)\lra\ov\cM_{k',l'}^{\tau}\EE
which drops the last $k\!-\!k'$ real marked points and the last $l\!-\!l'$ conjugate pairs
from the nodal marked curve~\eref{cCdfn_e} associated with each tuple~$\u$ as in~\eref{udfn_e}
and 
contracts the unstable irreducible components of the resulting curve.
Let $\bh$ as in~\eref{bhdfn_e} be a tuple of smooth maps from oriented manifolds
and 
\BE{pHdfn_e}\bp\!\equiv\!\big((p_i^{\R})_{i\in[k]},(p_i^+)_{i\in[l]-[l^*]}\big)\in
X_{k,l-l^*}\!\equiv\!(X^{\phi})^k\!\times\!(X\!-\!X^{\phi})^{l-l^*}.\EE

Let $\oGa\!\subset\!\ov\cM_{k',l'}^{\tau}$ be a primary codimension~2 stratum
and $\fo_{\Ga}^c$ be its canonical co-orientation as in Lemma~\ref{cNGa2_lmm}.
We denote~by 
$$\M_{\Ga;k,l}(B;J,\nu)\subset \ff_{k',l'}^{-1}(\oGa) \subset\ov\M_{k,l}(B;J,\nu)$$
the subspace consisting of maps from three-component domains.
The domain of every element~$\u$ of $\M_{\Ga;k,l}(B;J,\nu)$ is stable 
and thus $\u$ is automatically a simple~map.
Define
$$\cZ_{\Ga;k,l;\bh}^{\st}(B;J,\nu)=\big\{\big(\u,(y_i)_{i\in[l^*]}\big)\!\in\!
\cZ_{k,l;\bh}^{\st}(B;J,\nu)\!: \u\!\in\!\M_{\Ga;k,l}(B;J,\nu)\big\}.$$

For generic choices of~$(J,\nu)$ and $\bh$, 
$$\cZ_{\Ga;k,l;\bh}^{\st}(B;J,\nu)\subset \cZ_{k,l;\bh}^{\st}(B;J,\nu)$$
is a smooth submanifold of a smooth manifold with the normal bundle canonically isomorphic 
to~$\ff_{k',l'}^*\cN\Ga$.
We denote~by
$$\fo_{\Ga;\fp;\bh}\equiv \big(\ff_{k',l'}^*\fo_{\Ga}^c\big)\big(\fo_{\fp;l^*}\fo_{\bh}\big)$$ 
the relative orientation of the restriction
\BE{evGadfn_e}\ev_{\Ga;\bh}\!: \cZ_{\Ga;k,l;\bh}^{\st}(B;J,\nu)\lra X_{k,l-l^*}\EE
of~\eref{JakePseudo_e} determined by~$\ff_{k',l'}^*\fo_{\Ga}^c$ and 
the relative orientation~$\fo_{\fp;l^*}\fo_{\bh}$ of~\eref{JakePseudo_e},
with $\fo_{\fp;l^*}$ as in Lemma~\ref{orient_lmm};
see Section~\ref{SignConv_subs}.

With $l_0(\Ga),l_{\C}(\Ga),l_0^*(\Ga)\!\in\!\Z^{\ge0}$ 
as at the beginning of Section~\ref{NBstrata_subs}
and $B\!\in\!H_2(X)$, define
\begin{gather*}
\lr{l^*}_{\Ga}=l'\!-\!l^*\!-\!(l_0(\Ga)\!-\!l_0^*(\Ga))\in\Z^{\ge0}\,;\quad
\lr{B;k}_{\Ga}=\begin{cases}1,&\hbox{if}~
k'\!=\!k\!=\!1,\,l_0(\Ga)\!=\!0;\\
0,&\hbox{otherwise}.
\end{cases}\end{gather*}
Let $L_0^*(\Ga)\!\subset\![l^*]$ be the subset indexing the conjugate pairs of marked points
$(z_i^+,z_i^-)$ with $i\!\in\![l^*]$
carried by the real component of the curves in~$\oGa$.
Define
$$L_{\C}^*(\Ga)=[l^*]\!-\!L_0^*(\Ga), \quad
\lr{\bh}_{l^*;\Ga}=\begin{cases}h_i\!\cdot_X\!h_j,&\hbox{if}~
|L_{\C}^*(\Ga)|\!=\!l_{\C}(\Ga),\,L_{\C}^*(\Ga)\!=\!\{i,j\};\\
0,&\hbox{otherwise}.
\end{cases}$$

\begin{prp}\label{Cdecomp_prp}
Suppose $(X,\om,\phi)$ is a real symplectic fourfold,
$\fp$ is a $\Pin^-$-structure on~$X^{\phi}$,
$k',l\!\in\!\Z^{\ge0}$, $l'\!\in\![l]$, $l^*\!\in\![l']$, and $B\!\in\!H_2(X)$ are such~that
\BE{Cdecomp_e0}k\!\equiv\!\ell_{\om}(B)-2(l\!-\!l^*)-2 \ge\max(0,3\!-\!2l)\EE
and \eref{klcond_e} holds.
Let $\oGa\!\subset\!\ov\cM_{k',l'}^{\tau}$ be a primary codimension~2 stratum,
$\bh$ as in~\eref{bhdfn_e} be a tuple of pseudocycles of codimension~2 with 
\hbox{$\phi_*[h_i]\!=\!-[h_i]$} for every $i\!\in\![l^*]$, and $\bp$ be as in~\eref{pHdfn_e}.
If the elements of~$\bh$ and~$\bp$ are in general position and 
\hbox{$(J,\nu)\!\in\!\cH_{k,l}^{\om,\phi}$} is generic, then 
$\bp$ is a regular value of~\eref{evGadfn_e} and the set $\ev_{\Ga;\bh}^{-1}(\bp)$ is finite.
Furthermore,
\BE{Cdecomp_e}\begin{split}
&\big|\ev_{\Ga;\bh}^{-1}(\bp)\big|_{\fo_{\Ga;\fp;\bh}}^{\pm}
=2^{\ell_{\om}(B/2)-1-\lr{l^*}_{\Ga}-l^*}\lr{B;k}_{\Ga}
\bigg(\prod_{i\in[l^*]}\!\!\!h_i\!\cdot_X\!\!B\!\bigg)
\!\!\!
\sum_{\begin{subarray}{c}B'\in H_2(X)\\ \fd(B')=B\end{subarray}}\!\!\!\!\!\!\!N_{B'}^X\\
&\hspace{.5in}
+\lr{\bh}_{l^*;\Ga}\bigg(\prod_{i\in L_0^*(\Ga)}\!\!\!\!\!\!h_i\!\cdot_X\!\!B\!\!\bigg)
N_{B,l-l^*+1}^{\phi;\fp}
+\sum_{\begin{subarray}{c}B_0,B'\in H_2(X)-\{0\}\\ B_0+\fd(B')=B\end{subarray}}
\hspace{-.4in} 2^{\ell_{\om}(B')-\lr{l^*}_{\Ga}}\!
\Bigg(\!\!\big(B_0\!\!\cdot_X\!\!B'\big)\\
&\hspace{1.2in}\times
\bigg(\prod_{i\in L_0^*(\Ga)}\!\!\!\!\!\!h_i\!\cdot_X\!\!B_0\!\!\bigg)
\bigg(\prod_{i\in L_{\C}^*(\Ga)}\!\!\!\!\!\!h_i\!\cdot_X\!\!B'\!\bigg)
\!\!\binom{l\!-\!l'}{\ell_{\om}(B')\!-\!\lr{l^*}_{\Ga}}
N_{B'}^XN_{B_0,l-l^*-\ell_{\om}(B')}^{\phi;\fp}\Bigg).
\end{split}\EE
\end{prp}

\begin{rmk}\label{Cdecomp_rmk}
The domain of \eref{evGadfn_e} can be completed to a Steenrod pseudocycle by adding in
the codimension~1 strata of its closure. 
This implies that the set $\ev_{\Ga;\bh}^{-1}(\bp)$ is finite for a generic choice of~$\bp$
and $|\ev_{\Ga;\bh}^{-1}(\bp)|_{\fo_{\Ga;\fp;\bh}}^{\pm}$  is the degree of this pseudocycle.
The proof of Proposition~\ref{Cdecomp_prp} in Section~\ref{Cdecomp_subs} 
instead identifies $\ev_{\Ga;\bh}^{-1}(\bp)$ with a finite subset of 
the cross product of two moduli spaces with the signed cardinality given 
by the right-hand side of~\eref{Cdecomp_e}.
\end{rmk}

Suppose $\cS$ is an open codimension~1 stratum of $\ov\M_{k,l}(B;J,\nu)$.
Let
$$K_r(\cS)\subset[k],~~L_r(\cS)\subset[l],~~L_r^*(\cS)\subset[l^*],~~
k_r(\cS),l_r(\cS),l_r^*(\cS),\ep_{l^*}(\cS)\in\Z,~~B_r(\cS)\in H_2(X)$$
be as in Section~\ref{MapSpaces_subs} and $\cS^*\!\subset\!\cS$ be the subspace of simple maps.
With $M_{\bh}$ given by~\eref{Mhdfn_e}, define 
$$\cS_{\bh}^*=
\big\{\big(\u,(y_i)_{i\in[l^*]}\big)\!\in\!\cS^*\!\times\!M_{\bh}\!:
\ev_i^+(\u)\!=\!h_i(y_i)\,\forall\,i\!\in\![l^*]\big\}.$$
The (virtual) normal bundles $\cN\cS$ of $\cS$ in $\ov\M_{k,l}(B;J,\nu)$ and  
$\cN\cS_{\bh}^*$ of $\cS_{\bh}^*$ in
$$\ov\cZ_{k,l;\bh}(B;J,\nu)\equiv\big\{\big(\u,(y_i)_{i\in[l^*]}\big)
\!\in\!\ov\M_{k,l}(B;J,\nu)\!\times\!M_{\bh}\!:
\ev_i^+(\u)\!=\!h_i(y_i)\,\forall\,i\!\in\![l^*]\big\}$$
are canonically isomorphic.

If $\u\!\in\!\cS_{\bh}^*$, an orientation~$\fo_{\cS;\u}^c$ of~$\cN_{\u}\cS$ 
determines a direction of degeneration of elements of 
the main stratum of $\cZ_{k,l;\bh}^{\st}(B;J,\nu)$ to~$\u$.
The relative orientation~$\fo_{\fp;l^*}\fo_{\bh}$ of~\eref{JakePseudo_e}
limits to a relative orientation $\fo_{\fp;\bh;\u}$ of
\BE{JakePseudo_e2}\ev_{k,l;\bh}\!:\ov\cZ_{k,l;\bh}(B;J,\nu)\lra X_{k,l-l^*}\EE
at~$\u$ obtained by approaching~$\u$ from this direction.
Along with~$\fo_{\cS;\u}^c$,  $\fo_{\fp;\bh;\u}$ determines a relative 
orientation $\prt_{\fo_{\cS;\u}^c}\fo_{\fp;\bh;\u}$ of the restriction
\BE{evcSdfn_e}\ev_{\cS;\bh}\!:\cS_{\bh}^*\lra X_{k,l-l^*}\EE
of $\ev_{k,l;\bh}$ via the first isomorphism in~\eref{lasplits_e}.

\begin{lmm}\label{orient_lmm2}
Suppose $(X,\om,\phi)$, $\fp,k,l,l^*,B$, and $(J,\nu)$ are as in Lemma~\ref{orient_lmm}
and $\bh$ as in~\eref{bhdfn_e} is a generic tuple of smooth maps from oriented manifolds. 
If~$k$ and~$B$ satisfy~\eref{BKcond_e}, $\cS$ is an open codimension~1 stratum of 
$\ov\M_{k,l}(B;J,\nu)$, and $\u\!\in\!\cS_{\bh}^*$, then
the relative orientation $\prt_{\fo_{\cS;\u}^c}\fo_{\fp;\bh;\u}$ of~\eref{evcSdfn_e} at~$\u$
does not depend on the choice of~$\fo_{\cS;\u}^c$ if and only if $\ep_{l^*}(\cS)\!\cong\!2,3$ 
mod~4. 
\end{lmm}

The relative orientation~$\fo_{\fp;l^*}\fo_{\bh}$ of the restriction of~\eref{JakePseudo_e2} to
$$\M_{k,l;\bh}(B;J,\nu)\equiv\big\{\big(\u',(y_i)_{i\in[l^*]}\big)
\!\in\!\M_{k,l}(B;J,\nu)\!\times\!M_{\bh}\!:
\ev_i^+(\u')\!=\!h_i(y_i)\,\forall\,i\!\in\![l^*]\big\}$$
extends across $\cS_{\bh}^*$ if and only if
$\prt_{\fo_{\cS;\u}^c}\fo_{\fp;\bh;\u}$ {\it depends}  on the choice of~$\fo_{\cS;\u}^c$ 
for every $\u\!\in\!\cS_{\bh}^*$.
In particular, the first statement of Lemma~\ref{orient_lmm} is an immediate consequence
of Lemma~\ref{orient_lmm2}.
If $\ov\M_{k,l}(B;J,\nu)$ is cut along~$\ov\cS$ and
$\wh\cS^*$ is the double cover of~$\cS^*$ in the~cut, then 
$\prt_{\fo_{\cS;\u}^c}\fo_{\fp;\bh;\u}$ is the boundary relative orientation 
induced by~$\fo_{\fp;l^*}\fo_{\bh}$ at one of the copies~$\wh\u$ of~$\u$ in 
\BE{pHSinc_e2wh}
\wh\cS_{\bh}^*=
\big\{\big(\wh\u',(y_i)_{i\in[l^*]}\big)\!\in\!\wh\cS^*\!\times\!M_{\bh}\!:
\ev_i^+(\wh\u')\!=\!h_i(y_i)\,\forall\,i\!\in\![l^*]\big\};\EE
we then denote it by $\prt\fo_{\fp;\bh;\wh\u}$. 
If $\ep_{l^*}(\cS)\!\cong\!2,3$ mod~4, we abbreviate $\prt_{\fo_{\cS;\u}^c}\fo_{\fp;\bh;\u}$
as~$\prt\fo_{\fp;\bh;\u}$.

\begin{rmk}\label{orient2_rmk}
While Lemma~\ref{orient_lmm2} follows readily from \cite[Prop.~5.3]{Jake},
it is also immediately implied by our Lemmas~\ref{DMboundary_lmm} and~\ref{DorientComp_lmm} 
(which are also needed to establish Proposition~\ref{Rdecomp_prp} below).
The terms $w_2(\psi(d''))$ in \cite[(17),(18)]{Jake} appear to be extra 
(they are omitted in the key invariance argument on \cite[p53]{Jake}).
The first equation in~\cite{Jake} with 
$$\mu(d'')=\ell_{\om}\big(B_2(\cS)\big), \quad k''=k_2(\cS), \quad
l''=l_2(\cS)\!-\!l_2^*(\cS)$$
compares the two possibilities for $\prt_{\fo_{\cS;\u}^c}\fo_{\fp;\bh;\u}$ 
when $k\!=\!0$ or the real marked point~$x_1$ lies on the same component of~$\u$
as the marked points~$z_1^+$; the second equation treats the remaining case.
The right-hand side of the latter reduces to the right-hand side of the former
if \eref{BKcond_e} holds.
The right-hand side of \cite[(17)]{Jake}, without the $w_2(\psi(d''))$ term,
in turn reduces to 
$$\frac{(\mu(d'')\!-\!k''\!-\!2l'')(\mu(d'')\!-\!k''\!-\!2l''\!-\!1)}{2}\!+\!1
\cong\frac{\ep_{l^*}(\cS)(\ep_{l^*}(\cS)\!-\!1)}{2}\!+\!1 \mod2;$$
the last expression vanishes (i.e.~the two orientations are the same)
if and only if $\ep_{l^*}(\cS)\!\cong\!2,3$.
\end{rmk}

The stratum $\cS$ satisfies exactly one of the following conditions:
\BEnum{($\cS\arabic*$)}

\setcounter{enumi}{-1}

\item\label{cS0_it} $K_2(\cS)\!\cap\![k']\!=\!\eset$ and $L_2(\cS)\!\cap\![l']\!=\!\eset$;

\item\label{cS1_it} $|K_2(\cS)\!\cap\![k']|\!=\!1$ and $L_2(\cS)\!\cap\![l']\!=\!\eset$;

\item\label{cS2_it} there exists a codimension~1 stratum $S\!\subset\!\ov\cM_{k',l'}^{\tau}$
such that $\ff_{k',l'}(\cS)\!\subset\!S$.

\EEnum
We call a pair $(\cS,\Ups)$ consisting of $\cS$ as above and 
a (possibly bordered) hypersurface $\Ups\!\subset\!\ov\cM_{k,l}^{\tau}$ \sf{admissible}
if one of the following conditions holds:
\BEnum{($\cS\arabic*\Ups$)}

\item\label{cSUps_it1} $K_2(\cS)\!\cap\![k']\!=\!\{i\}$, $L_2(\cS)\!\cap\![l']\!=\!\eset$,
and $\Ups$ is regular with respect to~$\ff_{k',l';i}^{\R}$;  

\item\label{cSUps_it2} there exists a codimension~1 stratum 
$S\!\subset\!\ov\cM_{k',l'}^{\tau}$
such that $\ff_{k',l'}(\cS)\!\subset\!S$ and $\Ups$ is regular with respect to~$S$.

\EEnum
The notions of $\Ups$ being \sf{regular} with respect to $\ff_{k',l';i}^{\R}$ and~$S$
are defined in Section~\ref{cMorient_subs}. 

For $\Ups\!\subset\!\ov\cM_{k',l'}^{\tau}$, define
\begin{gather}\label{pHSinc_e}
f_{\bp;\Ups}\!: \Ups\lra X_{k,l-l^*}\!\times\!\ov\cM_{k',l'}^{\tau}, \quad
f_{\bp;\Ups}(P)=(\bp,P),\\
\label{pHSinc_e2}
\cS_{\bh,\bp;\Ups}^*=\big\{\big(\u,P\big)\!\in\!\cS_{\bh}^*\!\times\!\Ups\!:
\ev_{\cS;\bh}(\u)\!=\!\bp,\,\ff_{k',l'}(\u)=P\big\}.
\end{gather}
If $(\cS,\Ups)$ is an admissible pair and $\fo_\Ups^c$ is a co-orientation on~$\Ups$, 
we denote by $\deg(\cS,\fo_\Ups^c)\!\in\!\Z$
the corresponding degree $\deg_i^{\R}(\Ups,\fo_\Ups^c)$ or $\deg_S(\Ups,\fo_\Ups^c)$
defined in Section~\ref{cMorient_subs}.

\begin{prp}\label{Rdecomp_prp}
Suppose $(X,\om,\phi)$, $\fp$, $k,k',l,l^*,l',B,\bp,\bh$ are in as
Proposition~\ref{Cdecomp_prp}, 
$$\cS\subset \ov\M_{k,l}(B;J,\nu) \qquad\hbox{and}\qquad \Ups\subset\ov\cM_{k',l'}^{\tau}$$
form an admissible pair, and $\fo_\Ups^c$ is a co-orientation on~$\Ups$.
If $\ve_{l^*}(\cS)\!=\!2$,
the elements of~$\bh$ and~$\bp$ are in general position, and 
\hbox{$(J,\nu)\!\in\!\cH_{k,l}^{\om,\phi}$} is generic, then 
$$\big(\ev_{\cS;\bh},\ff_{k',l'}\big)\!:\cS_{\bh}^*\lra X_{k,l-l^*}\!\times\!\ov\cM^\tau_{k',l'}
\quad\hbox{and}\quad
f_{\bp;\Ups}\!:\Ups \lra X_{k,l-l^*}\!\times\!\ov\cM^\tau_{k',l'}$$
are transverse maps from manifolds 
of complementary dimensions and the set~$\cS_{\bh,\bp;\Ups}^*$ is finite.
Furthermore,
\BE{Rdecomp_e}\begin{split}
&\big|\cS_{\bh,\bp;\Ups}^*\big|_{\prt\fo_{\fp;\bh},\fo_\Ups^c}^{\pm}
=-(-1)^{\dim\,\Ups}\!\deg(\cS,\fo_\Ups^c)
\bigg(\prod_{i\in L_1^*(\cS)}\!\!\!\!\!\!h_i\!\cdot_X\!\!B_1(\cS)\!\!\bigg)
\!\bigg(\prod_{i\in L_2^*(\cS)}\!\!\!\!\!\!h_i\!\cdot_X\!\!B_2(\cS)\!\!\bigg)\\
&\hspace{2.8in}
\times N_{B_1(\cS),l_1(\cS)-l_1^*(\cS)}^{\phi;\fp}N_{B_2(\cS),l_2(\cS)-l_2^*(\cS)}^{\phi;\fp}.
\end{split}\EE
\end{prp}

Due to the condition $\ve_{l^*}(\cS)\!=\!2$, 
$$k_1(\cS)=\ell_{\om}\big(B_1(\cS)\!\big)\!-\!2\big(l_1(\cS)\!-\!l_1^*(\cS)\!\big) 
\quad\hbox{and}\quad
k_2(\cS)\!+\!1=\ell_{\om}\big(B_2(\cS)\!\big)\!-\!2\big(l_2(\cS)\!-\!l_2^*(\cS)\!\big),$$
i.e.~the second irreducible component of the maps in~$\cS$ passes through an extra real point.

\begin{rmk}\label{Rdecomp_rmk}
The crucial property of $(\cS,\Ups)$ used in the proof is that the condition
\hbox{$f_{k',l'}(\u)\!\in\!\Ups$} in~\eref{pHSinc_e2} factors through
\begin{equation*}\begin{split}
\cS&\lra \ov\M_{k_1(\cS)+1,l_1(\cS)}\big(B_1(\cS);J,\nu_1\big)\!\times\!
\ov\M_{k_2(\cS)+1,l_2(\cS)}\big(B_2(\cS);J,\nu_2\big)\\
&\lra \ov\M_{k_1(\cS)+1,l_1(\cS)}\big(B_1(\cS);J,\nu_1\big)
\end{split}\end{equation*}
for a good choice of~$\nu$.
Thus, Lemma~\ref{dropfactor_lmm3} applies.
\end{rmk}

\subsection{Proof of Theorem~\ref{SolWDVV_thm}}
\label{SolWDVVpf_subs}

Fix $(X,\om,\phi)$, $\fp$, and $B$ as in Theorem~\ref{SolWDVV_thm},
take~$k$ as in~\eref{dimcond_e} and $l^*\!\in\!\{2,3\}$, and
choose generic tuples
$$\bh\!\equiv\!\big(h_i\!:H_i\!\lra\!X\big)_{i\in[l^*]} \qquad\hbox{and}\qquad
\bp\!\equiv\!\big((p_i^{\R})_{i\in[k]},(p_i^+)_{i\in[l+l^*-1]-[l^*]}\big)\in
X_{k,l-1}$$
so that each $h_i$ is a codimension~2 pseudocycle, 
\BE{h1h2inter_e} \ell_{\om}(B)=2l\!+\!k,  \qquad\hbox{and}\qquad
h_1\!\cdot_X\!h_3=0.\EE
We deduce the three relations of Theorem~\ref{SolWDVV_thm}, with 
$N_{B,l}^{\phi;\fp}(X^{\phi})$ replaced by $N_{B,l}^{\phi;\fp}$ and
the left-hand sides multiplied by $h_1\!\cdot_X\!h_2$,
from Propositions~\ref{Cdecomp_prp} and~\ref{Rdecomp_prp} and several lemmas stated so~far.

For $(k',l')\!=\!(1,2),(0,3)$ such that $k'\!\le\!k$ and $l'\!\le\!l\!+\!l^*\!-\!1$,
we denote~by
$$\ff_{k',l'}\!:\wh\M_{k,l+l^*-1;l^*}(B;J,\nu)\lra\ov\cM_{k',l'}^{\tau}$$
the composition of~\eref{ffdfn_e} and the quotient map~$q$ in~\eref{qdfn_e}
with $l$ replaced by $l\!+\!l^*\!-\!1$.
For a stratum~$\cS$ of $\ov\M_{k,l+l^*-1}(B;J,\nu)$, let
$$\wh\cS^*=q^{-1}(\cS^*)\subset\wh\M_{k,l+l^*-1;l^*}(B;J,\nu).$$
With the notation as in~\eref{whMstdfn_e}, let
\begin{gather*}
M_{\bh}=H_1\!\times\!\ldots\!\times\!H_{l^*}, \\
\wh\cZ_{k,l+l^*-1;\bh}^{\st}(B;J,\nu)=\big\{(\u,y_1,\ldots,y_{l^*})\!\in\!
\wh\M_{k,l+l^*-1;l^*}^{\st}(B;J,\nu)\!\times\!M_{\bh}\!:
\ev_i^+(\u)\!=\!h_i(y_i)~\forall\,i\!\in\![l^*]\big\}.
\end{gather*}
For $(J,\nu)\!\in\!\cH_{k,l+l^*-1}^{\om,\phi}$ generic,
the relative orientation $\wh\fo_{\fp;l^*}$ of Lemma~\ref{orient_lmm} and 
the orientation~$\fo_{\bh}$ of~$M_{\bh}$ 
determine a relative orientation~$\wh\fo_{\fp;\bh}$
of the~map
$$\ev_{k,l+l^*-1;\bh}\!: \wh\cZ_{k,l+l^*-1;\bh}^{\st}\big(B;J,\nu\big)\lra X_{k,l-1}$$
induced by~\eref{fMevdfn_e} with $l$ replaced by $l\!+\!l^*\!-\!1$.

Below we take $\Ups\!\subset\!\ov\cM_{k',l'}^{\tau}$ to be the bordered compact hypersurfaces
of Lemmas~\ref{M12rel_lmm} and~\ref{M03rel_lmm} with their co-orientations~$\fo_\Ups^c$.
For a stratum~$\cS$ of $\ov\M_{k,l+l^*-1}(B;J,\nu)$, let 
$$\cS_{\bh,\bp;\Ups}^*\subset \cS^*\!\times\!M_{\bh}\!\times\!\Ups
\quad\hbox{and}\quad
\wh\cS_{\bh}^*\equiv  
\wh\cZ_{k,l+l^*-1;\bh}^{\st}\big(B;J,\nu\big)\!\cap\!\big(\wh\cS^*\!\times\!M_{\bh}\big)$$
be as in~\eref{pHSinc_e2} and~\eref{pHSinc_e2wh}, respectively, and
$$\wh\cS_{\bh,\bp;\Ups}^*=
\big\{\big(\wh\u,P\big)\!\in\!\wh\cS_{\bh}^*\!\times\!\Ups\!:
\ev_{k,l+l^*-1;\bh}(\wh\u)\!=\!\bp,\,\ff_{k',l'}(\wh\u)\!=\!P\big\}\,.$$
We establish the next two statements at the end of this section.

\begin{lmm}\label{mainsetup_lmm}
For a generic choice of $(J,\nu)\!\in\!\cH_{k,l+l^*-1}^{\om,\phi}$, the~map
\BE{evfkl_e} \big(\ev_{k,l+l^*-1;\bh},\ff_{k',l'}\big)\!: 
\wh\cZ_{k,l+l^*-1;\bh}^{\st}(B;J,\nu)\lra X_{k,l-1}\!\times\!\ov\cM_{k',l'}^{\tau}\EE
is a bordered $\Z_2$-pseudocycle of dimension $4l\!+\!2k\!-\!2$
transverse to~\eref{pHSinc_e}.
\end{lmm}

\begin{crl}\label{mainsetup_crl}
For a generic choice of $(J,\nu)\!\in\!\cH_{k,l+l^*-1}^{\om,\phi}$, 
$$M_{(\ev_{k,l+l^*-1;\bh},\ff_{k',l'}),f_{\bp;\Ups}}
\subset \wh\cZ_{k,l+l^*-1;\bh}^{\st}(B;J,\nu)\!\!\times\!\Ups$$ 
is a compact one-dimensional manifold with boundary~\eref{bndsplit_e0}
and 
\BE{mainsetup_e2}
\big(\prt\wh\cZ_{k,l+l^*-1;\bh}^{\st}(B;J,\nu)\!\big)
_{(\ev_{k,l+l^*-1;\bh},\ff_{k',l'})}\!\!\times\!_{f_{\bp;\Ups}}\Ups
=\bigsqcup_{\ep_{l^*}(\cS)=2}\!\!\!\!\!\wh\cS_{\bh,\bp;\Ups}^*,\EE
with the union taken over the codimension~1 strata $\cS$ of $\ov\M_{k,l+l^*-1}(B;J,\nu)$
that satisfy either~\ref{cS1_it} or~\ref{cS2_it} above Proposition~\ref{Rdecomp_prp}.
\end{crl}

In our case, $\Ups\!\cap\!\ov{S}_1\!=\!\eset$.
If $\cS_{\bh,\bp;\Ups}^*\!\neq\!\eset$ and $\cS$ satisfies~\ref{cS2_it}, then $S\!\!\neq\!S_1$.
Combined with the assumption that $(k',l')$ is either~$(1,2)$ or~$(0,3)$,
this implies that the pair $(\cS,\Ups)$ is admissible in the sense defined
above Proposition~\ref{Rdecomp_prp} whenever~$\cS$ contributes to
the right-hand side of~\eref{mainsetup_e2}. 

The identity~\eref{bndsplit_e} follows from Lemma~\ref{InterOrient_lmm}.
We use Corollary~\ref{mainsetup_crl} and Proposition~\ref{Rdecomp_prp} 
to express the right-hand side of this identity, 
i.e.~the signed cardinality of~\eref{mainsetup_e2}, in terms of the real invariants~$N_{B',l'}^{\phi,\fp}$.
We use Lemma~\ref{dropfactor_lmm}\ref{dropfactor_it1} and Proposition~\ref{Cdecomp_prp} 
to express the left-hand side of~\eref{bndsplit_e} in terms of 
the real invariants~$N_{B',l'}^{\phi,\fp}$ and the complex invariants~$N_{B'}^X$.
Setting the two expressions equal and dividing by~2,
we obtain the two identities of Theorem~\ref{SolWDVV_thm}.

\begin{proof}[{\bf{\emph{Proof of~\ref{Sol12rec_it}}}}]
We take $l^*\!=\!2$.
Since $k,l\!\ge\!1$ in this case, the morphism
$$\ff_{1,2}\!:  \wh\cZ_{k,l+1;\bh}^{\st}(B;J,\nu)\lra \ov\cM_{1,2}^{\tau}$$
is well-defined.
Let $P^{\pm}\!\in\!\ov\cM_{1,2}^{\tau}$, $\Ups\!\subset\!\ov\cM_{1,2}^{\tau}$,
$\fo_{P^{\pm}}^c$, and $\fo_\Ups^c$ be as in the statement of Lemma~\ref{M12rel_lmm}.
Let $\cA_1^{\R}$ (resp.~$\cA_2$) be the collection of the codimension~1 strata~$\cS$
of $\ov\M_{k,l+1}(B;J,\nu)$ with $\ep_2(\cS)\!=\!2$
such~that the irreducible component~$\P^1_1$ 
carrying $(z_1^+,z_1^-)$ (resp.~the other component~$\P^1_2$)
of the maps in~$\cS$ carries the conjugate pair $(z_2^+,z_2^-)$, but 
not the real marked point~$x_1$.
Each such stratum is doubly covered by a stratum~$\wh\cS$ of $\prt\wh\M_{k,l+1;2}^{\st}(B;J,\nu)$.

By Corollary~\ref{mainsetup_crl} and Proposition~\ref{Rdecomp_prp},
{\it half} of the right-hand side of~\eref{bndsplit_e},
{\it not} including the sign in~front, equals
\begin{equation*}\begin{split}
\sum_{\cS\in\cA_1^{\R}}\!\!\!
\big|\cS_{\bh,\bp;\Ups}^*\big|_{\prt\fo_{\fp;\bh},\fo_\Ups^c}^{\pm}
+\!\!\sum_{\cS\in\cA_2}\!\!\!\big|\cS_{\bh,\bp;\Ups}^*\big|_{\prt\fo_{\fp;\bh},\fo_\Ups^c}^{\pm}
&=\!\!\sum_{\cS\in\cA_1^{\R}}\!\!\!
\big(h_1\!\cdot_X\!B_1(\cS)\!\big)\!\big(h_2\!\cdot_X\!B_1(\cS)\!\big)
N_{B_1(\cS),l_1(\cS)-2}^{\phi;\fp}N_{B_2(\cS),l_2(\cS)}^{\phi;\fp}\\
&-\!\!\!\sum_{\cS\in\cA_2}\!\!\!\big(h_1\!\cdot_X\!B_1(\cS)\!\big)\!\big(h_2\!\cdot_X\!B_2(\cS)\!\big)
N_{B_1(\cS),l_1(\cS)-1}^{\phi;\fp}N_{B_2(\cS),l_2(\cS)-1}^{\phi;\fp}.
\end{split}\end{equation*}
Summing over all splittings of $B\!\in\!H_2(X)$ into~$B_1$ and~$B_2$,
of $l\!-\!1$ conjugate pairs of points into sets of cardinalities~$l_1$ and~$l_2$,
and of $k\!-\!1$ real points into sets of cardinalities $\ell_{\om}(B_i)\!-\!2l_i$, 
we obtain
\BE{Sol12rec_e5}\begin{split}
&\frac12\big|\big(\prt\wh\cZ_{k,l+1;\bh}^{\st}(B;J)\!\big)\,_{(\ev_{k,l+1;\bh},\ff_{k',l'})}
\!\!\times_{f_{\bp;\Ups}}\Ups\big|_{\prt\wh\fo_{\fp;\bh},\fo_\Ups^c}^{\pm}\\
&\hspace{.7in}
=\sum_{\begin{subarray}{c}B_1,B_2\in H_2(X)-\{0\}\\ B_1+B_2=B\\  
l_1+l_2=l-1,\,l_1,l_2\ge0\end{subarray}} \hspace{-.4in}
\big(h_1\!\cdot_X\!B_1\!\big)\!\big(h_2\!\cdot_X\!B_1\!\big)\binom{l\!-\!1}{l_1}
\!\binom{\ell_{\om}(B)\!-\!2l\!-\!1}{\ell_{\om}(B_1)\!-\!2l_1}
N_{B_1,l_1}^{\phi;\fp}N_{B_2,l_2}^{\phi;\fp}\\
&\hspace{1in}-\!\!\!\!\!\sum_{\begin{subarray}{c}B_1,B_2\in H_2(X)-\{0\}\\ B_1+B_2=B\\  
l_1+l_2=l-1,\,l_1,l_2\ge0\end{subarray}} \hspace{-.4in}
\big(h_1\!\cdot_X\!B_1\!\big)\!\big(h_2\!\cdot_X\!B_2\!\big)\binom{l\!-\!1}{l_1}
\!\binom{\ell_{\om}(B)\!-\!2l\!-\!1}{\ell_{\om}(B_1)\!-\!2l_1\!-\!1}
N_{B_1,l_1}^{\phi;\fp}N_{B_2,l_2}^{\phi;\fp}\,.
\end{split}\EE
We note that $l_1\!\equiv\!l_1(\cS)\!-\!2$ in the $\cA_1^{\R}$ sum above and 
$l_1\!\equiv\!l_1(\cS)\!-\!1$ in the $\cA_2$ sum,
because the subtractions from~$l_1(\cS)$ correspond to the insertions of
the divisors~$H_1,H_2$; the meaning of~$l_2$ is analogous.

By Lemma~\ref{dropfactor_lmm}\ref{dropfactor_it1} and Proposition~\ref{Cdecomp_prp}, 
{\it half} of the left-hand side of~\eref{bndsplit_e} equals
\begin{equation*}\begin{split}
\big|\ev_{P^+;\bh}^{-1}(\bp)\big|_{\fo_{P^+;\fp;\bh}}^{\pm}
&=\big(h_1\!\cdot_X\!\!h_2\big)N_{B,l}^{\phi;\fp}+2^{\ell_{\om}(B/2)-3}\lr{B}_l
\big(h_1\!\cdot_X\!B\big)\!\big(h_2\!\cdot_X\!B\big)
\!\!\!\!\!
\sum_{\begin{subarray}{c}B'\in H_2(X)\\ \fd(B')=B\end{subarray}}\!\!\!\!\!\!\!\!N_{B'}^X\\
&~~~+\!\!\!\!
\sum_{\begin{subarray}{c}B_0,B'\in H_2(X)-\{0\}\\ B_0+\fd(B')=B\end{subarray}}
\hspace{-.4in} 2^{\ell_{\om}(B')}\!\big(B_0\!\!\cdot_X\!\!B'\big)
\big(h_1\!\cdot_X\!B'\big)\!\big(h_2\!\cdot_X\!B'\big)
\binom{l\!-\!1}{\ell_{\om}(B')}N_{B'}^XN_{B_0,l-1-\ell_{\om}(B')}^{\phi;\fp}.
\end{split}\end{equation*}
Equating this expression with the negative of~\eref{Sol12rec_e5}, 
as dictated by~\eref{bndsplit_e},
we obtain the first identity in Theorem~\ref{SolWDVV_thm}
with the left-hand side multiplied by~$h_1\!\cdot_X\!h_2$.
\end{proof}

\begin{proof}[{\bf{\emph{Proof of~\ref{Sol03rec_it}}}}]
We again take $l^*\!=\!2$.
Since $l\!\ge\!2$ in this case, the morphism
$$\ff_{0,3}\!: \wh\cZ_{k,l+1;\bh}^{\st}(B;J,\nu)\lra \ov\cM_{0,3}^{\tau}$$
is well-defined.
Let $\Ga^{\pm}_i,\Ups,\ga'\!\subset\!\ov\cM_{0,3}^{\tau}$,
$\fo_{\Ga_i^{\pm}}^c$, $\fo_\Ups^c$, and~$\fo_{\ga'}^c$ 
be as in the statement of Lemma~\ref{M03rel_lmm}.
Let $\cA_2$ (resp.~$\cA_3$) be the collection of the codimension~1 strata~$\cS$
of $\ov\M_{k,l+1}(B;J,\nu)$ with $\ep_2(\cS)\!=\!2$
such~that $\P^1_1$ (resp.~$\P^1_2$)
carries the conjugate pair $(z_3^+,z_3^-)$, but not $(z_2^+,z_2^-)$.

By Corollary~\ref{mainsetup_crl} and Proposition~\ref{Rdecomp_prp},
{\it half} of the right-hand side of~\eref{bndsplit_e} equals
\begin{equation*}\begin{split}
\sum_{\cS\in\cA_2}\!\!\!
\big|\cS_{\bh,\bp;\Ups}^*\big|_{\prt\fo_{\fp;\bh},\fo_\Ups^c}^{\pm}
+\!\!\sum_{\cS\in\cA_3}\!\!\!
\big|\cS_{\bh,\bp;\Ups}^*\big|_{\prt\fo_{\fp;\bh},\fo_\Ups^c}^{\pm}
&
=\sum_{\cS\in\cA_3}\!\!\!\big(h_1\!\cdot_X\!B_1(\cS)\!\big)\!\big(h_2\!\cdot_X\!B_1(\cS)\!\big)
N_{B_1(\cS),l_1(\cS)-2}^{\phi;\fp}N_{B_2(\cS),l_2(\cS)}^{\phi;\fp}\\
&
-\!\!\sum_{\cS\in\cA_2}\!\!\!\big(h_1\!\cdot_X\!B_1(\cS)\!\big)\!\big(h_2\!\cdot_X\!B_2(\cS)\!\big)
N_{B_1(\cS),l_1(\cS)-1}^{\phi;\fp}N_{B_2(\cS),l_2(\cS)-1}^{\phi;\fp}.
\end{split}\end{equation*}
Summing over all splittings of $B\!\in\!H_2(X)$ into~$B_1$  and~$B_2$,
of $l\!-\!2$ conjugate pairs of points into sets of cardinalities~$l_1$ and~$l_2$,
and of $k$ real points into sets of the appropriate cardinalities, 
we obtain
\BE{Sol03rec_e5}\begin{split}
&\frac12\big|\big(\prt\wh\cZ_{k,l+1;\bh}^{\st}(B;J)\!\big)_{(\ev_{k,l+1;\bh},\ff_{k',l'})}
\!\times_{f_{\bp;\Ups}}\!\Ups\big|_{\prt\wh\fo_{\fp;\bh},\fo_\Ups^c}^{\pm}\\
&\hspace{.2in}=\!\sum_{\begin{subarray}{c} B_1,B_2\in H_2(X)-\{0\}\\ B_1+B_2=B\\
l_1+l_2=l-2,\,l_1,l_2\ge0\end{subarray}} \hspace{-.42in}
\big(h_1\!\cdot_X\!B_1\big)\!\big(h_2\!\cdot_X\!B_1\big)
\binom{l\!-\!2}{l_1}
\binom{\ell_{\om}(B)\!-\!2l}{\ell_{\om}(B_1)\!-\!2l_1}
N_{B_1,l_1}^{\phi;\fp}N_{B_2,l_2+1}^{\phi;\fp}\\
&\hspace{.4in}
-\!\!\!\!\!\!\sum_{\begin{subarray}{c}B_1,B_2\in H_2(X)-\{0\}\\ B_1+B_2=B\\ 
l_1+l_2=l-2,\,l_1,l_2\ge0\end{subarray}} \hspace{-.42in}
\big(h_1\!\cdot_X\!B_1\big)\!\big(h_2\!\cdot_X\!B_2\!\big)
\binom{l\!-\!2}{l_1}
\binom{\ell_{\om}(B)\!-\!2l}{\ell_{\om}(B_1)\!-\!2l_1\!-\!2}
N_{B_1,l_1+1}^{\phi;\fp}N_{B_2,l_2}^{\phi;\fp}\,.
\end{split}\EE
We note that $l_1\!\equiv\!l_1(\cS)\!-\!2$ in both sums above, because
$l_1(\cS)\!-\!1$ includes the conjugate pair $(z_3^+,z_3^-)$ in
the $\cA_2$ sum;
this pair is included into $l_2(\cS)$ in the $\cA_3$~sum.

By Lemma~\ref{dropfactor_lmm}\ref{dropfactor_it1} and Proposition~\ref{Cdecomp_prp},  
{\it half} of the {\it negative} of the left-hand side of~\eref{bndsplit_e} equals
\begin{equation*}\begin{split}
&\big|\ev_{\Ga_3^+;\bh}^{\,-1}(\bp)\big|_{\fo_{\Ga_3^+;\fp;\bh}}^{\pm}
-\big|\ev_{\Ga_2^+;\bh}^{\,-1}(\bp)\big|_{\fo_{\Ga_2^+;\fp;\bh}}^{\pm}
=\big(h_1\!\cdot_X\!h_2\big)N_{B,l}^{\phi;\fp}\\
&\hspace{.5in}+\!\!\!\!
\sum_{\begin{subarray}{c}B_0,B'\in H_2(X)-\{0\}\\ B_0+\fd(B')=B\end{subarray}}
\hspace{-.4in} 2^{\ell_{\om}(B')}\!\big(B_0\!\!\cdot_X\!\!B'\big)
\big(h_1\!\cdot_X\!B'\big)\!\big(h_2\!\cdot_X\!B'\big)
\binom{l\!-\!2}{\ell_{\om}(B')}N_{B'}^XN_{B_0,l-1-\ell_{\om}(B')}^{\phi;\fp}\\
&\hspace{.5in}-\!\!\!\!
\sum_{\begin{subarray}{c}B_0,B'\in H_2(X)-\{0\}\\ B_0+\fd(B')=B\end{subarray}}
\hspace{-.4in} 2^{\ell_{\om}(B')-1}\!\big(B_0\!\!\cdot_X\!\!B'\big)
\big(h_1\!\cdot_X\!B'\big)\!\big(h_2\!\cdot_X\!B_0\big)
\binom{l\!-\!2}{\ell_{\om}(B')\!-\!1}N_{B'}^XN_{B_0,l-1-\ell_{\om}(B')}^{\phi;\fp}\,.
\end{split}\end{equation*}
Equating this expression with the negative of~\eref{Sol03rec_e5}, 
we obtain the second identity in Theorem~\ref{SolWDVV_thm}
with the left-hand side multiplied by~$h_1\!\cdot_X\!h_2$.
\end{proof}

\begin{proof}[{\bf{\emph{Proof of~\ref{Sol03rec2_it}}}}]
We now take $l^*\!=\!3$.
Since $l\!\ge\!1$ in this case, the morphism
$$\ff_{0,3}\!: \wh\cZ_{k,l+2;\bh}^{\st}(B;J,\nu)\lra \ov\cM_{0,3}^{\tau}$$
is well-defined.
Let $\Ga^{\pm}_i,\Ups,\ga'\!\subset\!\ov\cM_{0,3}^{\tau}$,
$\fo_{\Ga_i^{\pm}}^c$, $\fo_\Ups^c$, and~$\fo_{\ga'}^c$ 
be as in the statement of Lemma~\ref{M03rel_lmm}.
Let $\cA_2$ (resp.~$\cA_3$) be the collection of the codimension~1 strata~$\cS$
of $\ov\M_{k,l+2}(B;J,\nu)$ with \hbox{$\ep_3(\cS)\!=\!2$}
such~that $\P^1_1$ (resp.~$\P^1_2$)
carries the conjugate pair $(z_3^+,z_3^-)$, but not $(z_2^+,z_2^-)$.

By Corollary~\ref{mainsetup_crl} and Proposition~\ref{Rdecomp_prp},
{\it half} of the right-hand side of~\eref{bndsplit_e} equals
\begin{equation*}\begin{split}
&\sum_{\cS\in\cA_2}\!\!
\big|\cS_{\bh,\bp;\Ups}^*\big|_{\prt\fo_{\fp;\bh},\fo_\Ups^c}^{\pm}
+\sum_{\cS\in\cA_3}\!\!
\big|\cS_{\bh,\bp;\Ups}^*\big|_{\prt\fo_{\fp;\bh},\fo_\Ups^c}^{\pm}\\
&\hspace{.5in}
=\!\!\sum_{\cS\in\cA_3}\!\!
\big(h_1\!\cdot_X\!B_1(\cS)\!\big)\!\big(h_2\!\cdot_X\!B_1(\cS)\!\big)
\big(h_3\!\cdot_X\!B_2(\cS)\!\big)
N_{B_1(\cS),l_1(\cS)-2}^{\phi;\fp}N_{B_2(\cS),l_2(\cS)-1}^{\phi;\fp}\\
&\hspace{.8in}
-\!\!\sum_{\cS\in\cA_2}\!\!
\big(h_1\!\cdot_X\!B_1(\cS)\!\big)\!\big(h_3\!\cdot_X\!B_1(\cS)\!\big)
\big(h_2\!\cdot_X\!B_2(\cS)\!\big)
N_{B_1(\cS),l_1(\cS)-2}^{\phi;\fp}N_{B_2(\cS),l_2(\cS)-1}^{\phi;\fp}.
\end{split}\end{equation*}
Summing over all splittings of $B\!\in\!H_2(X)$ into~$B_1$  and~$B_2$,
of $l\!-\!1$ conjugate pairs of points into sets of cardinalities~$l_1$ and~$l_2$,
and of $k$ real points into sets of the appropriate cardinalities, 
we obtain
\BE{Sol03rec2_e5}\begin{split}
&\frac12\big|\big(\prt\wh\cZ_{k,l+2;\bh}^{\st}(B;J)\!\big)_{(\ev_{k,l+2;\bh},\ff_{k',l'})}
\!\times_{f_{\bp;\Ups}}\!\Ups\big|_{\prt\wh\fo_{\fp;\bh},\fo_\Ups^c}^{\pm}\\
&\hspace{.2in}=\!\sum_{\begin{subarray}{c} B_1,B_2\in H_2(X)-\{0\}\\ B_1+B_2=B\\
l_1+l_2=l-1,\,l_1,l_2\ge0\end{subarray}} \hspace{-.4in}
\big(h_1\!\cdot_X\!B_1\big)\!\big(h_2\!\cdot_X\!B_1\!\big)\big(h_3\!\cdot_X\!B_2\!\big)
\binom{l\!-\!1}{l_1}\!\binom{\ell_{\om}(B)\!-\!2l}{\ell_{\om}(B_1)\!-\!2l_1}
N_{B_1,l_1}^{\phi;\fp}N_{B_2,l_2}^{\phi;\fp}\\
&\hspace{.4in}
-\!\!\!\!\!\!\sum_{\begin{subarray}{c}B_1,B_2\in H_2(X)-\{0\}\\ B_1+B_2=B\\ 
l_1+l_2=l-1,\,l_1,l_2\ge0\end{subarray}} \hspace{-.4in}
\big(h_1\!\cdot_X\!B_1\big)\!\big(h_3\!\cdot_X\!B_1\big)
\big(h_2\!\cdot_X\!B_2\!\big)\binom{l\!-\!1}{l_1}\!
\binom{\ell_{\om}(B)\!-\!2l}{\ell_{\om}(B_1)\!-\!2l_1}
N_{B_1,l_1}^{\phi;\fp}N_{B_2,l_2}^{\phi;\fp}\,.
\end{split}\EE
We note that $l_1\!\equiv\!l_1(\cS)\!-\!2$ and $l_2\!\equiv\!l_2(\cS)\!-\!1$
in both sums above, because the subtractions from~$l_1(\cS)$ and $l_2(\cS)$
correspond to the insertions of the divisors~$H_1,H_2,H_3$.

By Lemma~\ref{dropfactor_lmm}\ref{dropfactor_it1}, Proposition~\ref{Cdecomp_prp}, 
and the last equation in~\eref{h1h2inter_e}, 
{\it half} of the {\it negative} of the left-hand side of~\eref{bndsplit_e} equals
\begin{equation*}\begin{split}
&\big|\ev_{\Ga_3^+;\bh}^{\,-1}(\bp)\big|_{\fo_{\Ga_3^+;\fp;\bh}}^{\pm}
-\big|\ev_{\Ga_2^+;\bh}^{\,-1}(\bp)\big|_{\fo_{\Ga_2^+;\fp;\bh}}^{\pm}
=\big(h_1\!\cdot_X\!h_2\big)\big(h_3\!\cdot_X\!B\big)N_{B,l}^{\phi;\fp}\\
&\hspace{.5in}+\!\!\!\!
\sum_{\begin{subarray}{c}B_0,B'\in H_2(X)-\{0\}\\ B_0+\fd(B')=B\end{subarray}}
\hspace{-.4in} 2^{\ell_{\om}(B')}\!\big(B_0\!\!\cdot_X\!\!B'\big)
\big(h_1\!\cdot_X\!B'\big)\!\big(h_2\!\cdot_X\!B'\big)\big(h_3\!\cdot_X\!B_0\!\big)
\binom{l\!-\!1}{\ell_{\om}(B')}N_{B'}^XN_{B_0,l-1-\ell_{\om}(B')}^{\phi;\fp}\\
&\hspace{.5in}-\!\!\!\!
\sum_{\begin{subarray}{c}B_0,B'\in H_2(X)-\{0\}\\ B_0+\fd(B')=B\end{subarray}}
\hspace{-.4in} 2^{\ell_{\om}(B')}\!\big(B_0\!\!\cdot_X\!\!B'\big)
\big(h_1\!\cdot_X\!B'\big)\!\big(h_3\!\cdot_X\!B'\big)\big(h_2\!\cdot_X\!B_0\big)
\binom{l\!-\!1}{\ell_{\om}(B')}N_{B'}^XN_{B_0,l-1-\ell_{\om}(B')}^{\phi;\fp}\,.
\end{split}\end{equation*}
Equating this expression with the negative of~\eref{Sol03rec_e5}, 
we obtain the last identity in Theorem~\ref{SolWDVV_thm}
with the left-hand side multiplied by~$h_1\!\cdot_X\!h_2$.
\end{proof}

\begin{proof}[{\bf{\emph{Proof of Lemma~\ref{mainsetup_lmm}}}}]
It is sufficient to show that 
\BE{evfkl_e2}   \big(\ev,\ff_{k',l'}\big)\!: 
\wh\M_{k,l+l^*-1;l^*}^{\st}(B;J,\nu)\lra X_{k,l+l^*-1}\!\times\!\ov\cM_{k',l'}^{\tau}\EE
is a bordered $\Z_2$-pseudocycle of dimension $4l\!+\!2k\!+\!2(l^*\!-\!1)$
transverse~to
\begin{gather}\label{pHSinc_e4}
 f_{\bh;\bp;\Ups}\!: M_{\bh}\!\times\!\Ups\lra X_{k,l+l^*-1}\!\times\!\ov\cM_{k',l'}^{\tau},\\
\notag
f_{\bh;\bp;\Ups}(y_1,\ldots,y_{l^*},P)=\big(h_1(y_1),\ldots,h_{l^*}(y_{l^*}),\bp,P\big).
\end{gather}
Since $\dim_{\R}X\!=\!4$, $\lr{c_1(X,\om),B'}\!\ge\!1$ for every $B'\!\in\!H_2(X)\!-\!\{0\}$
which can be represented by a $J$-holomorphic map \hbox{$u\!:\P^1\!\lra\!X$}
for a generic $J\!\in\!\cJ_{\om}^{\phi}$.

For a stratum $\cS$ of $\wh\M_{k,l+l^*-1;l^*}(B;J,\nu)$, 
we denote by $\cS^*\!\subset\!\cS$ the subspace of simple maps
and by~$\fc(\cS^*)$ the number of nodes of maps in~$\cS$.
The image of~$\cS$ 
under $\ff_{k',l'}$ is contained in a stratum $\cS^{\vee}$ of $\ov\cM^\tau_{k',l'}$
with $\fc(\cS^{\vee})\!\le\!\fc(\cS)$. 
For a generic choice of~$(J,\nu)$, $\cS^*\!\subset\!\cS$ is a smooth manifold of dimension
\BE{mainsetup_e3a}
\dim\,\cS^*=\ell_{\om}(B)\!+\!2(l\!+\!l^*\!-\!1)\!+\!k-\fc(\cS)
=4l\!+\!2k\!+\!2(l^*\!-\!1)-\fc(\cS)\,.\EE
The image of $\cS\!-\!\cS^*$ under
\BE{mainsetup_e4}\big(\ev,\ff_{k',l'}\big)\!: 
\wh\M_{k,l+l^*-1;l^*}(B;J,\nu)\lra X_{k,l+l^*-1}\!\times\!\ov\cM_{k',l'}^{\tau}\EE
is covered by smooth maps from manifolds~$\cS'$ with
\BE{mainsetup_e3b}
\dim\,\cS'\le\ell_{\om}(B)\!+\!2(l\!+\!l^*\!-\!1)\!+\!k-2-\fc(\cS^{\vee})
=4l\!+\!2k\!\!+\!2(l^*\!-\!2)-\fc(\cS^{\vee});\EE
see \cite[Section~3]{RT2} and~\cite[Section~3.4]{RealRT}.

The space $\wh\M_{k,l+l^*-1;l^*}^{\st}(B;J,\nu)$ consists of the main stratum
$\M_{k,l+l^*-1;l^*}(B;J,\nu)$ and the subspaces~$\cS^*$ of the strata~$\cS$
with either one real node only or one conjugate pair of nodes only.
Such strata have disjoint open neighborhoods in $\wh\M_{k,l+l^*-1;l^*}^{\st}(B;J,\nu)$.
Thus, the gluing maps as in~\cite{LT} for these strata can be chosen so that their images
do not overlap.
Along with the smooth structure of $\M_{k,l+l^*-1;l^*}(B;J,\nu)$, these maps
then determine a smooth structure on $\wh\M_{k,l+l^*-1;l^*}^{\st}(B;J,\nu)$
with respect to which the map~\eref{evfkl_e2} is smooth.

Since the space $\wh\M_{k,l+l^*-1;l^*}(B;J,\nu)$ is compact,
\BE{mainsetup_e5}
\Om\big((\ev,\ff_{k',l'})\big|_{\wh\M_{k,l+l^*-1;l^*}^{\st}(B;J,\nu)}\big)
\subset \big\{(\ev,\ff_{k',l'})\big\}\!
\big(\wh\M_{k,l+l^*-1;l^*}(B;J,\nu)\!-\!\wh\M_{k,l+l^*-1;l^*}^{\st}(B;J,\nu)\!\big).\EE 
The complement on the right-hand side above consists of the subspaces~$\cS^*$
of the strata~$\cS$ with $\fc(\cS)\!\ge\!2$ nodes
and of the subspaces $\cS\!-\!\cS^*$ with $\fc(\cS)\!\ge\!1$.
Combining this with~\eref{mainsetup_e3a} and~\eref{mainsetup_e3b}, 
we conclude that the left-hand side of~\eref{mainsetup_e5} is covered by
smooth maps from manifolds of dimension at most $4l\!+\!2k\!+\!2(l^*\!-\!2)$.
Thus, \eref{evfkl_e2} is a bordered $\Z_2$-pseudocycle of dimension 
\hbox{$4l\!+\!2k\!+\!2(l^*\!-\!1)$}.

For a generic $(J,\nu)$, the restriction
\BE{mainsetup_e9}\big(\ev,\ff_{k',l'}\big)\!:\cS^*\lra X_{k,l+l^*-1}\!\times\!\cS^{\vee}\EE
of~\eref{mainsetup_e4} to~$\cS^*$ is transverse (in the target above) 
to $f_{\bh;\bp;\Ups'}$ for every given submanifold $\Ups'\!\subset\!\cS^{\vee}$.
Along with the smoothings of the nodes, this implies that~\eref{evfkl_e2} 
is transverse to~$f_{\bh;\bp;\Ups'}$.
Since 
$$\ff_{k',l'}\big(\prt\wh\M_{k,l+l^*-1;l^*}^{\st}(B;J,\nu)\big)\cap \prt \Ups=\eset$$
with our choices of~$\Ups$, we conclude that the restriction of~\eref{evfkl_e2}
to $\prt\wh\M_{k,l+l^*-1;l^*}^{\st}(B;J,\nu)$ is transverse to~\eref{pHSinc_e4} and 
to the restriction of~\eref{pHSinc_e4} to the boundary 
\hbox{$M_{\bh}\!\times\!\prt\Ups$} of its domain.

Since the image of $\ff_{k',l'}$ is disjoint from $\prt\ov\cM_{k',l}^{\tau}$, 
it is also disjoint from the limit~set
$$\Om(\Ups)\equiv \ov\Ups\!-\!\Ups$$
of~$\Ups$ ($\Om(\Ups)$ is empty in the case of Lemma~\ref{M12rel_lmm}
and consists of~8 points in $\prt\ov\cM_{0,3}^{\tau}$ in the case of Lemma~\ref{M03rel_lmm}).
It remains to show that smooth maps from manifolds of dimensions at most 
\hbox{$4l\!+\!2k\!+\!2(l^*\!-\!2)$}
covering the right-hand side of~\eref{mainsetup_e5} can be chosen so that they are transverse
to~\eref{pHSinc_e4} and 
to the restriction of~\eref{pHSinc_e4} to \hbox{$M_{\bh}\!\times\!\prt\Ups$}.
If $\fc(\cS)\!\ge\!2$ and $\cS$ is not a stratum of $\wh\M_{k,l+l^*-1;l^*}^{\st}(B;J,\nu)$,
i.e.~the maps in~$\cS$ do not just contain a conjugate pair of nodes, 
then the transversality of~$\Ups$ to every stratum of $\ov\cM_{k',l'}^{\tau}\!-\!\prt \Ups$,
the transversality of~\eref{mainsetup_e9} to $f_{\bh;\bp;\Ups'}$ for every 
given submanifold $\Ups'\!\subset\!\cS^{\vee}$, and \eref{mainsetup_e3a} imply~that 
\BE{mainsetup_e11}\big\{(\ev,\ff_{k',l'})\big\}(\cS^*)\cap f_{\bh;\bp;\Ups}
\big(M_{\bh}\!\times\!\Ups\big) =\eset.\EE
For any stratum~$\cS$ of $\wh\M_{k,l+l^*-1;l^*}(B;J,\nu)$,
the image of $\cS\!-\!\cS^*$ under~\eref{evfkl_e2} is
covered by smooth maps 
$$h_{\cS'}\!:\cS'\lra X_{k,l+l^*-1}\!\times\!\cS^{\vee}$$
satisfying~\eref{mainsetup_e3b};
these maps are transverse (in the target above) 
to $f_{\bh;\bp;\Ups'}$ for every given submanifold $\Ups'\!\subset\!\cS^{\vee}$
for a generic~$(J,\nu)$.
This implies that \eref{mainsetup_e11} 
holds with $\cS^*$ replaced by $\cS\!-\!\cS^*$.
Thus, the bordered $\Z_2$-pseudocycle~\eref{evfkl_e2} 
is transverse to~\eref{pHSinc_e4}.
\end{proof}

\begin{proof}[{\bf{\emph{Proof of Corollary~\ref{mainsetup_crl}}}}]
For a codimension~1 stratum~$\cS$ of $\ov\M_{k,l+l^*-1}(B;J,\nu)$ and $r\!=\!1,2$, let
\BE{JakePseudo_e8}k_r\equiv k_r(\cS), \qquad l_r\equiv l_r(\cS), \qquad
l_r^*\equiv l_r^*(\cS), \qquad B_r\equiv B_r(\cS)\EE
be as in Section~\ref{MapSpaces_subs}\,.
Suppose that $\wh\cS^*$ is a stratum of $\prt\wh\M_{k,l+l^*-1;l^*}^{\st}(B;J,\nu)$,
i.e.
$$\ep_{l^*}(\cS)\equiv \blr{c_1(X,\om),B_2}\!-\!2(l_2\!-\!l_2^*)\!-\!k_2$$
is congruent to~2 or 3 modulo~4.

Since $\Ups\!\cap\!S_1\!=\!\eset$, $(l_1,k_1)\!\neq\!(1,0)$ if $\wh\cS_{\bh,\bp;\Ups}^*\!\neq\!\eset$.
If $B_2\!=\!0$, $l_2,l_2^*\!=\!1$, and $k_2\!=\!0$, then \hbox{$\ep_{l^*}(\cS)\!=\!0$},
contrary to the assumption on~$\cS$ above.
Suppose $B_2\!=\!0$, $l_2\!=\!1$, and $l_2^*,k_2\!=\!0$.
For good choices of~$\nu$ (still sufficiently generic), the restriction to~$\wh\cS^*$
of~\eref{whfMevdfn_e} with $l$ replaced by $l\!+\!l^*\!-\!1$ then factors~as 
$$\wh\cS^*\lra  \M_{k+1,l+l^*-2}\big(B;J,\nu_1\big)\!\times\!\M_{1,1}(0;J,0)
\lra X_{k,l+l^*-2}\!\times\!X^{\phi} \lra X_{k,l+l^*-1}\,.$$
Thus, $\wh\cS_{\bh,\bp;\Ups}^*\!=\!\eset$ for generic choices of~$\bh$ and~$\bp$.
Suppose $B_2\!=\!0$, $l_2\!=\!0$, and $k_2\!=\!2$.
For good choices of~$\nu$, the restriction to~$\wh\cS^*$
of~\eref{whfMevdfn_e} with $l$ replaced by $l\!+\!l^*\!-\!1$ then factors~as 
$$\wh\cS^*\lra  \M_{k-1,l+l^*-1}\big(B;J,\nu_1\big)\!\times\!\M_{3,0}(0;J,0)
\lra X_{k-2,l+l^*-1}\!\times\!\De_X^{\phi} \lra X_{k,l+l^*-1}\,,$$
where $\De_X^{\phi}\!\subset\!(X^{\phi})^2$ is the diagonal.
Thus, $\wh\cS_{\bh,\bp;\Ups}^*\!=\!\eset$ for generic choices of~$\bh$ and~$\bp$.

We can thus assume that either $B_r\!\neq\!0$ or $2l_r\!+\!k_r\!\ge\!3$ for $r\!=\!1,2$.
For good choices of~$\nu$, the restriction to~$\wh\cS^*$
 of~\eref{whfMevdfn_e} 
with $l$ replaced by $l\!+\!l^*\!-\!1$ then factors~as 
\begin{equation*}\begin{split}
\wh\cS^*&\lra  \M_{k_1+1,l_1}\!\big(B_1;J,\nu_1\big)\!\times\!
\M_{k_2+1,l_2}\!\big(B_2;J,\nu_2\big)\\
&\lra\M_{k_1,l_1}\!\big(B_1;J,\nu_1'\big)\!\times\!
\M_{k_2,l_2}\big(B_2;J,\nu_2'\big)
\lra X_{k_1,l_1-l_1^*}\!\times\!X_{k_2,l_2-l_2^*} \lra X_{k,l-1}\,.
\end{split}\end{equation*}
Thus, $\wh\cS_{\bh,\bp;\Ups}^*\!=\!\eset$ for generic choices of $\bh$, $\bp$, and $(J,\nu)$ unless
$$\ell_{\om}\big(B_r\big)\!+\!2\big(l_r\!-\!l_r^*\!\big)\!+\!k_r\ge 
4\big(l_r\!-\!l_r^*\!\big)\!+\!2k_r
\qquad\forall\,r\!=\!1,2.$$
Since 
\begin{gather*}
\ell_{\om}(B_1)\!+\!\ell_{\om}(B_2)= \ell_{\om}(B)\!-\!1=2l\!+\!k\!-\!1, \quad
k_1\!+\!k_2=k, \quad l_1\!+\!l_2=l\!+\!l^*\!-\!1, \quad l_1^*\!+\!l_2^*=l^*,
\end{gather*}
and $\ep_{l^*}(\cS)\!\cong\!2,3$ mod~4, it follows that $\ep_{l^*}(\cS)\!=\!2$ and 
\BE{mainsetup_e25}
\ell_{\om}(B_1)=2\big(l_1\!-\!l_1^*\!\big)\!+\!k_1.\EE
If $\cS$ satisfies~\ref{cS0_it} above Proposition~\ref{Rdecomp_prp},
the restriction to~$\wh\cS^*$ 
of the composition of~\eref{evfkl_e2} with the projection to the product
\hbox{$X_{k_1,l_1}\!\times\!\ov\cM_{k',l'}^{\tau}$} factors~as 
\begin{equation*}\begin{split}
\wh\cS^*&\lra  \M_{k_1+1,l_1}\!\big(B_1;J,\nu_1\big)\!\times\!\M_{k_2+1,l_2}\!\big(B_2;J,\nu_2\big)
\lra\M_{k_1,l_1}\!\big(B_1;J,\nu_1'\big)\lra X_{k_1,l_1}\!\times\!\ov\cM_{k',l'}^{\tau}.
\end{split}\end{equation*}
Since the restriction of~\eref{evfkl_e2} to~$\wh\cS^*$ is transverse to~\eref{pHSinc_e4}
and $\Ups$ is a real hypersurface, \eref{mainsetup_e25} then implies that 
$\wh\cS_{\bh,\bp;\Ups}^*\!=\!\eset$.
\end{proof}

\section{Proofs of structural statements}
\label{proofs_sec}

\subsection{Orienting the linearized $\dbar$-operator}
\label{CRdet_subs}

For $\u$ as in~\eref{udfn_e}, let 
\begin{equation*}\begin{split}
D_{J,\nu;\u}^{\phi}\!:  \Ga(\u)
\equiv&\big\{\xi\!\in\!\Ga(\Si;u^*TX)\!:\,\xi\!\circ\!\si\!=\!\nd\phi\!\circ\!\xi\big\}\\
&\lra
\Ga^{0,1}(\u)\equiv
\big\{\ze\!\in\!\Ga(\Si;(T^*\Si,\fj)^{0,1}\!\otimes_{\C}\!u^*(TX,J))\!:\,
\ze\!\circ\!\nd\si=\nd\phi\!\circ\!\ze\big\}
\end{split}\end{equation*}
be the linearization of the $\{\dbar_J\!-\!\nu\}$-operator on the space of 
real maps from~$(\Si,\si)$ with its complex structure~$\fj$.
We define
\begin{alignat*}{2}
&\la_{\u}^{\C}(X)=\bigotimes_{i=1}^l\la\big(T_{u(z_i^+)}X\big), &\qquad
&\la_{\u}^{\R}(X)=\la\bigg(\bigoplus_{i=1}^kT_{u(x_i)}X^{\phi}\!\bigg)
=\bigotimes_{i=1}^k\la\big(T_{u(x_i)}X^{\phi}\big),\\
&\la_{\u}\big(D_{J,\nu}^{\phi}\big)=\det D_{J,\nu;\u}^{\phi}\,, &\qquad
&\wt\la_{\u}\big(D_{J,\nu}^{\phi},X\big)=
\la_{\u}^{\R}(X)^*\!\otimes\!\la_{\u}^{\C}(X)^*\!\otimes\!\la_{\u}\big(D_{J,\nu}^{\phi}\big);
\end{alignat*}
the summands and the factors in the definition of $\la_{\u}^{\R}(X)$ 
are {\it not} ordered.
By~\cite[Appendix]{RealGWsI}, the projection
\BE{Dorient_e} \wt\la\big(D_{J,\nu}^{\phi},X\big)\equiv
\bigcup_{\u\in\ov\M_{k,l}(B;J,\nu)}\hspace{-.32in}
\big\{\u\}\!\times\!\wt\la_{\u}\big(D_{J,\nu}^{\phi},X\big)\lra \ov\M_{k,l}(B;J,\nu)\EE
is a line orbi-bundle with respect to a natural topology on its domain.

For $i\!\in\![k]$ and $\u\!\in\!\M_{k,l}(B;J,\nu)$ with the associated marked curve~$\cC$
as in~\eref{cCdfn_e}, let
$$j_i(\u)=j_i(\cC)\in[k]$$
be as in Section~\ref{cMstrata_subs}.
The next statement is a consequence of the orienting construction of \cite[Prop.~3.1]{Jake},
a more systematic perspective of which appears in the proof 
of \cite[Thm.~7.1]{SpinPin}.

\begin{lmm}\label{Dorient_lmm} 
Suppose $(X,\om,\phi)$ is a real symplectic fourfold, $l\!\in\!\Z^+$,
$k\!\in\!\Z^{\ge0}$ with $k\!+\!2l\!\ge\!3$,
$B\!\in\!H_2(X)$, and $(J,\nu)\!\in\!\cH_{k,l}^{\om,\phi}$.
If~$k$ and~$B$ satisfy~\eref{BKcond_e}, then a $\Pin^-$-structure~$\fp$ on~$X^{\phi}$
determines an orientation~$\fo_{\fp}^D$ on the restriction of~\eref{Dorient_e} to $\M_{k,l}(B;J,\nu)$
with the following properties:  
\BEnum{($\fo_{\fp}^D\arabic*$)}

\item\label{MijRch_it}  the interchange of two real points $x_i$ and $x_j$ preserves $\fo_{\fp}^D$;

\item\label{M1Rch_it} if $\u\!\in\!\M_{k,l}(B;J,\nu)$,
the interchange of the real points $x_1$ and $x_{j_i(\u)}$ with $2\!\le\!i\!\le\!k$
preserves~$\fo_{\fp}^D$ at~$\u$ if and only if $(k\!-\!1)(i\!-\!1)\!\in\!2\Z$;

\item\label{DorientCij_it} if $\u\!\in\!\M_{k,l}(B;J,\nu;\wch X^\phi)$ and 
the marked points $z_i^+$ and $z_j^+$ are not separated by the fixed locus~$S^1$
of the domain of~$\u$, then
the interchange of the conjugate pairs $(z_i^+,z_i^-)$ and $(z_j^+,z_j^-)$
preserves~$\fo_{\fp}^D$ at~$\u$; 

\item\label{DorientCipm_it}  
the interchange of the points in a conjugate pair $(z_i^+,z_i^-)$ with $1\!<\!i\!\le\!l$
preserves~$\fo_{\fp}^D$;

\item\label{DorientC1pm_it} the interchange of the points in the conjugate pair $(z_1^+,z_1^-)$ preserves
$\fo_{\fp}^D$ if and only~if 
$$k\neq0~~\hbox{and}~~\ell_{\om}(B)\cong2,3~\tn{mod~4} \qquad\hbox{or}\qquad
k=0~~\hbox{and}~~\ell_{\om}(B)\cong0~\tn{mod}~4;$$

\item\label{DB0_it} if $k,l,l^*\!=\!1$, $B\!=\!0$, and $\nu$ is small, 
then $\fo_{\fp}^D$ is the orientation induced
by the evaluation at~$x_1$.

\EEnum
\end{lmm}

\begin{proof}
Let $\u$ be as in~\eref{udfn_e}.
For the purposes of applying \cite[Thm.~7.1]{SpinPin},
we take the distinguished half-surface \hbox{$\D^2\!\subset\!\P^1$} to be the disk 
so that $\prt\D^2$ is the fixed locus~$S^1$ of~$\tau$ and $z_1^+\!\in\!\D^2$.
A $\Pin^-$-structure~$\fp$ on~$X^{\phi}$ then determines an orientation~$\fo_{\fp;\R}^D$
on the line  \hbox{$\la_{\u}^{\R}(X)^*\!\otimes\!\la_{\u}(D_{J,\nu}^{\phi})$}
varying continuously with~$\u$.
Since $\fo^D_{\fp;\R}$ does not depend on the conjugate pairs of marked points, 
except for~$z^\pm_1$ which determines~$\D^2$, 
$\fo^D_{\fp;\R}$ satisfies~\ref{DorientCij_it} and~\ref{DorientCipm_it}.
By the CROrient~1$\fp$ property in \cite[Section~7.2]{SpinPin},
$\fo_{\fp;\R}^D$ satisfies \ref{MijRch_it}, \ref{M1Rch_it}, and~\ref{DorientC1pm_it}.
By the CROrient~5b and~6b properties in \cite[Section~7.2]{SpinPin},
it also satisfies~\eref{DB0_it}. 
Along with the symplectic orientations of $T_{u(z_i^+)}X$, $\fo_{\fp;\R}^D$
determines an orientation~$\fo_{\fp}^D$ on $\wt\la_{\u}(D_{J,\nu}^{\phi},X)$
varying continuously with~$\u$.
Since the complex dimension of~$X$ is even, $\fo_{\fp}^D$~also satisfies all six properties.
\end{proof}

Suppose now that $l\!\in\!\Z^+$ and $\cS$ is an open codimension~1 stratum of $\ov\M_{k,l}(B;J,\nu)$.
Define
$$r(\cS)=\begin{cases}1,&\hbox{if}~k\!=\!0~\hbox{or}~1\!\in\!K_1(\cS);\\
2,&\hbox{if}~1\!\in\!K_2(\cS).\end{cases}$$
An orientation~$\fo_{\cS;\u}^c$ of~$\cN_{\u}\cS$ determines a direction of degeneration 
of elements of $\M_{k,l}(B;J,\nu)$ to~$\u$.
The orientation~$\fo_{\fp}^D$ on~\eref{Dorient_e} limits to an orientation 
$\fo_{\fp;\u}^D$ of $\wt\la_{\u}(D_{J,\nu}^{\phi},X)$ by approaching~$\u$ from this direction.
The orientation $\fo_{\fp;\u}^D$ is called the \sf{limiting} orientation induced by~$\fp$ 
and~$\fo_{\cS;\u}^c$ in \cite[Section~7.3]{SpinPin}.
If in addition $l_2^*(\cS)\!\ge\!1$,  the possible orientations~$\fo_{\cS;\u}^{c;\pm}$ of~$\cN_{\u}\cS$ 
are distinguished as above Lemma~\ref{DMboundary_lmm}. 
We denote by $\fo_{\fp;\u}^{D;\pm}$ the limiting orientation of $\wt\la_{\u}(D_{J,\nu}^{\phi},X)$
induced by~$\fp$ and~$\fo_{\cS;\u}^{c;\pm}$.

The domain of each element $\u\!\in\!\cS$ consists of an irreducible component~$\P^1_1$
carrying the marked points~$z_1^{\pm}$ with fixed locus~$S^1_1$ 
and another irreducible component~$\P^1_2$ with fixed locus~$S^1_2$.
The fixed locus $S^1_r$ splits~$\P^1_r$ into two disks.
Let $\cS_*\!\subset\!\cS$ be the subspace of all maps with fixed distributions 
of the marked points~$z_i^+$ with $i\!\in\![l]$ between the four disks and 
with fixed orderings of the marked points~$x_i$ with $i\!\in\![k]$ and the nodal points
on the two fixed loci.
We call such a subspace a \sf{substratum} of~$\cS$.
If $k_2(\cS)\!+\!2l_2(\cS)\!\ge\!2$, i.e.~the marked domain~\eref{cCdfn_e} of every element 
$\u\!\in\!\cS$ is stable, then the image of~$\cS$ under the forgetful morphism
$$\ff_{k,l}\!:\ov\M_{k,l}(B;J,\nu)\lra \ov\cM_{k,l}^{\tau}$$
is contained in a codimension~1 stratum~$\cS^{\vee}$.
In such a case, a substratum~$\cS_*$ of~$\cS$ is given~by
$$\cS_*=\cS\cap f_{k,l}^{-1}(\cS^{\vee}_*)$$
for some topological component~$\cS^{\vee}_*$ of~$\cS^{\vee}$.

For good choices of~$\nu$, there are a natural embedding
\BE{cSsplit_e} \cS_* \lhra{~~~} \M_1\!\times\!\M_2
\subset\M_{k_1(\cS)+1,l_1(\cS)}\big(B_1(\cS);J,\nu_1\big)\!\times\!
\M_{k_2(\cS)+1,l_2(\cS)}\big(B_2(\cS);J,\nu_2\big)\EE
for some unions $\M_1$ and $\M_2$ of topological components of the moduli spaces
on the right-hand side above and forgetful morphisms
\BE{cSsplit_e2}\begin{split}
\ff_{\nod}\!:\M_{k_1(\cS)+1,l_1(\cS)}\big(B_1(\cS);J,\nu_1\big)
&\lra \M_{k_1(\cS),l_1(\cS)}\big(B_1(\cS);J,\nu_1'\big),\\
\ff_{\nod}\!:\M_{k_2(\cS)+1,l_2(\cS)}\big(B_2(\cS);J,\nu_2\big)
&\lra \M_{k_2(\cS),l_2(\cS)}\big(B_2(\cS);J,\nu_2'\big)
\end{split}\EE
dropping the real marked points corresponding to the nodal points~$\nod$ on the two components.
We choose the embedding in~\eref{cSsplit_e} so that 
it satisfies~\ref{SsplitC_it} and~\ref{SsplitR_it} in Section~\ref{NBstrata_subs}.
For an element $\u\!\in\!\cS$, we denote~by  
$$\u_1\in\M_{k_1(\cS)+1,l_1(\cS)}\big(B_1(\cS);J,\nu_1\big)
\quad\hbox{and}\quad 
\u_2\in\M_{k_2(\cS)+1,l_2(\cS)}\big(B_2(\cS);J,\nu_2\big)$$
the pair of maps corresponding to~$\u$ via~\eref{cSsplit_e}.
Let
$$\u_1'\in\M_{k_1(\cS),l_1(\cS)}\big(B_1(\cS);J,\nu_1'\big)
\quad\hbox{and}\quad 
\u_2'\in\M_{k_2(\cS),l_2(\cS)}\big(B_2(\cS);J,\nu_2'\big)$$
be the images of~$\u_1$ and~$\u_2$ under the forgetful morphisms in~\eref{cSsplit_e2}.

\noindent
Suppose $k$ and $B$ satisfy~\eref{BKcond_e}, $l_2^*(\cS)\!\ge\!1$, and $i^*\!\in\!L_2(\cS_*)$ 
is as above Lemma~\ref{DMboundary_lmm}.
For each $\u\!\in\!\cS_*$, the exact sequences
\begin{alignat}{2}\label{Dses_e}
0&\lra D_{J,\nu;\u}^{\phi}\lra 
D_{J,\nu_1';\u_1'}^{\phi}\!\oplus\!D_{J,\nu_2;\u_2}^{\phi}\lra T_{u(\nod)}X^{\phi}\lra0, &\quad 
\big(\xi_1,\xi_2\big)&\lra \xi_2(\nod)\!-\!\xi_1(\nod),\\
\notag
0&\lra D_{J,\nu;\u}^{\phi}\lra 
D_{J,\nu_1;\u_1}^{\phi}\!\oplus\!D_{J,\nu_2;\u_2'}^{\phi}\lra T_{u(\nod)}X^{\phi}\lra0, &\quad 
\big(\xi_1,\xi_2\big)&\lra \xi_2(\nod)\!-\!\xi_1(\nod),
\end{alignat}
of Fredholm operators determine isomorphisms 
\BE{Du0isom_e}\begin{split} 
\la_{\u}\big(D_{J,\nu}^{\phi}\big)\!\otimes\!\la\big(T_{u(\nod)}X^{\phi}\big)
&\approx \la_{\u_1'}\big(D_{J,\nu_1'}^{\phi}\big)\!\otimes\!\la_{\u_2}\big(D_{J,\nu_2}^{\phi}\big)\,,\\
\la_{\u}\big(D_{J,\nu}^{\phi}\big)\!\otimes\!\la\big(T_{u(\nod)}X^{\phi}\big)
&\approx \la_{\u_1}\big(D_{J,\nu_1}^{\phi}\big)\!\otimes\!\la_{\u_2'}\big(D_{J,\nu_2'}^{\phi}\big).
\end{split}\EE
If $\ep_{l^*}(\cS)\!\in\!2\Z$ (for any $l^*\!\in\![l]$), 
a $\Pin^-$-structure~$\fp$ on~$X^{\phi}$ determines homotopy classes of isomorphisms
$$\la_{\u_1'}\big(D_{J,\nu_1'}^{\phi}\big)\lra 
\la_{\u_1'}^{\R}(X)\!\otimes\!\la_{\u_1'}^{\C}(X) \quad\hbox{and}\quad
\la_{\u_2}\big(D_{J,\nu_2}^{\phi}\big)\lra 
\la_{\u_2}^{\R}(X)\!\otimes\!\la_{\u_2}^{\C}(X);$$
see Lemma~\ref{Dorient_lmm}.
Combining these isomorphisms with the first isomorphism in~\eref{Du0isom_e}, 
we obtain a homotopy class of isomorphisms
\BE{Dsplitisom_e}\begin{split}
\la_{\u}\big(D_{J,\nu}^{\phi}\big)\!\otimes\!\la\big(T_{u(\nod)}X^{\phi}\big)
&\approx 
\la_{\u_1'}^{\R}(X)\!\otimes\!\la_{\u_1'}^{\C}(X)\!\otimes\!
\la_{\u_2}^{\R}(X)\!\otimes\!\la_{\u_2}^{\C}(X)\\
&\approx \la_{\u_1'}^{\R}(X)\!\otimes\!\la_{\u_1'}^{\C}(X)
\!\otimes\!\la\big(T_{u(\nod)}X^{\phi}\big)
\!\otimes\!\la_{\u_2'}^{\R}(X)\!\otimes\!\la_{\u_2'}^{\C}(X)\\
&\approx \la_{\u}^{\R}(X)\!\otimes\!\la_{\u}^{\C}(X)\!\otimes\!
\la\big(T_{u(\nod)}X^{\phi}\big)\,.
\end{split}\EE
If $\ep_{l^*}(\cS)\!\not\in\!2\Z$, a $\Pin^-$-structure~$\fp$ on~$X^{\phi}$
similarly determines a homotopy class of isomorphisms
\begin{equation*}\begin{split}
\la_{\u}\big(D_{J,\nu}^{\phi}\big)\!\otimes\!\la\big(T_{u(\nod)}X^{\phi}\big)
&\approx 
\la_{\u_1}^{\R}(X)\!\otimes\!\la_{\u_1}^{\C}(X)\!\otimes\!
\la_{\u_2'}^{\R}(X)\!\otimes\!\la_{\u_2'}^{\C}(X)\\
&\approx \la\big(T_{u(\nod)}X^{\phi}\big)\!\otimes\!\la_{\u_1'}^{\R}(X)\!\otimes\!\la_{\u_1'}^{\C}(X)
\!\otimes\!\la_{\u_2'}^{\R}(X)\!\otimes\!\la_{\u_2'}^{\C}(X)\\
&\approx \la_{\u}^{\R}(X)\!\otimes\!\la_{\u}^{\C}(X)\!\otimes\!
\la\big(T_{u(\nod)}X^{\phi}\big).
\end{split}\end{equation*}
In either case, we denote the associated orientation on $\wt\la_{\u}(D_{J,\nu}^{\phi},X)$
by~$\fo_{\fp;\u}^D$.

If $l_2^*(\cS)\!\ge\!1$, 
we choose the embedding~\eref{cSsplit_e} so that
the real marked points of the tuples of~$\u_1$ and~$\u_2$ corresponding 
to $\u\!\in\!\cS_*$ are ordered by their position on $S^1_1\!\subset\!\P^1_1$
and $S^1_2\!\subset\!\P^1_2$, respectively, starting from the node
in the counterclockwise
direction with respect to $z_1^+\!\in\!\P^1_1$ and $z_{i^*}^+\!\in\!\P^1_2$.
We define $\de_{\R}^{\pm}(\cS_*)\!\in\!\Z$ as above Lemma~\ref{DMboundary_lmm} and set
\begin{equation*}\begin{split}
\de_D^+(\cS)=k_2(\cS)\lr{w_2(X),B_1(\cS)}, ~~
\de_D^-(\cS)&=\de_D^+(\cS)\!+\!
\frac{(\ell_{\om}(B_2(\cS))\!-\!k_2(\cS))(\ell_{\om}(B_2(\cS))\!-\!k_2(\cS)\!+\!1)}{2}\,.
\end{split}\end{equation*}

\begin{lmm}\label{DorientComp_lmm}
Suppose $(X,\om,\phi)$, $\fp$, $k,l,B$, and $(J,\nu)$
are as in Lemma~\ref{Dorient_lmm}, 
the pair $(k,B)$ satisfies~\eref{BKcond_e},
and $\cS_*$ is a substratum of a codimension~1 stratum~$\cS$ of $\ov\M_{k,l}(B;J,\nu)$ with 
$l_2^*(\cS)\!\ge\!1$.
The orientations~$\fo_{\fp}^{D;\pm}$ and~$\fo_{\fp}^D$ on 
$\wt\la(D_{J,\nu}^{\phi},X)|_{\cS_*}$ are the same  
if and only~if $\de_D^{\pm}(\cS)\!\cong\!\de_{\R}^{\pm}(\cS_*)$ mod~2.
\end{lmm}

\begin{proof} 
Let $\u\!\in\!\cS_*$.
We define $r_{\ep}(\cS)$ to be~1 if $\ep_{l^*}(\cS)\!\in\!2\Z$ and 
2 if $\ep_{l^*}(\cS)\!\not\in\!2\Z$.
Let $j_{\ep}'(\u)\!\in\!\Z^{\ge0}$ be the number of real marked points that lie  
on the oriented arc of~$S_{r_{\ep}(\cS)}^1$ between the node  and the real marked point 
$x_i\!\in\!S^1_{r_{\ep}(\cS)}$ with the smallest value of~$i$;
if $k_{r_{\ep}(\cS)}(\cS)\!=\!0$, we take $j_{\ep}'(\u)\!=\!0$.
The marked points $z_1^+\!\in\!\P^1_1$ and $z_{i^*}^+\!\in\!\P^1_2$
determine the distinguished disks as in the proof of Lemma~\ref{Dorient_lmm}.
By Lemma~\ref{Dorient_lmm}\ref{M1Rch_it}, the orientation~$\fo_{\fp}^D$  
at~$\u$ agrees with the \sf{split} orientation of \cite[Section~7.4]{SpinPin}
if and only if $(k_{r_{\ep}(\cS)}\!-\!1)j_{\ep}'(\u)$ is even.
Thus, \cite[Cor.~7.5]{SpinPin} implies the claim for 
\hbox{$\la_{\u}^{\R}(X)^*\!\otimes\!\la_{\u}(D_{J,\nu}^{\phi})$}.
Since the conjugate pairs of marked points have the same effect on
$\fo_{\fp}^{D;\pm}$ and~$\fo_{\fp}^D$, the claim follows.
\end{proof}

\subsection{Proofs of Lemmas~\ref{orient_lmm} and~\ref{orient_lmm2} and 
Proposition~\ref{JakePseudo_prp}}
\label{orient_subs}

Suppose $(X,\om,\phi)$, $\fp,k,l,l^*,B$, and $(J,\nu)$ are as in Lemma~\ref{orient_lmm},
the pair $(k,B)$ satisfies~\eref{BKcond_e}, and $(J,\nu)$ is generic.
The exact sequences
$$0\lra\ker D_{J,\nu;\u}^{\phi}\lra T_{\u}\M_{k,l}^*(B;J,\nu)\lra
T_{\ff_{k,l}(\u)}\cM_{k,l}^{\tau}\lra0$$
with \hbox{$\u\in\M_{k,l}^*(B;J,\nu)$} induced by the forgetful morphism~$\ff_{k,l}$ 
determine an isomorphism
\BE{OrientSubs_e3}\begin{split}
\la\big(\ev\big|_{\M_{k,l}^*(B;J,\nu)}\big)&\equiv
\ev^*\la^{\R}(X)^*\!\otimes\!\ev^*\la^{\C}(X)^*\!\otimes\!
\la\big(\M_{k,l}^*(B;J,\nu)\big)\\
&\approx \wt\la\big(D_{J,\nu}^{\phi},X\big)\!\otimes\!
\ff_{k,l}^*\la\big(\cM_{k,l}^{\tau}\big)
\end{split}\EE
of line bundles over $\M_{k,l}^*(B;J,\nu)$.
By Lemma~\ref{Dorient_lmm}, the $\Pin^-$-structure~$\fp$ on~$X^{\phi}$ induces 
an orientation~$\fo_{\fp}^D$ on the first factor on the right-hand side above.
Along with the orientation~$\fo_{k,l;l^*}$ on the second factor defined in 
Section~\ref{cMstrata_subs}, it determines a relative orientation~$\fo_{\fp;l^*}$  
on the restrictions of~\eref{orienlmm_e} to $\M_{k,l}^*(B;J,\nu)$
via~\eref{OrientSubs_e3}.

\begin{proof}[{\bf{\emph{Proofs of Lemmas~\ref{orient_lmm} and~\ref{orient_lmm2}}}}]
By Lemmas~\ref{cMorient_lmm} and~\ref{Dorient_lmm}, the relative orientation~$\fo_{\fp;l^*}$ 
above satisfies all properties listed in Lemma~\ref{orient_lmm} wherever it is defined.
Every (continuous) extension of~$\fo_{\fp;l^*}$ to subspaces of 
the domains of the maps in~\eref{orienlmm_e} satisfies the same properties.
The relative orientation~$\fo_{\fp;l^*}$ automatically extends over all strata 
of codimension~2 and higher.
By Lemma~\ref{orient_lmm2}, it extends over the codimension~1 strata of 
the two domains as well.
Lemma~\ref{orient_lmm2} in turn follows immediately from 
Lemmas~\ref{DMboundary_lmm} and~\ref{DorientComp_lmm}.
\end{proof}
 
The next observation, which is used in the proof of Proposition~\ref{JakePseudo_prp},
is straightforward.

\begin{lmm}\label{3by3_lmm}
Suppose $A_{ij}$ with $i,j\!\in\![3]$ are oriented finite-dimensional vector spaces,
the rows and columns in the diagram in Figure~\ref{3by3_fig} are exact sequences
of vector-space homomorphisms, and this diagram commutes.
The total number of rows and columns in this diagram which (do not)
respect the orientations is congruent to
$\dim(A_{13})\dim(A_{31})$ mod~2.
\end{lmm}

\begin{figure}
$$\xymatrix{& 0\ar[d] & 0\ar[d] & 0\ar[d]  \\
0\ar[r] & A_{11}\ar[r]\ar[d] & A_{12}\ar[r]\ar[d] & A_{13}\ar[r]\ar[d] & 0 \\
0\ar[r] & A_{21}\ar[r]\ar[d] & A_{22}\ar[r]\ar[d] & A_{23}\ar[r]\ar[d] & 0 \\
0\ar[r] & A_{31}\ar[r]\ar[d] & A_{32}\ar[r]\ar[d] & A_{33}\ar[r]\ar[d] & 0 \\
& 0 & 0 & 0}$$ 
\caption{Commutative square of vector spaces with
exact rows and columns for the statement of Lemma~\ref{3by3_lmm}}
\label{3by3_fig}
\end{figure}

\begin{proof}[{\bf{\emph{Proof of Proposition~\ref{JakePseudo_prp}}}}]
We continue with the notation in the proof of Lemma~\ref{mainsetup_lmm},
but apply it to the strata~$\cS$ of $\ov\M_{k,l}(B;J,\nu)$.
Let 
\BE{JakePseudo_e1}
\ev_{[l^*]}=\prod_{i=1}^{l^*}\!\ev^+_i\!: \ov\M_{k,l}(B;J,\nu)\lra X^{l^*}\EE
and \hbox{$h\!:Z\!\lra\!X^{l^*}$} be a smooth map from a manifold of dimension
$2l^*\!-\!2$ that covers~$\Om(f_{\bh})$.
Let
\begin{alignat}{2}
\label{JakePseudo_e3}
\ev_{k,l;\bh}\!:\ov\cZ_{k,l;\bh}(B;J,\nu)&\!\equiv\!\big\{
(\u,y)\!\in\!\ov\M_{k,l}(B;J,\nu)\!\times\!M_{\bh}\!:
\ev_{[l^*]}(\u)\!=\!f_{\bh}(y)\big\} \lra\! X_{k,l-l^*},\\
\notag
\ev_{k,l;h}\!:\ov\cZ_{k,l;h}(B;J,\nu)&\!\equiv\!\big\{
(\u,z)\!\in\!\ov\M_{k,l}(B;J,\nu)\!\times\!Z\!:\ev_{[l^*]}(\u)\!=\!h(z)\big\}
\lra X_{k,l-l^*}
\end{alignat}
be the maps induced by~\eref{fMevdfn_e}.

For each stratum~$\cS$ of $\ov\M_{k,l}(B;J,\nu)$, define
$$\cS_{\bh}=  \ov\cZ_{k,l;\bh}(B;J,\nu)\cap\big(\cS\!\times\!M_{\bh}\big), \quad
\cS^*_{\bh}=\ov\cZ_{k,l;\bh}(B;J,\nu)\cap\big(\cS^*\!\times\!M_{\bh}\big).$$
For a generic $(J,\nu)$, the subspace~$\cS^*$ of simple maps in~$\cS$ is a smooth manifold
of dimension
\BE{JakePseudo_e4a}
\dim\,\cS^*=\ell_{\om}(B)\!+\!2l\!+\!k-\fc(\cS)=4l\!-2l^*\!+\!2k\!-\!\fc(\cS)\EE
and the restriction of~\eref{JakePseudo_e1} to~$\cS^*$ is transverse to~$f_{\bh}$
and to~$h$.
Along with~\eref{JakePseudo_e4a}, the first transversality property implies~that 
$\cS^*_{\bh}$ is a smooth manifold of dimension
\BE{JakePseudo_e3a}
\dim\,\cS^*_{\bh}=\dim\,\cS^*\!-\!2l^*
=4(l\!-\!l^*)\!+\!2k\!-\!\fc(\cS)\,.\EE

For every stratum $\cS$ of $\ov\M_{k,l}(B;J,\nu)$, there is a smooth manifold~$\cS'$
and smooth maps
$$\ev_{[l^*]}\!:\cS'\lra X^{l^*} \qquad\hbox{and}\qquad
\ev_{k,l;l^*}\!:\cS'\lra X_{k,l-l^*}$$
such that $\ev_{[l^*]}$ is transverse to $f_{\bh}$ and to~$h$,
\BE{JakePseudo_e4b}
\ev_{[l^*]}(\cS\!-\!\cS^*)\subset \ev_{[l^*]}(\cS'),
\quad\hbox{and}\quad 
\dim\,\cS'\le\ell_{\om}(B)\!+\!2l\!+\!k-2=4l\!-2l^*\!+\!2k\!-\!2\,;\EE
see \cite[Section~3]{RT2} and~\cite[Section~3.4]{RealRT}.
In particular, the map 
$$\ev_{k,l;\bh}\!:\cS'_{\bh}\!\equiv\!\big\{(\u,y)\!\in\!\cS'\!\times\!M_{\bh}\!:
\ev_{[l^*]}(\u)\!=\!f_{\bh}(y)\big\}\lra X_{k,l-l^*}$$
induced by $\ev_{k,l;l^*}$ is smooth,
\BE{JakePseudo_e3b}\ev_{k,l;\bh}\big(\cS_{\bh}\!-\!\cS_{\bh}^*\big)
\subset\ev_{k,l;\bh}(\cS_{\bh}'), 
\quad\hbox{and}\quad \dim\,\cS'_{\bh}\le 4(l\!-\!l^*)\!+\!2k\!-\!2\,.\EE

By the reasoning in the proof of Lemma~\ref{mainsetup_lmm} applied
to the space $\M_{k,l;l^*}^{\st}(B;J,\nu)$ instead of $\wh\M_{k,l+l^*-1;l^*}^{\st}(B;J,\nu)$, 
$\M_{k,l;l^*}^{\st}(B;J,\nu)$ is a smooth manifold.
Along with the first transversality property after~\eref{JakePseudo_e4a}, 
this implies that~\eref{JakePseudo_e} is a smooth map between 
smooth manifolds of the same dimension.
The relative orientation~$\fo_{\fp;l^*}$ of the first map in~\eref{orienlmm_e}
and the orientation~$\fo_{\bh}$ determine a relative orientation~$\fo_{\fp;l^*}\fo_{\bh}$
of~\eref{JakePseudo_e}.

Since the space $\ov\M_{k,l}(B;J,\nu)$ is compact,
\BE{JakePseudo_e5}\begin{split}
\Om\big(\ev_{k,l;\bh}\big|_{\cZ_{k,l;\bh}^{\st}(B;J,\nu)}\big)
\subset \ev_{k,l;\bh}\big(\ov\cZ_{k,l;\bh}(B;J,\nu)
\!-\!\cZ_{k,l;\bh}^{\st}(B;J,\nu)\big) \qquad&\\
\cup  \ev_{k,l;h}\big(\ov\cZ_{k,l;h}(B;J,\nu)\big)&\,.
\end{split}\EE
By~\eref{JakePseudo_e4a}, \eref{JakePseudo_e4b}, and the transversality of
the maps $\ev_{[l^*]}$ on~$\cS^*$ and~$\cS'$ to~$h$,
the last set above is covered by smooth maps from finitely many manifolds
of dimension at most 
\BE{JakePseudo_e7}\dim\,\cZ_{k,l;\bh}^{\st}(B;J,\nu)\!-\!2=4(l\!-\!l^*)\!+\!2k-2.\EE
The~set
$$\ov\cZ_{k,l;\bh}(B;J,\nu)\!-\!\cZ_{k,l;\bh}^{\st}(B;J,\nu)\subset
\ov\M_{k,l}(B;J,\nu)\!\times\!M_{\bh}$$
consists of the subspaces~$\cS_{\bh}^*$ corresponding to the strata~$\cS$ 
of $\ov\M_{k,l}(B;J,\nu)$ with 
either $\fc(\cS)\!\ge\!2$ nodes or $\ep_{l^*}(\cS)\!\cong\!2,3$ mod~2
and of the subspaces $\cS_{\bh}\!-\!\cS_{\bh}^*$ with $\fc(\cS)\!\ge\!1$.
By~\eref{JakePseudo_e3a} and~\eref{JakePseudo_e3b}, 
a smooth map from manifold of dimension~\eref{JakePseudo_e7} 
covers $\ev_{k,l;\bh}(\cS_{\bh}^*)$ if $\fc(\cS)\!\ge\!2$ and 
$\ev_{k,l;\bh}(\cS_{\bh}\!-\!\cS_{\bh}^*)$ for any stratum~$\cS$ of $\ov\M_{k,l}(B;J,\nu)$.
We show in the next two paragraphs that a smooth map from manifold of dimension~\eref{JakePseudo_e7} 
also covers $\ev_{k,l;\bh}(\cS_{\bh}^*)$ if
$\cS$ is a stratum of $\ov\M_{k,l}(B;J,\nu)$ with $\fc(\cS)\!=\!1$ and 
 $\ep_{l^*}(\cS)\!\cong\!2,3$ mod~4.
This will conclude the proof of the first claim of the proposition.

Suppose $\cS$ is a stratum of $\ov\M_{k,l}(B;J,\nu)$ with $\fc(\cS)\!=\!1$
and $k_i,l_i,l_i^*,B_i$ are as in~\eref{JakePseudo_e8}.
Let
$$f_{\bh_1}\!:M_{\bh_1}\lra X^{l_1^*}  \qquad\hbox{and}\qquad 
f_{\bh_2}\!:M_{\bh_2}\lra X^{l_2^*}$$
be the pseudocycles determined by the maps $h_1,\ldots,h_{l^*}$ corresponding
to the conjugate pairs of marked points indexed by $i\!\in\![l^*]$ 
that are carried by the first and second components of the maps in~$\cS$,
respectively.
If
$$\ep_{l^*}(\cS)\equiv \blr{c_1(X,\om),B_2}\!-\!2(l_2\!-\!l_2^*)\!-\!k_2
\equiv \big(\ell_{\om}(B_2)\!-\!2(l_2\!-\!l_2^*)\!-\!k_2\big)+1 $$
is congruent to~2 or 3 modulo~4, then 
$$(B_1,l_1,k_1)\neq(0,1,0)  \qquad\hbox{and}\qquad
(B_2,l_2,l_2^*,k_2)\neq(0,1,1,0).$$
Suppose $B_2\!=\!0$, $l_2\!=\!1$, and $l_2^*,k_2\!=\!0$.
For good choices of~$\nu$ (still sufficiently generic), 
the restriction of~\eref{JakePseudo_e3} to~$\cS_{\bh}^*$ then factors~as 
$$\cS_{\bh}^*\lra  \cZ_{k+1,l-1;\bh}^{\st}(B;J,\nu_1)\!\times\!\M_{1,1}(0;J,0)
\lra X_{k,l-1-l^*}\!\times\!X^{\phi} \lra X_{k,l-l^*}\,.$$
Thus,  $\ev_{k,l;\bh}(\cS_{\bh}^*)$ is contained in a smooth manifold of dimension 
\hbox{$4(l\!-\!l^*)\!+\!2k\!-\!2$}.
Suppose $B_2\!=\!0$, $l_2,l_2^*\!=\!0$, and $k_2\!=\!2$.
For good choices of~$\nu$, the restriction of~\eref{JakePseudo_e3} to~$\cS_{\bh}^*$
then factors~as 
$$\cS_{\bh}^*\lra  \cZ_{k-1,l;\bh}^{\st}(B;J,\nu_1)\!\times\!\M_{3,0}(0;J,0)
\lra X_{k-2,l-l^*}\!\times\!\De_X^{\phi} \lra X_{k,l-l^*}\,,$$
where $\De_X^{\phi}\!\subset\!(X^{\phi})^2$ is the diagonal.
Thus,  $\ev_{k,l;\bh}(\cS_{\bh}^*)$ is again contained in a smooth manifold of dimension 
\hbox{$4(l\!-\!l^*)\!+\!2k\!-\!2$}.

We can thus assume that either $B_i\!\neq\!0$ or $2l_i\!+\!k_i\!\ge\!3$ for $i\!=\!1,2$.
For good choices of~$\nu$, 
the restriction of~\eref{JakePseudo_e3} to~$\cS_{\bh}^*$ then factors~as 
$$\xymatrix{\cS^*_{\bh}\ar[r] &
\cZ_{k_1,l_1;\bh_1}^{\st}(B_1;J,\nu_1')\times\cZ_{k_2,l_2;\bh_2}^{\st}(B_2;J,\nu_2')
\ar@<-7ex>[d]_{\ev_{k_1,l_1;\bh_1}}\ar@<7ex>[d]^{\ev_{k_2,l_2;\bh_2}}\\
&X_{k_1,l_1-l_1^*}\times X_{k_2,l_2-l_2^*} \ar[r]& X_{k,l-l^*}.}$$
Thus, $\ev_{k,l;\bh}(\cS_{\bh}^*)$ is covered by a smooth map from a manifold
of dimension
\begin{equation*}\begin{split}
\dim\,\cZ_{k_i,l_i;\bh_i}^{\st}(B_i;J,\nu_i)\!+\!
\dim\,X_{k_{3-i},l_{3-i}-l_{3-i}^*}
&=\ell_{\om}(B_i)\!+\!2(l_i\!-\!l_i^*)\!+\!k_i
+4(l_{3-i}-l_{3-i}^*)\!+\!2k_{3-i}\\
&=4(l\!-\!l^*)\!+\!2k+\big(\ell_{\om}(B_i)\!-\!2(l_i\!-\!l_i^*)\!-\!k_i\big)
\end{split}\end{equation*}
for $i\!=\!1,2$.
Since
$$\big(\ell_{\om}(B_1)\!-\!2(l_1\!-\!l_1^*)\!-\!k_1\big)
+\big(\ell_{\om}(B_2)\!-\!2(l_2\!-\!l_2^*)\!-\!k_2\big)
=\ell_{\om}(B)\!-\!1\!-\!2(l\!-\!l^*)\!-\!k=-1,$$
it follows that $\ev_{k,l;\bh}(\cS^*)$ is covered by a smooth map from a manifold
of dimension~\eref{JakePseudo_e7} unless 
$\ep_{l^*}(\cS)$ is either~0 or~1.
Along with the previous paragraph, this confirms the claim at the end of 
the paragraph containing~\eref{JakePseudo_e7}.

It remains to establish~\eref{RdivRel_e}.
We can assume that $B\!\neq\!0$ and can be represented by a $J$-holomorphic map;
thus, $\lr{\om,B}\!\neq\!0$.
Let $H\!\in\!H^2(X;\Z)$ be such $\phi^*H\!=\!-H$ and $\lr{H,B}\!\neq\!0$;
such a class~$H$ can be obtained by slightly deforming~$\om$ so that it represents a rational class,
taking a multiple of the deformed class that represents an integral class,
and then taking the anti-invariant part of the multiple.
Let $h_1$ and~$h_2$ be two pseudocycles as in the statement of the proposition 
representing the Poincare dual of~$H$.
By definition,
$$N_{B,l-l^*}^{\phi;\fp}=\frac{1}{\lr{H,B}^2}
\deg\!\big(\ev_{k,l-l^*+2;(h_1,h_2)},\fo_{\fp;2}\fo_{(h_1,h_2)}\big)\,.$$
An implicit implication of a similar definition in \cite[Section~4]{Jake}
is that $N_{B,l-l^*}^{\phi;\fp}$ does not depend on the choices of~$H$, $h_1$, and~$h_2$.
This follows from~\eref{JakePseudo_e9} below, which also implies~\eref{RdivRel_e}. 

Let $k,l,l^*,B,\bh$ be as in the statement of the proposition and $h'\!:H'\!\lra\!X$ be
another codimension~2 pseudocycle in general position.
We denote by $\bh h'$ the tuple $(h_1,\ldots,h_{l^*},h')$ and show below~that
\BE{JakePseudo_e9}
\deg\!\big(\ev_{k,l+1;\bh h'},\fo_{\fp;l^*+1}\fo_{\bh h'}\big)
=\big(h'\!\cdot_X\!B\big)\deg\!\big(\ev_{k,l;\bh},\fo_{\fp;l^*}\fo_{\bh}\big),\EE
with $\ev_{k,l;\bh}$ as in~\eref{JakePseudo_e} and 
$$\ev_{k,l+1;\bh h'}\!: \cZ_{k,l+1;\bh h'}^{\st}\big(B;J,\ff_{k,l+1;l^*+1}^{\,*}\nu\big)\lra
X_{k,(l+1)-(l^*+1)}\!=\!X_{k,l-l^*}\,.$$

The second forgetful morphism in~\eref{ffMdfn_e} with $(l,i)$ replaced by $(l\!+\!1,l^*\!+\!1)$
induces a morphism
$$\ff\!:\cZ_{k,l+1;\bh h'}^{\st}\big(B;J,\ff_{k,l+1;l^*+1}^{\,*}\nu\big)\lra
 \cZ_{k,l;\bh}^{\st}\big(B;J,\nu\big)$$
so that $\ev_{k,l+1;\bh h'}\!=\!\ev_{k,l;\bh}\!\circ\!\ff$.
The relative orientations $\fo_{\fp;l^*+1}\fo_{\bh h'}$ of $\ev_{k,l+1;\bh h'}$
and $\fo_{\fp;l^*}\fo_{\bh}$ of $\ev_{k,l;\bh}$ determine a relative orientation~$\fo_{\fp}$
of~$\ff$.
The number of the preimages 
$$\wt\u\equiv (\u,(z_{l^*+1}^+,z_{l^*+1}^-),y,y')$$ 
of a generic point 
$$(\u,y)\in \cZ_{k,l;\bh}^{\st}(B;J,\nu)
\cap \big(\M_{k,l}(B;J,\nu)\!\times\!M_{\bh}\big)$$
under~$\ff$ is finite.
For such a preimage~$\wt\u$, $\nd_{\wt\u}\ff$ is an isomorphism. 
With $\u$ as in~\eref{udfn_e}, the homomorphism 
\BE{JakePseudo_e10}
T_{z_{l^*+1}^+}\P^1\!\oplus\!T_{y'}H'\lra T_{u(z_{l^*+1}^+)}X\!=\!T_{h'(y')}X, \quad
(v,w)\lra \nd_{z_{l^*+1}^+}u(v)\!+\!\nd_{y'}h'(w),\EE
is an isomorphism.
Its domain and target are oriented by the complex orientation of~$\P^1$
(i.e.~the vertical orientation~$\fo_{l^*+1}^+$ in the notation of Lemma~\ref{orient_lmm}\ref{fforient_it}),
the given orientation~$\fo_{h'}$ of~$H'$, and the symplectic orientation~$\fo_{\om}$ of~$X$.
We set $\fs_{\wt\u}$ to be $+1$ if this isomorphism 
is orientation-preserving and to be $-1$ if it is orientation-reversing.
We show below that $\fs_{\wt\u}(\fo_{\fp})\!=\!\fs_{\wt\u}$.
Since
\BE{JakePseudo_e11} \sum_{\wt\u\in\ff^{-1}(\u,y)}\!\!\!\!\!\!\!\fs_{\wt\u}=h'\!\cdot_X\!B,\EE
the desired identity~\eref{JakePseudo_e9} then follows from~\eref{fsfoprod_e}.

\begin{figure}
$$\xymatrix{&& 0\ar[d]& 0\ar[d]\\
&0 \ar[r]\ar[d]& T_{\wt\u}\cZ_{\bh h'} \ar[r]^{\nd_{\wt{u}}\ff}\ar[d]& T_{(\u,y)}\cZ_{\bh}\ar[r]\ar[d]& 0\\
0\ar[r]&  T_{z_{l^*+1}^+}\P^1\!\oplus\!T_{y'}H' \ar[r]\ar[d]& 
T_{\wt\u}(\M_{l+1}\!\times\!M_{\bh h'}) \ar[r]\ar[d]& 
T_{(\u,y)}(\M_l\!\times\!M_{\bh})\ar[r]\ar[d]& 0\\
0\ar[r]&  \cN\De \ar[r]\ar[d]& \cN\De^{l^*+1}\ar[r]\ar[d]& 
\cN\De^{l^*}\ar[r]\ar[d]& 0\\  
& 0& 0& 0}$$
\caption{Commutative square of vector spaces with exact rows and columns 
for the proof of~\eref{RdivRel_e}.}
\label{RdivRel_fig}
\end{figure}

Let $\De\!\subset\!X^2$ and $\De^{l^*}\!\subset\!(X^{l^*})^2$ denote the diagonals.
The orientation~$\fo_{\om}$ of~$X$ induces an orientation~$\fo_{\De}$ on 
the normal bundle~$\cN\De$ of~$\De$ and an orientation~$\fo_{\De}^{l^*}$ on 
the normal bundle~$\cN\De^{l^*}$ of~$\De^{l^*}$.
Define 
\begin{alignat*}{2}
\cZ_{\bh}&=\cZ_{k,l;\bh}^{\st}\big(B;J,\nu\big), &\quad
\M_l&=\M_{k,l}\big(B;J,\nu\big),\\
\cZ_{\bh h'}&=\cZ_{k,l+1;\bh h'}^{\st}\big(B;J,\ff_{k,l+1;l^*+1}^{\,*}\nu\big), &\quad
\M_{l+1}&=\M_{k,l+1}\big(B;J,\ff_{k,l+1;l^*+1}^{\,*}\nu\big).
\end{alignat*}
Let $(\u,y)$ and $\wt\u$ be as above.
Fix an orientation~$\fo$ on $T_{\ev(\u)}X_{k,l}$.
The differentials of the obvious maps induce a commutative square in Figure~\ref{RdivRel_fig}
with exact rows and columns.
Since the dimensions of~$X$ and~$H'$ are even, the sign~$\fs_{\wt\u}$ of~\eref{JakePseudo_e10}
is the sign of the isomorphism in the left column with respect to the orientations
$\fo_{l^*+1}^+$, $\fo_h'$, and~$\fo_{\De}$.
Along with the relative orientation $\fo_{\fp;l^*}$ (resp.~$\fo_{\fp;l^*+1}$)
and the orientation~$\fo_{\bh}$ (resp.~$\fo_{\bh h'}$),
$\fo$ induces an orientation~$\fo_{(\u,y)}$ on $T_{(\u,y)}(\M_l\!\times\!M_{\bh})$
(resp.~$\fo_{\wt\u}$ on \hbox{$T_{\wt\u}(\M_{l+1}\!\!\times\!M_{\bh h'})$}).
Since the dimension of~$H'$ is even, Lemma~\ref{orient_lmm}\ref{fforient_it} implies
that the middle row respects the orientations.
Along with the orientation~$\fo_{\De}^{l^*}$ on~$\cN\De^{l^*}$
(resp.~$\fo_{\De}^{l^*+1}$ on~$\cN\De^{l^*+1}$),
$\fo_{(\u,y)}$ (resp.~$\fo_{\wt\u}$) induces an orientation $\fo_{(\u,y)}'$ on~$T_{(\u,y)}\cZ_{\bh}$
(resp.~$\fo_{\wt\u}'$ on $T_{\wt\u}\cZ_{\bh h'}$) so that the right (resp.~middle) column
of the diagram respects the orientations.
The bottom row respects the orientations.
Lemma~\ref{3by3_lmm} then implies that $\nd_{\wt{u}}\ff$ is 
orientation-preserving with respect to~$\fo_{\wt\u}'$ and $\fo_{(\u,y)}'$ if and only if
the isomorphism in the left column~is.
The latter is the case if and only if \hbox{$\fs_{\wt\u}\!=\!+1$}.
These two statements imply that $\fs_{\wt\u}(\fo_{\fp})\!=\!\fs_{\wt\u}$.
\end{proof}

\subsection{Proof of Proposition~\ref{Rdecomp_prp}}
\label{Rdecomp_subs}

For $k',l'\!\in\Z^{\ge0}$ with $k'\!+\!2l'\!\le\!2$,
we denote by $\cH_{k',l'}^{\om,\phi}$ the set of pairs $(J,0)$ with $J\!\in\!\cJ_{\om}^{\phi}$.
We continue with the notation in the statement of this proposition and just above.
Let $k_r,l_r,l_r^*,B_r$ be as in~\eref{JakePseudo_e8} and
$$\M^{\st}=\M_{k,l;l^*}^{\st}(B;J,\nu).$$
Since $(\cS,\Ups)$ is admissible, \hbox{$k_1\!+\!2l_1\!\ge\!3$} and 
either $k_2\!\ge\!1$ or $l_2\!\ge\!1$. 
We assume that there exist \hbox{$\nu_1'\!\in\!\cH_{k_1,l_1}^{\om,\phi}$} and 
$\nu_2\!\in\!\cH_{k_2+1,l_2}^{\om,\phi}$
so that every substratum $\cS_*\!\subset\!\cS$ admits an embedding as in~\eref{cSsplit_e}
with $\nu_1\!=\!\ff_{k_1+1,l_1;\nod}^*\nu_1'$ subject to the conditions specified below~\eref{cSsplit_e2}
and above Lemma~\ref{DorientComp_lmm}.

We first assume that $l_2^*\!\neq\!0$  and
take $i^*\!\in\!L_2^*(\cS)$ to be the smallest element of this set.
By this assumption, the image of~$\cS$ under the forgetful morphism~$\ff_{k,l}$
is contained in a codimension~1 stratum~$\cS^{\vee}$ of~$\ov\cM_{k,l}^{\tau}$.
By Lemma~\ref{orient_lmm2}, we can assume that the orientation~$\fo_{\cS}^c$
of~$\cN\cS$ used to define the relative orientation
$\prt\fo_{\fp;l^*}\!\equiv\!\prt_{\fo_{\cS}^c}\fo_{\fp;l^*}$ of~\eref{evcSdfn_e}
is~$\fo_{\cS}^{c;+}$ in the notation of Lemma~\ref{DorientComp_lmm}.
 
For $\u\in\cS$, let 
\begin{gather*}\begin{aligned}
\u_1\in  \M_1&\equiv\!\M_{k_1+1,l_1}\big(B_1;J,\nu_1\big), &\qquad
\u_1'\in \M_1'&\equiv\!\M_{k_1,l_1}\big(B_1;J,\nu_1'\big),\\
\u_2\in \M_2&\equiv\!\M_{k_2+1,l_2}\big(B_2;J,\nu_2\big), &\qquad
\nod&\in\P^1_1,\P^1_2,~~S^1_1\subset\P^1_1,
\end{aligned}\\
D_{\u}^{\phi}=D_{J,\nu;\u}^{\phi}, \qquad 
D_{\u_1}^{\phi}=D_{J,\nu_1;\u_1}^{\phi}=D_{J,\nu_1';\u_1'}^{\phi},\qquad
D_{\u_2}^{\phi}=D_{J,\nu_2;\u_2}^{\phi}\
\end{gather*}
be as above Lemma~\ref{DorientComp_lmm} and in Section~\ref{cMorient_subs}.
We denote by
\begin{alignat*}{2}
\cC&\equiv\!\ff_{k,l}(\u)\in\cS^{\vee}\subset\ov\cM\!\equiv\!\ov\cM_{k,l}^{\tau}, &\qquad
\cC_1&\equiv\!\ff_{k_1+1,l_1}(\u_1)\in\cM_1\!\equiv\!\cM_{k_1+1,l_1}^{\tau}, \\
\cC_1'&\equiv\!\ff_{k_1,l_1}(\u_1')\in\cM_1'\!\equiv\!\cM_{k_1,l_1}^{\tau},&\qquad
\cC_2&\equiv\!\ff_{k_2+1,l_2}(\u_2)\in\!\cM_2\!\equiv\!\cM_{k_2+1,l_2}^{\tau} 
\end{alignat*}
the marked domains of the maps $\u$, $\u_1$, $\u_1'$, and $\u_2$, respectively.

The exact sequence
\BE{TcSses_e}
0\lra T_{\u}\cS\lra T_{\u_1}\M_1\!\oplus\!T_{\u_2}\M_2\lra T_{u(\nod)}X^{\phi}\lra0, ~~ 
\big(\xi_1,\xi_2\big)\lra \xi_2(\nod)\!-\!\xi_1(\nod),\EE
of vector spaces determines an isomorphism 
\BE{Rdecomppf_e3}
\la_{\u}(\cS)\!\otimes\!\la\big(T_{u(\nod)}X^{\phi}\big)
\approx \la_{\u_1}(\M_1)\!\otimes\!\la_{\u_2}(\M_2).\EE
Since $\ep_{l^*}(\cS)\!\in\!2\Z$, 
the $\Pin^-$-structure~$\fp$ on~$X^{\phi}$ determines homotopy classes~$\fo_{\fp;l_1^*}$
and~$\fo_{\fp;l_2^*}$ of isomorphisms
\BE{Rdecomppf_e3b}\la_{\u_1'}(\M_1')\lra \la_{\u_1'}^{\R}(X)\!\otimes\!\la_{\u_1'}^{\C}(X) 
\quad\hbox{and}\quad
\la_{\u_2}(\M_2)\lra  \la_{\u_2}^{\R}(X)\!\otimes\!\la_{\u_2}^{\C}(X),\EE
respectively; see Lemma~\ref{orient_lmm}.
Combining the first homotopy class of isomorphisms above with 
the first $S^1$-fibration in~\eref{cSsplit_e2}
and the orientation~$\fo_{\nod}^{\R}$ on its vertical tangent bundle 
\hbox{$T_{\u_1}\M_1^v\!=\!T_{\nod}S^1_1$},
we obtain a homotopy class $\wt\fo_{\fp;l_1^*}\!\equiv\!\fo_{\nod}^{\R}\fo_{\fp;l_1^*}$ 
of isomorphisms
\BE{Rdecomppf_e5a}
\la_{\u_1}(\M_1) \approx \la_{\u_1'}(\M_1')\!\otimes\!T_{\nod}S^1_1
\approx  \la_{\u_1'}^{\R}(X)\!\otimes\!\la_{\u_1'}^{\C}(X)\,.\EE
Along with~\eref{Rdecomppf_e3} and the second homotopy class of isomorphisms in~\eref{Rdecomppf_e3b},
it determines a homotopy class of isomorphisms
\BE{Rdecomppf_e5}\begin{split}
\la_{\u}(\cS)\!\otimes\!\la\big(T_{u(\nod)}X^{\phi}\big)
&\approx \la_{\u_1'}^{\R}(X)\!\otimes\!\la_{\u_1'}^{\C}(X)\!\otimes\!
\la_{\u_2}^{\R}(X)\!\otimes\!\la_{\u_2}^{\C}(X)\\
&\approx \la_{\u_1'}^{\R}(X)\!\otimes\!\la_{\u_1'}^{\C}(X)\!\otimes\!
\la\big(T_{u(\nod)}X^{\phi}\big)\!\otimes\!
\la_{\u_2'}^{\R}(X)\!\otimes\!\la_{\u_2'}^{\C}(X)\\
&\approx \la_{\u}^{\R}(X)\!\otimes\!\la_{\u}^{\C}(X)\!\otimes\!\la\big(T_{u(\nod)}X^{\phi}\big)\,.
\end{split}\EE
We denote by~$\fo_{\fp;l^*;\u}^{\cS}$ the homotopy class of isomorphisms
$$\la_{\u}(\cS)\lra \la_{\u}^{\R}(X)\!\otimes\!\la_{\u}^{\C}(X)$$
determined by~\eref{Rdecomppf_e5}.
The next lemma is deduced from Lemmas~\ref{DMboundary_lmm}  and~\ref{DorientComp_lmm} 
at the end of this section.

\begin{lmm}\label{TMcomp_lmm}
The orientations $\prt\fo_{\fp;l^*}$ and $\fo_{\fp;l^*}^{\cS}$ of $\la(\ev|_{\cS})$
are opposite.
\end{lmm}

We take $\bh_1$ and~$\bh_2$ to be the components of~$\bh$ as in the proof 
of Proposition~\ref{JakePseudo_prp} and
$$\bp_1\in X_{k_1,l_1-l_1^*} \qquad\hbox{and}\qquad \bp_2\in X_{k_2,l_2-l_2^*} $$
to be the components of $\bp\!\in\!X_{k,l-l^*}$ defined analogously.
Let
\begin{gather*}
\cZ_1=\cZ_{k_1+1,l_1;\bh_1}^{\st}(B_1;J,\nu_1)\!\cap\!\big(\M_1\!\times\!M_{\bh_1}\big), 
\quad
\cZ_1'=\cZ_{k_1,l_1;\bh_1}^{\st}(B_1;J,\nu_1')\!\cap\!\big(\M_1'\!\times\!M_{\bh_1}\big),\\
\cZ_2=\cZ_{k_2+1,l_2;\bh_2}^{\st}(B_2;J,\nu_2)\!\cap\!\big(\M_2\!\times\!M_{\bh_2}\big).
\end{gather*}
The first forgetful morphism in~\eref{cSsplit_e2} induces a fibration~$\ff_{\cZ_1}$
so that the diagram 
$$\xymatrix{ \cZ_1 \ar[rr]^{\ff_{\cZ_1}} \ar[d]_{\pi_{\cZ}}&& \cZ_1'\ar[d]^{\pi_{\cZ'}}\\
\M_1 \ar[rr]^{\ff_{\nod}} && \M_1'}$$
commutes.
Since $\pi_{\cZ}$ induces an isomorphism between the vertical tangent bundles
$T\cZ_1^v$ of~$\ff_{\cZ_1}$ and $T\M_1^v$ of~$\ff_{\nod}$, 
it pulls back $\fo_{\nod}^{\R}$ to an orientation~$\fo_{\cZ_1}^v$ on the fibers of~$\ff_{\cZ_1}$.
The relative orientations $\fo_{\nod}^{\R}\fo_{\fp;l_1^*}$, $\fo_{\fp;l_1^*}$, and $\fo_{\fp;l_2^*}$ 
on
$$\ev\!:\M_1\lra X_{k_1,l_1}, \quad \ev'\!:\M_1'\lra X_{k_1,l_1}, 
\quad\hbox{and}\quad \ev\!:\M_2\lra X_{k_2+1,l_2}\,,$$
respectively, the orientations $\fo_{h_i}$ of $H_i$, and 
the symplectic orientation~$\fo_{\om}$ on~$X$ determine 
relative orientations $\wt\fo_{\fp;\bh_1}$, $\fo_{\fp;\bh_1}$, and $\fo_{\fp;\bh_2}$
of
$$\ev_{\bh_1}\!:\cZ_1\lra X_{k_1,l_1-l_1^*}, \quad \ev_{\bh_1}'\!:\cZ_1'\lra X_{k_1,l_1-l_1^*}, 
\quad\hbox{and}\quad \ev_{\bh_2}\!:\cZ_2\lra X_{k_2+1,l_2-l_2^*}\,,$$
respectively.
Since the dimensions of $X$ and $H_i$ are even, 
\BE{Rdecomppf_e8} \wt\fo_{\fp;\bh_1}=\fo_{\cZ_1}^v\fo_{\fp;\bh_1}\equiv
\big(\pi_{\cZ}^*\fo_{\nod}^{\R}\big)\fo_{\fp;\bh_1}\,.\EE

For $\wt\u\!\in\!\cS_{\bh}^*$, we denote by 
$$\wt\u_1\in\cZ_1, \qquad  \wt\u_1'\in\cZ_1', \qquad  \wt\u_2\in\cZ_2$$
the images of~$\wt\u$ under the projections induced by the embedding~\eref{cSsplit_e},
the first forgetful morphism in~\eref{cSsplit_e2},  and the decomposition 
$$M_{\bh}\approx M_{\bh_1}\!\times\!M_{\bh_2}\,.$$
The exact sequence
$$0\lra T_{\wt\u}\cS_{\bh}^*\lra T_{\wt\u_1}\cZ_1\!\oplus\!T_{\wt\u_2}\cZ_2\lra 
T_{u(\nod)}X^{\phi}\lra0, ~~ 
\big(\xi_1,\xi_2\big)\lra \xi_2(\nod)\!-\!\xi_1(\nod),$$
of vector spaces determines an isomorphism 
$$\la_{\wt\u}(\cS_{\bh}^*)\!\otimes\!\la\big(T_{u(\nod)}X^{\phi}\big)
\approx \la_{\wt\u_1}(\cZ_1)\!\otimes\!\la_{\wt\u_2}(\cZ_2).$$
Along with the relative orientations $\wt\fo_{\fp;\bh_1}$ and $\fo_{\fp;\bh_2}$ above,
this isomorphism determines a homotopy class of isomorphisms
\begin{equation*}\begin{split}
\la_{\wt\u}(\cS_{\bh}^*)\!\otimes\!\la\big(T_{u(\nod)}X^{\phi}\big)
&\approx \la_{\wt\u_1'}^{\R}(X)\!\otimes\!\la_{\wt\u_1'}^{\C}(X)\!\otimes\!
\la_{\wt\u_2}^{\R}(X)\!\otimes\!\la_{\wt\u_2}^{\C}(X)\\
&\approx \la_{\wt\u_1'}^{\R}(X)\!\otimes\!\la_{\wt\u_1'}^{\C}(X)\!\otimes\!
\la\big(T_{u(\nod)}X^{\phi}\big)\!\otimes\!
\la_{\wt\u_2'}^{\R}(X)\!\otimes\!\la_{\wt\u_2'}^{\C}(X)\\
&\approx \la_{\wt\u}^{\R}(X)\!\otimes\!\la_{\wt\u}^{\C}(X)\!\otimes\!
\la\big(T_{u(\nod)}X^{\phi}\big).
\end{split}\end{equation*}
We denote the associated relative orientation of~\eref{evcSdfn_e} by 
\BE{Rdecomppf_e11a}\fo_{\cS;\bh}\equiv(\wt\fo_{\fp;\bh_1})_{\nod}\!\!\times\!_{\nod}\fo_{\fp;\bh_2}.\EE
Since the dimensions of $X$ and $H_i$ are even, 
Lemma~\ref{TMcomp_lmm} implies~that 
\BE{Rdecomppf_e11} \big|\cS_{\bh,\bp;\Ups}^*\big|_{\prt\fo_{\fp;\bh},\fo_\Ups^c}^{\pm}
=-\big|\cS_{\bh,\bp;\Ups}^*\big|_{\fo_{\cS;\bh},\fo_\Ups^c}^{\pm}\,.\EE

If $\cS$ and $\Ups$ satisfy~\ref{cSUps_it1} above Proposition~\ref{Rdecomp_prp}
with $i\!\in\![k']$ as in~\ref{cSUps_it1}, let
$$k_1'=k'\!-\!1, \quad l_1'=l', \quad k_2'=1, \quad l_2'=0, \quad \Ups_1=\Ups,
\quad \nod=i\,.$$
If $\cS$ and $\Ups$ satisfy~\ref{cSUps_it2} and $S\!\subset\!\ov\cM_{k',l'}^{\tau}$ 
as in~\ref{cSUps_it2}, let
$$k_1'=k_1(S), \quad l_1'=l_1(S), \quad k_2'=k_2(S), \quad l_2'=l_2(S)$$
and denote by 
$$\pi_1\!: S\!\approx\!\ov\cM_{k_1'+1,l_1'}^{\tau}\!\times\!\ov\cM_{k_2'+1,l_2'}^{\tau}
\lra \ov\cM_{k_1'+1,l_1'}^{\tau}$$
the projection to the first component in the second identification in~\eref{Ssplit_e0}.
In this case,
$$\Ups\!\cap\!\ov{S}\approx\Ups_1\!\times\!\ov\cM_{k_2'+1,l_2'}^{\tau}
\subset \ov\cM_{k_1'+1,l_1'}^{\tau}\!\times\!\ov\cM_{k_2'+1,l_2'}^{\tau}$$
for some $\Ups_1\!\subset\!\ov\cM_{k_1'+1,l_1';\nod}^{\tau;\st}$.
The co-orientation $\fo_{\Ups\cap S}^c$ on 
$\Ups\!\cap\!\ov{S}$ in~$\ov{S}$ induced by~$\fo_{\Ups}^c$ is the pullback by~$\pi_1$
of a co-orientation $\fo_{\Ups_1}^c$ on $\Ups_1$ in~$\ov\cM_{k_1'+1,l_1'}^{\tau}$.
Let
$$\ff_{k_1'+1,l_1'}\!=\!\pi_1\!\circ\!\ff_{k',l'}\!:
\cS_{\bh}^*\lra S\lra \ov\cM_{k_1'+1,l_1'}^{\tau}\,.$$
In both cases, 
\BE{Rdecomppf_e9}\dim\,\Ups_1=\dim\,\Ups\!+\!1\!-\!k_2'\!-\!2l_2'\EE
and the forgetful morphism $\ff_{k_1'+1,l_1'}$ factors as 
$$\cS_{\bh}^* \lhra{~~~} \cZ_1\!\times\!\cZ_2\lra \cZ_1 
\xlra{\ff_{k_1'+1,l_1'}}\ov\cM_{k_1'+1,l_1'}^{\tau}\,.$$
We define
$$\ff_{\cM}\!=\!\ff_{k_1'+1,l_1';\nod}\!:
\ov\cM_{k_1'+1,l_1'}^{\tau}\lra\ov\cM_{k_1',l_1'}^{\tau}\,.$$

If $\cS$ and $\Ups$ satisfy~\ref{cSUps_it2}, 
\ref{dropfactor_it2} and~\ref{dropfactor_it3} in Lemma~\ref{dropfactor_lmm} give
\BE{Rdecomppf_e21}\begin{split}
\big|\cS_{\bh,\bp;\Ups}^*\big|_{\fo_{\cS;\bh},\fo_\Ups^c}^{\pm}
&=-\big|M_{(\ev_{\cS;\bh},\ff_{k',l'}),f_{\bp;\Ups}|_{\Ups\cap\ov{S}}}
\big|_{\fo_{\cS;\bh},\pi_1^*\fo_{\Ups_1}^c|_{\Ups\cap\ov{S}}}^{\pm}\\
&=-(-1)^{k_2'}\big|M_{(\ev_{\cS;\bh},\ff_{k_1'+1,l_1'}),f_{\bp;\Ups_1}}
\big|_{\fo_{\cS;\bh},\fo_{\Ups_1}^c}^{\pm};
\end{split}\EE
the signed fiber products in the second and third expressions above are taken 
with respect to
$X_{k,l-l^*}\!\times\!S$ and \hbox{$X_{k,l-l^*}\!\times\!\ov\cM_{k_1'+1,l_1'}^{\tau}$}, 
respectively.
The first and last expressions in~\eref{Rdecomppf_e21} are the same if 
$\cS$ and $\Ups$ satisfy~\ref{cSUps_it1}.
By~\eref{Rdecomppf_e11a} and Lemma~\ref{dropfactor_lmm3}, 
\BE{Rdecomppf_e23}
\big|M_{(\ev_{\cS;\bh},\ff_{k_1'+1,l_1'}),f_{\bp;\Ups_1}}
\big|_{\fo_{\cS;\bh},\fo_{\Ups_1}^c}^{\pm}
=\big|M_{(\ev_{\bh_1},\ff_{k_1'+1,l_1'}),f_{\bp_1;\Ups_1}}
\big|_{\wt\fo_{\fp;\bh_1},\fo_{\Ups_1}^c}^{\pm}
\!\!\deg\!\big(\ev_{\bh_2},\fo_{\fp;\bh_2}\!\big).\EE
By Lemma~\ref{dropfactor_lmm}\ref{dropfactor_it1},
$$\big|M_{(\ev_{\bh_1},\ff_{k_1'+1,l_1'}),f_{\bp_1;\Ups_1}}
\big|_{\wt\fo_{\fp;\bh_1},\fo_{\Ups_1}^c}^{\pm}=
(-1)^{\dim\,\Ups_1}\deg\!\big(\ev_{\bh_1}|_{\ff_{k_1'+1,l_1'}^{-1}\!(\Ups_1)},
(\ff_{k_1'+1,l_1'}^*\fo_{\Ups_1}^c)\wt\fo_{\fp;\bh_1}\big)\,.$$
By the first identity in~\eref{Rdecomppf_e8} and~\eref{fsfoprod_e}, 
$$\deg\!\big(\ev_{\bh_1}|_{\ff_{k_1'+1,l_1'}^{-1}\!(\Ups_1)},
(\ff_{k_1'+1,l_1'}^*\fo_{\Ups_1}^c)\wt\fo_{\fp;\bh_1}\big)
=\deg\!\big(\ff_{\cZ_1}|_{\ff_{k_1'+1,l_1'}^{-1}\!(\Ups_1)},
(\ff_{k_1'+1,l_1'}^*\fo_{\Ups_1}^c)\fo_{\cZ_1}^v\big)
\deg\!\big(\ev_{\bh_1}',\fo_{\fp;\bh_1}\big).$$
By the second identity in~\eref{Rdecomppf_e8} and Lemma~\ref{fibrasign_lmm1b},
\begin{equation*}\begin{split}
\fs_{\wt\u}\!\big((\ff_{k_1'+1,l_1'}^*\fo_{\Ups_1}^c)\fo_{\cZ_1}^v\big)
=\fs_{\wt\u}\!\big(\ff_{k_1'+1,l_1'},\pi_{\cZ}^*\fo_{\nod}^{\R},\fo_{\nod}^{\R}\big)
\fs_{\ff_{k_1'+1,l_1'}(\wt\u)}\!\big(\fo_{\Ups_1}^c\fo_{\nod}^{\R}\big)
=\fs_{\ff_{k_1'+1,l_1'}(\wt\u)}\!\big(\fo_{\Ups_1}^c\fo_{\nod}^{\R}\big)
\end{split}\end{equation*}
for a generic $\wt\u\!\in\!\ff_{k_1'+1,l_1'}^{-1}\!(\Ups_1)$.

Combining the last three equations with~\eref{Rdecomppf_e9} and
Lemma~\ref{fibrasign_lmm1a}\ref{fibisom_it2}, we obtain
\begin{equation*}\begin{split}
\big|M_{(\ev_{\bh_1},\ff_{k_1'+1,l_1'}),f_{\bp_1;\Ups_1}}
\big|_{\wt\fo_{\fp;\bh_1},\fo_{\Ups_1}^c}^{\pm}
&=-(-1)^{\dim\,\Ups+k_2'}\deg\!\big(\ff_{\cM}|_{\Ups_1},\fo_{\Ups_1}^c\fo_{\nod}^{\R}\big)
\deg\!\big(\ev_{\bh_1}',\fo_{\fp;\bh_1}\big)\\
&=-(-1)^{\dim\,\Ups+k_2'}\deg_S\!\big(\Ups,\fo_{\Ups}^c\big)
\deg\!\big(\ev_{\bh_1}',\fo_{\fp;\bh_1}\big).
\end{split}\end{equation*}
Along with~\eref{Rdecomppf_e21} and~\eref{Rdecomppf_e23}, this gives
$$\big|\cS_{\bh,\bp;\Ups}^*\big|_{\fo_{\cS;\bh},\fo_\Ups^c}^{\pm}
=(-1)^{\dim\,\Ups}\deg_S\!\big(\Ups,\fo_{\Ups}^c\big)
\deg\!\big(\ev_{\bh_1}',\fo_{\fp;\bh_1}\big)
\deg\!\big(\ev_{\bh_2},\fo_{\fp;\bh_2}\!\big)\,.$$
Combining this equation with~\eref{Rdecomppf_e11} and~\eref{RdivRel_e}, 
we obtain~\eref{Rdecomp_e}.

Suppose $l_2^*\!=\!0$.
Choose another codimension~2 pseudocycle $h'\!:H'\!\lra\!X$  in general position
with $h'\!\cdot_X\!B_2\!\neq\!0$. 
Let
$$\cS_{\bh h'}^*\subset \cZ_{k,l+1;\bh h'}^{\st}\big(B;J,\ff_{k,l+1;l^*+1}^{\,*}\nu\big)$$
be the codimension~1 stratum so that $\ff_{k,l+1;l^*+1}(\cS_{\bh h'}^*)\!=\!\cS_{\bh}^*$
and the irreducible component~$\P^1_2$ of the maps in this stratum carries the marked points
$z_{l^*+1}^{\pm}$.
The vertical tangent bundle of the projection 
$$\ff_{k,l+1;l^*+1}\!: \cZ_{k,l+1;\bh}^{\st}\big(B;J,\ff_{k,l+1;l^*+1}^{\,*}\nu\big)
\lra \cZ_{k,l;\bh}^{\st}\big(B;J,\nu\big)$$
is oriented by the position of $z_{l^*+1}^{\pm}$; 
we denote this orientation by~$\fo_{l^*}^+$.

By the $l_2^*\neq\!0$ case above,
\BE{Rdecomppf_e33}
\big|\cS_{\bh h',\bp;\Ups}^*\big|_{\prt\fo_{\fp;\bh h'},\fo_\Ups^c}^{\pm}
=-(-1)^{\dim\,\Ups}\deg_S\!\big(\Ups,\fo_{\Ups}^c\big)
\deg\!\big(\ev_{\bh_1}',\fo_{\fp;\bh_1}\big)
\big(h'\!\cdot_X\!B_2\big)N_{B_2,l_2}^{\phi;\fp}\,.\EE
Since the dimensions of~$X$ and~$H'$ are even,
$$\prt\fo_{\fp;\bh h'}=\prt\big(\fo_{l^*}^+\fo_{\fp;\bh}\fo_{h'}\big)
=\fo_{l^*}^+\big(\prt\fo_{\fp;\bh}\big)\fo_{h'}\,.$$
By the reasoning in the proof of~\eref{RdivRel_e}, this implies that 
\BE{Rdecomppf_e35}
\big|\cS_{\bh h',\bp;\Ups}^*\big|_{\prt\fo_{\fp;\bh h'},\fo_\Ups^c}^{\pm}
=\big(h'\!\cdot_X\!B_2\big)
\big|\cS_{\bh,\bp;\Ups}^*\big|_{\prt\fo_{\fp;\bh},\fo_\Ups^c}^{\pm}.\EE
Combining~\eref{Rdecomppf_e33} and~\eref{Rdecomppf_e35} with~\eref{RdivRel_e},
we again obtain~\eref{Rdecomp_e}.

\begin{proof}[{\bf{\emph{Proof of Lemma~\ref{TMcomp_lmm}}}}]
Fix orientations of $T_{u(\nod)}X^{\phi}$ and $T_{u(x_i)}X^{\phi}$ for all $i\!\in\![k]$.
We can then view all relevant relative orientations as orientations in the usual sense.
Let $\cS_*\!\subset\!\cS$ be the substratum containing~$\u$.

The differential of the forgetful morphism $\ff_{k,l}$ induces 
the first exact square of Figure~\ref{TMcomp_fig}.
The two spaces in the bottom row are oriented by~$\fo_{\cS}^{c;+}$ and~$\fo_{\cS^{\vee}}^{c;+}$
with the isomorphism between them being orientation-preserving.
These orientations and the orientations~$\fo_{\fp}^D$, $\fo_{\fp;l^*}$, and $\fo_{k,l;l^*}$
determine the limiting orientations $\fo_{\fp}^{D;+}$ on $\ker D_{\u}^{\phi}$,
$\fo_{\fp;l^*}^+$ on $T_{\u}\M^{\st}$, and $\fo_{k,l;l^*}^+$ on $T_{\cC}\ov\cM$, respectively.
By~\eref{OrientSubs_e3}, the middle row respects these orientations.
The middle (resp.~right) column respects the orientations 
$\prt\fo_{\fp;l^*}$ on~$T_{\u}\cS$, $\fo_{\fp;l^*}^+$ on $T_{\u}\M^{\st}$, and
$\fo_{\cS}^{c;+}$ on~$\cN_{\u}\cS$
(resp.~$\fo_{\cS^{\vee};l^*}^+$ on~$T_{\u}\cS^{\vee}$, 
$\fo_{k,l;l^*}^+$ on $T_{\cC}\ov\cM$, and
$\fo_{\cS^{\vee}}^{c;+}$ on~$\cN_{\u}\cS^{\vee}$).
Lemma~\ref{3by3_lmm} then implies that the top row in 
the first exact square of Figure~\ref{TMcomp_fig}
{\it respects} the orientations
$\fo_{\fp}^{D;+}$ on $\ker D_{\u}^{\phi}$,
$\prt\fo_{\fp;l^*}$ on~$T_{\u}\cS$, and $\fo_{\cS^{\vee};l^*}^+$ on~$T_{\u}\cS^{\vee}$.

\begin{figure}
$$\xymatrix{ &0\ar[d] &0\ar[d] &0\ar[d]\\
0\ar[r]& \ker D_{\u}^{\phi} \ar[r]\ar@{=}[d]&
T_{\u}\cS\ar[r]\ar[d]& T_{\cC}\cS^{\vee}\ar[r]\ar[d]& 0\\
0\ar[r]& \ker D_{\u}^{\phi} \ar[r]\ar[d]&
T_{\u}\M^{\st}\ar[r]\ar[d]& T_{\cC}\ov\cM\ar[r]\ar[d]& 0\\
&0\ar[r]& \cN_{\u}\cS \ar[r]\ar[d]& \cN_{\cC}\cS^{\vee}\ar[r]\ar[d]& 0\\
&&0&0}$$
$$\xymatrix{ &&0\ar[d] &0\ar[d]\\
&0\ar[d]\ar[r]& T_{\nod}S^1_1\ar[d]\ar[r]& T_{\nod}S^1_1\ar[d]\ar[r]& 0\\
0\ar[r]& \ker D_{\u_1}^{\phi} \ar[r]\ar@{=}[d]&T_{\u_1}\M_1 \ar[r]\ar[d]& 
T_{\cC_1}\cM_1 \ar[r]\ar[d]& 0\\
0\ar[r]& \ker D_{\u_1}^{\phi} \ar[r]\ar[d]&T_{\u_1'}\M_1' \ar[r]\ar[d]& 
T_{\cC_1'}\cM_1' \ar[r]\ar[d]& 0\\
&0&0&0}$$
$$\xymatrix{ &0\ar[d] &0\ar[d] &0\ar[d]\\
0\ar[r]& \ker D_{\u}^{\phi} \ar[r]\ar[d]&
T_{\u}\cS\ar[r]\ar[d]& T_{\cC}\cS^{\vee}\ar[r]\ar[d]& 0\\
0\ar[r]& \ker D_{\u_1}^{\phi}\!\oplus\!\ker D_{\u_2}^{\phi} \ar[r]\ar[d]&
T_{\u_1}\M_1\!\oplus\!T_{\u_2}\M_2 \ar[r]\ar[d]& 
T_{\cC_1}\cM_1\!\oplus\!T_{\cC_2}\cM_2 \ar[r]\ar[d]& 0\\
0\ar[r]& T_{u(\nod)}X^{\phi} \ar@{=}[r]\ar[d]& T_{u(\nod)}X^{\phi}\ar[r]\ar[d]& 0\\
&0&0}$$
\caption{Commutative squares of vector spaces with
exact rows and columns for the proof of Lemma~\ref{TMcomp_lmm}}
\label{TMcomp_fig}
\end{figure}

The differentials of forgetful morphisms induce the second exact square of Figure~\ref{TMcomp_fig}.
The two spaces in the top row are oriented by~$\fo_{\nod}^{\R}$ as in Section~\ref{cMorient_subs} 
with the isomorphism between them being orientation-preserving.
The first real marked point of~$\u_1$ is the node, while the second one (if $k_1\!\neq\!0$)
is the next marked point on the fixed locus $S^1_1\!\subset\!\P^1_1$ in the counterclockwise
direction with respect to~$z_1^+$.
By~\ref{cMorientR_it} and~\ref{cM1Rch_it} in Lemma~\ref{cMorient_lmm}, 
the right column thus does not respect the orientations
$\fo_{\nod}^v$ on $T_{\nod}S^1_1$, $\fo_{k_1+1,l_1;l_1^*}$ on~$T_{\cC_1}\cM_1$,
and $\fo_{k_1,l_1;l_1^*}$ on~$T_{\cC_1'}\cM_1'$ because 
$$k_1 +\dim\,\cM_1'=2k_1\!+\!2l_1\!-\!3 \not\in2\Z.$$
By~\eref{OrientSubs_e3},
the bottom row respects the orientations $\fo_{\fp}^D$ on $\ker D_{\u_1}^{\phi}$,
$\fo_{\fp;l_1^*}$ on $T_{\u_1'}\M_1'$, and $\fo_{k_1,l_1;l_1^*}$ on~$T_{\cC_1'}\cM_1'$.
By~\eref{Rdecomppf_e5a}, 
the middle column respects the orientations $\fo_{\nod}^v$ on $T_{\nod}S^1_1$,
$\wt\fo_{\fp;l_1^*}$ on $T_{\u_1}\M_1$, and $\fo_{\fp;l_1^*}$ on $T_{\u_1'}\M_1'$ because 
$$\dim\,\M_1'=\ell_{\om}(B_1)\!+\!2l_1\!+\!k_1$$
is even by the assumption that $\ep_{l^*}(\cS)\!=\!2$.
Lemma~\ref{3by3_lmm} then implies that the middle row respects the orientations
$\fo_{\fp}^D$ on $\ker D_{\u_1}^{\phi}$,
$\wt\fo_{\fp;l_1^*}$ on $T_{\u_1}\M_1$, and $\fo_{k_1+1,l_1;l_1^*}$ on~$T_{\cC_1}\cM_1$
if and only~if
$$1+\dim \ker D_{\u_1}^{\phi}=1\!+\!2\!+\!\lr{c_1(X,\om),B_1}$$
is even. 
Since $\ep_{l^*}(\cS)\!=\!2$, we conclude that the middle row 
in the second exact square of Figure~\ref{TMcomp_fig} respects the orientations
$\fo_{\fp}^D$ on $\ker D_{\u_1}^{\phi}$,
$\wt\fo_{\fp;l_1^*}$ on $T_{\u_1}\M_1$, and $\fo_{k_1+1,l_1;l_1^*}$ on~$T_{\cC_1}\cM_1$
if and only if $k_1$ is even.

The short exact sequences~\eref{Dses_e} and~\eref{TcSses_e} and 
the differential of the forgetful morphism~$\ff_{k,l}$
induce the third exact square of Figure~\ref{TMcomp_fig}.
By~\eref{OrientSubs_e3}, the short exact sequence of the second summands in the middle row respects
the orientations $\fo_{\fp}^D$ on $\ker D_{\u_2}^{\phi}$,
$\fo_{\fp;l_2^*}$ on $T_{\u_2}\M_2$, and $\fo_{k_2+1,l_2;l_2^*}$ on~$T_{\cC_2}\cM_2$.
Along with the conclusion of the previous paragraph and Lemma~\ref{3by3_lmm},
this implies that the middle row respects the orientations 
$\fo_{\fp}^D\!\oplus\!\fo_{\fp}^D$,  $\wt\fo_{\fp;l_1^*}\!\oplus\!\fo_{\fp;l_2^*}$,
and $\fo_{k_1+1,l_1;l_1^*}\oplus\!\fo_{k_2+1,l_2;l_2^*}$ if and only~if 
$$k_1+\big(\dim\ker D_{\u_2}^{\phi}\big)\big(\dim\,\cM_1\big)
=k_1\!+\!\big(2\!+\!\lr{c_1(X,\om),B_2}\big)\big(k_1\!+\!2l_1\!-\!2\big)$$
is even.
Since $\ep_{l^*}(\cS)\!=\!2$, this is the case if and only if $k_1\!+\!k_1k_2\!\in\!2\Z$.
By Lemma~\ref{DorientComp_lmm}, the left column respects the orientations 
$\fo_{\fp}^{D;+}$, $\fo_{\fp}^D\!\oplus\!\fo_{\fp}^D$, and
the chosen orientation~$\fo_{\nod}^X$ on~$T_{u(\nod)}X^{\phi}$ if and only~if
\hbox{$\de_D^+(\cS)\!\cong\!\de_{\R}^+(\cS_*)$} mod~2.
Since $\ep_{l^*}(\cS)\!=\!2$, this is the case if and only~if the number
\hbox{$k_2\!+\!k_1k_2\!+\!\de_{\R}^+(\cS_*)$} is even.
By Lemma~\ref{DMboundary_lmm}, the non-trivial isomorphism in the right column respects 
the orientations $\fo_{\cS^{\vee};l^*}^+$ and $\fo_{k_1+1,l_1;l_1^*}\oplus\!\fo_{k_2+1,l_2;l_2^*}$ 
if and only~if $\de_{\R}^+(\cS_*)\!\cong\!k\!+\!1$ mod~2.
By~\eref{Rdecomppf_e5}, the middle column respects the orientations $\fo_{\fp;l^*}^{\cS}$, 
$\wt\fo_{\fp;l_1^*}\!\oplus\!\fo_{\fp;l_2^*}$, and~$\fo_{\nod}^X$.
Combining these statements with Lemma~\ref{3by3_lmm},
we conclude that the top row does {\it not} respect the orientations 
$\fo_{\fp}^{D;+}$, $\fo_{\fp;l^*}^{\cS}$, and $\fo_{\cS^{\vee};l^*}^+$ because
$$(k_1\!+\!k_1k_2)\!+\!\big(k_2\!+\!k_1k_2\!+\!\de_{\R}^+(\cS_*)\big)\!+\!
\big(k\!+\!1\!+\!\de_{\R}^+(\cS_*)\big)\!+\!\big(\dim\,\cS^{\vee}\big)\big(\dim\,X^{\phi}\big)
=1 \mod2.$$
Comparing this conclusion with the conclusion concerning the top row  
in the first exact square of Figure~\ref{TMcomp_fig} above, we obtain the claim.
\end{proof}

\subsection{Proof of Proposition~\ref{Cdecomp_prp}}
\label{Cdecomp_subs}

We denote by $\cH_{0,m}^{\om}$ the space of pairs $(J,\nu')$ 
consisting of $J\!\in\!\cJ_{\om}$ and a Ruan-Tian perturbation~$\nu'$ of 
the $\dbar_J$-equation if $m\!\ge\!3$ and take $\cH_{0,2}^{\om}$ to be the set of
pairs $(J,0)$ with $J\!\in\!\cJ_{\om}$.
For $B'\!\in\!H_2(X)$ and $\nu'\!\in\!\cH_{0,m}^{\om}$, we denote by $\M^\C_m(B';J,\nu')$
the moduli space of (complex) genus~0 degree~$B'$ $(J,\nu')$-holomorphic maps 
from smooth domains with $m$~marked points and~by
$$\ev_i\!: \M^\C_m(B';J,\nu')\lra X, \qquad i\!\in\![m],$$
the evaluation maps at the marked points.
For $I\!\subset\![m]$, let $\fo_{\C;I}$ be the orientation of $\M^\C_m(B';J,\nu')$ 
obtained by twisting the standard complex orientation by~$(-1)^{|I|}$.
Define 
\begin{alignat*}{2}
\Th_i^I\!:X&\lra X,&\quad \Th_i^I&=\begin{cases}\id_X,&\hbox{if}~i\!\not\in\!I;\\
\phi,&\hbox{if}~i\!\in\!I;\end{cases}\\
\ev^I\!:\M^\C_m(B';J,\nu')&\lra X^m, &\quad 
\ev^I(\u)&=\big(\big(\Th^I_i(\ev_i(\u))\big)_{i\in[m]}\big).
\end{alignat*}

We continue with the notation in the statement of Proposition~\ref{Cdecomp_prp} 
and just above.
The co-orientation~$\fo_{\Ga}^c$ of~$\Ga$ in~$\ov\cM_{k',l'}^{\tau}$ and
the relative orientation~$\fo_{\fp;l^*}$ of Lemma~\ref{orient_lmm} induce
a relative orientation $(\ff^*_{k',l'}\fo^c_\Ga)\fo_{\fp;l^*}$ of the restriction 
$$\ev_{\Ga}\!:\M_{\Ga;k,l}(B;J,\nu)\lra X_{k,l}$$
of ~\eref{whevdfn_e}.

Fix a stratum $\cS\!\subset\!\M_{\Ga;k,l}(B;J,\nu)$. 
Let $B_\R$ be the degree of the restrictions of the maps in~$\cS$ to the real component $\P^1_0$ 
of the domain and $B_\C$ be the degree of their restrictions to the component~$\P^1_+$ 
of the domain carrying the marked point~$z^+_1$. 
Denote by $L_0,L_\C\!\subset\![l]$ the subsets indexing the conjugate pairs 
of marked points carried by~$\P^1_0$ and $\P^1_+$, respectively. 
Let $L^*_-\!\subset\!L^*_\C$ be the subset indexing the conjugate pairs 
of marked points $(z^+_i,z^-_i)$ of curves in~$\Ga$ so that  
$z^-_i$ lies on~$\P^1_+$.
Define
$$L^*_0=L^*_0(\Ga),\quad L^*_\C= L^*_\C(\Ga), \quad l_0=|L_0|,\quad
l_\C=|L_\C|, \quad l^*_0=l^*_0(\Ga), \quad l^*_\C\equiv l^*\!-\!l^*_0.$$
By~\eref{Cdecomp_e0}, 
\BE{Cdecomp_e3} (l_0\!-\!l_0^*)\!+\!(l_{\C}\!-\!l_{\C}^*)=l\!-\!l^*, \quad
\big(\ell_{\om}(B_{\R})\!-\!(k\!+\!2(l_0\!-\!l_0^*)\!\big)\!\big)
+2\big(\ell_{\om}(B_{\C})\!-\!(l_{\C}\!-\!l_{\C}^*)\!\big)=0.\EE

For a good choice of $\nu$, there exist \hbox{$\nu_{\R}\!\in\!\cH_{k,l_0+1}^{\om,\phi}$},
\hbox{$\nu_{\C}\!\in\!\cH_{l_{\C}+1}^{\om}$}, and a natural embedding
\BE{Gasplit_e1} 
\io_{\cS}\!:\cS\lhra{~~~} \M_\R\times\!\M_\C\!\equiv\!
\M_{k,l_0+1}\big(B_\R;J,\nu_\R\big)\!\times\!\M^\C_{l_\C+1}\big(B_\C;J,\nu_\C\big)\EE
satisfying~\ref{GasplitC_it}~-~\ref{GasplitIndex_it} in Section~\ref{NBstrata_subs}.
If $B_{\R}\!\neq\!0$, we also assume that there exists \hbox{$\nu_{\R}'\!\in\!\cH_{k,l_0}^{\om,\phi}$}
so that the forgetful morphism
\BE{Gasplit_e2a}\ff_{\nod}\!:\M_\R
\lra \M_\R'\equiv\!\M_{k,l_0}\big(B_\R;J,\nu'_\R\big)\EE
dropping the conjugate pair corresponding to the node~$\nod$ (i.e.~the first one)
is defined.
If $B_{\C}\!\neq\!0$, we similarly assume that there exists $\nu_{\C}'\!\in\!\cH_{l_{\C}}^{\om}$
so that the analogous forgetful morphism
\BE{Gasplit_e2b}\ff_{\nod}\!:\M_\C
\lra \M_\C'\equiv\!\M^\C_{l_\C}\big(B_\C;J,\nu'_\C\big)\EE
is defined.
Denote by $I\!\subset\![l_\C+1]$ (resp.~$I'\!\subset\![l_\C]$) the subset indexing 
the marked points of a map in $\M_\C$ (resp.~$\M_\C'$) 
corresponding to the marked points on the left-hand side of~\eref{Gasplit_e1}
indexed by~$L^*_-$ under~\eref{Gasplit_e1} (resp.~\eref{Gasplit_e1} and~\eref{Gasplit_e2b}). 

For an element $\u\!\in\!\cS$, we denote~by $\u_0\!\in\!\M_\R$ and $\u_+\!\in\!\M_\C$
the pair of maps corresponding to~$\u$ via~\eref{Gasplit_e1}.
Let $\u_0'\!\in\!\M'_\R$ and $\u_+'\!\in\!\M'_\C$
be the image of~$\u_0$ under~\eref{Gasplit_e2a} if $B_{\R}\!\neq\!0$ and 
the image of~$\u_+$ under~\eref{Gasplit_e2b} if $B_{\C}\!\neq\!0$, respectively.
The exact sequence
$$0\lra T_{\u}\cS\lra T_{\u_0}\M_\R\!\oplus\!T_{\u_+}\M_\C\lra T_{u(\nod)}X\lra0,~~
\big(\xi_1,\xi_2\big)\lra \xi_2(\nod)\!-\!\xi_1(\nod),$$
of vector spaces determines an isomorphism 
$$\la_{\u}(\cS)\!\otimes\!\la\big(T_{u(\nod)}X\big)
\approx \la_{\u_0}(\M_\R)\!\otimes\!\la_{\u_+}(\M_\C).$$
The $\Pin^-$-structure~$\fp$ on~$X^{\phi}$ determines a homotopy class~$\fo_{\fp;l_0^*}$ of isomorphisms
$$\la_{\u_0}(\M_\R)\lra  \la_{\u_0}^{\R}(X)\!\otimes\!\la_{\u_0}^{\C}(X);$$
see Lemma~\ref{orient_lmm}.
Combining the above two homotopy classes of isomorphisms with the complex orientations of 
$\la(T_{u(\nod)}X)$ and~$\la_{\u_+}^{\C}(X)$ and the orientation~$\fo_{\C;I}$ of $\la_{\u_+}(\M_\C)$,
we obtain a homotopy class~$\fo_{\fp;l_0^*;\u}^{\Ga}$ of isomorphisms
$$\la_{\u}(\cS)\lra \la_{\u}^{\R}(X)\!\otimes\!\la_{\u}^{\C}(X).$$

\begin{lmm}\label{TMcompGa_lmm}
The relative orientations $(\ff^*_{k',l'}\fo^c_\Ga)\fo_{\fp;l^*}$ 
and $\fo_{\fp;l_0^*}^{\Ga}$ of $\la(\ev_\Ga)$ are the same.
\end{lmm}

\begin{proof}
The proof of Lemma~\ref{TMcomp_lmm} readily adapts using Lemma~\ref{cNGa2_lmm}.
The relevant analogue of Lemma~\ref{DorientComp_lmm} follows readily 
from~\cite[Cor.~7.3]{SpinPin}.
In light of Lemma~\ref{cNGa2_lmm}, the $(k',l')=(0,3)$ case of Lemma~\ref{TMcompGa_lmm}
also follows from Lemma~5.2 and Remark~5.3 in~\cite{RealEnum}; 
the proof in~\cite{RealEnum} extends to arbitrary $(k',l')$. 
\end{proof}

We denote by 
$$\bh_\R\!:M_{\bh_\R}\lra X^{l^*_0}, \quad \bh_\C\!:M_{\bh_\C}\lra X^{l^*_\C}, \quad
\bp_\R\in X_{k,l_0-l^*_0}, \quad\hbox{and}\quad \bp_\C\in X^{l_\C-l^*_\C}$$
the components of $\bh$ and $\bp$ corresponding to the marked points on $\P^1_0$ and $\P^1_\pm$ 
for the maps~$\u$ in~$\cS$. 
With the notation as in~\eref{Mhdfn_e}, let
$$\cS_\bh={\cS}{}_{\ev}\!\!\times_{f_{\bh}}\!M_{\bh},\quad
\cZ_\R=({\M_\R})_{\ev}\!\!\times_{f_{\bh_\R}}\!M_{\bh_\R}, \quad
\cZ_\C=(\M_\C)_{\ev^I}\!\!\times_{f_{\bh_\C}}\!M_{\bh_\C}$$
be the corresponding spaces cut out by~$\bh$ and 
$$\ev_{\C}\!:\cZ_\C\lra X^{l_\C-l^*_\C}$$
be the evaluation map induced by~$\ev^I$. 
The relative orientations~$\fo_{\fp;l_0^*+1}$ and~$\fo_{\C;I}$ of
$$\ev\!:\M_\R\lra X_{k,l_0+1} \qquad\hbox{and}\qquad
\ev^I\!:\M_\C\lra X^{l_{\C}+1}\,,$$
respectively, the orientations $\fo_{h_i}$ of $H_i$, and the symplectic 
orientation~$\fo_{\om}$ on~$X$ determine relative orientations~$\fo_{\fp;\bh_\R}$  
and~$\fo_{\bh_{\C};I}$ of the induced maps
$$\ev_{\R}\!:\cZ_\R\lra X_{k,l_0-l^*_0}  \qquad\hbox{and}\qquad 
\ev_{\C}\!:\cZ_\C\lra X^{l_\C-l^*_\C}\,$$
respectively.

For $\wt\u\!\in\!\cS_\bh$, let 
$\wt\u_0\!\in\!\cZ_\R$ and $\wt\u_+\!\in\!\cZ_\C$ be the components 
of $\wt\u$ in the corresponding spaces. 
The exact sequence 
\BE{Cdecomppf_e4}0\lra T_{\wt\u}\cS_\bh\lra T_{\wt\u_0}\cZ_\R\!\oplus\!T_{\wt\u_+}\cZ_\C
\lra T_{u(\nod)}X\lra0\EE
of vector spaces determines an isomorphism
$$\la_{\wt\u}(\cS_\bh)\!\otimes\!\la\big(T_{u(\nod)}X\big)
\approx \la_{\wt\u_0}(\cZ_\R)\!\otimes\!\la_{\wt\u_+}(\cZ_\C).$$
Along with $\fo_{\fp;\bh_\R}$ and~$\fo_{\bh_{\C};I}$, 
these isomorphisms determine a relative orientation
\BE{Cdecomppf_e6}
\fo_{\fp;\bh}^{\Ga}\equiv \big(\fo_{\fp;\bh_\R}\big){}_{\nod}\!\times\!_{\nod}\fo_{\bh_{\C};I}\EE
of the restriction of~\eref{evGadfn_e} to~$\cS_{\bh}$.
Since the dimensions of $H_i$ and $X$ are even, 
Lemma~\ref{TMcompGa_lmm} implies~that 
\BE{Cdecomppf_e11}\big|\ev_{\Ga;\bh}^{-1}(\bp)\!\cap\!\cS_\bh\big|_{\fo_{\Ga;\fp;\bh}}^{\pm}
=\big|\ev_{\Ga;\bh}^{-1}(\bp)\!\cap\!\cS_\bh\big|_{\fo_{\fp;\bh}^{\Ga}}^{\pm}\,.\EE
We denote by $\fo_{\cS}^c$ the co-orientation of $\cS$ in \hbox{$\cZ_{\R}\!\times\!\cZ_{\C}$}
determined by the symplectic orientation~$\fo_{\om}$ of~$X$ via~\eref{Cdecomppf_e4}.

If $B_{\R}\!\neq\!0$ (resp.~$B_{\C}\!\neq\!0$), we also define
$$\cZ'_\R=(\M'_\R)_{\ev}\!\!\times_{\bh_\R}\!M_{\bh_\R}
\qquad\big(\hbox{resp.}~\cZ'_\C=(\M'_\C)_{\ev^{I'}}\!\!\times_{\bh_\C}\!M_{\bh_\C}\big)$$
and denote by $\wt\u'_0\!\in\!\cZ'_\R$ (resp.~$\wt\u'_+\!\in\!\cZ'_\C$)
the image of~$\wt\u_0$ (resp.~$\wt\u_+$) under the forgetful morphism
\BE{cZGasplit_e1c}\ff_{\R}\!:\cZ_{\R}\lra \cZ_{\R}' \qquad
\big(\hbox{resp.}~\ff_{\C}\!:\cZ_{\C}\lra \cZ_{\C}'\big)\EE
dropping the marked points corresponding to the nodes.
If $B_{\R}\!\neq\!0$ and $l_0^*\!\neq\!0$, $\fo_{\fp;l_0^*}$ 
determines a relative orientation~$\fo_{\fp;\bh_\R}'$ of the evaluation~map
\BE{cZGasplit_e2a}\ev_{\R}'\!:\cZ_\R'\lra X_{k,l_0-l^*_0}\EE
induced by~$\ev$.
If $B_{\C}\!\neq\!0$, $\fo_{\C;I}$ determines a relative orientation~$\fo_{\bh_{\C};I'}'$ 
of the evaluation~map
\BE{cZGasplit_e2b}\ev_{\C}'\!:\cZ_\C'\lra X^{l_\C-l^*_\C}\EE
induced by~$\ev^{I'}$.

Since the projections $\pi_{\R}$ and~$\pi_{\C}$ in the commutative diagrams
$$\xymatrix{ \cZ_{\R} \ar[rr]^{\ff_{\R}} \ar[d]_{\pi_{\R}}&& \cZ_{\R}'\ar[d]^{\pi_{\R}'}&& 
\cZ_{\C} \ar[rr]^{\ff_{\C}} \ar[d]_{\pi_{\C}}&& \cZ_{\C}'\ar[d]^{\pi_{\C}'}\\
\M_{\R} \ar[rr]^{\ff_{\nod}} && \M_{\R}'&&
\M_{\C} \ar[rr]^{\ff_{\nod}} && \M_{\C}'}$$
induce isomorphisms between the vertical tangent bundles of $\ff_{\R}$, $\ff_{\C}$,
and $\ff_{\nod}$, they pull back the relative orientations~$\fo_{\nod}^+$ 
of~\eref{Gasplit_e2a} and~\eref{Gasplit_e2b}, respectively, to 
relative orientations~$\fo_{\nod}^+$ of~\eref{cZGasplit_e1c}.
Since the dimensions of $X$ and $H_i$ are~even, 
\BE{Cdecomppf_e8} 
\fo_{\fp;\bh_\R}=\fo_{\R}^+\fo_{\fp;\bh_\R}'\equiv
\big(\pi_{\R}^*\fo_{\nod}^+\big)\fo_{\fp;\bh_\R}' \quad\hbox{and}\quad
\fo_{\bh_{\C};I}=\fo_{\C}^+\fo_{\bh_{\C};I'}'\equiv
\big(\pi_{\C}^*\fo_{\nod}^+\big)\fo_{\bh_{\C};I'}'\,,\EE
whenever $B_{\R}\!\neq\!0$ and $B_{\C}\!\neq\!0$, respectively.

Suppose $B_\R,B_\C\!\neq\!0$.
By~\eref{Cdecomp_e3}, we can assume~that  
\BE{Cdecomppf_e10} 
\ell_{\om}(B_{\R})=k\!+\!2(l_0\!-\!l_0^*) \qquad\hbox{and}\qquad
\ell_{\om}(B_{\C})=l_{\C}\!-\!l_{\C}^*;\EE
otherwise, either $\ev_\R^{-1}(\bp_{\R})\!=\!\eset$ or 
$\ev_\C^{-1}(\bp_{\C})\!=\!\eset$ for a generic $\bp\!\in\!X_{k,l-l^*}$.
By the first statement above and Lemma~\ref{orient_lmm}\ref{orient1pm_it}, 
the interchange of the marked points of the elements in~$\cZ_{\R}$ 
reverses the orientation~$\fo_{\fp;\bh_\R}$.
This interchange also reverses the vertical orientation~$\fo_{\R}^+$ of
the fibration~$\ff_{\R}$.
Thus, the orientation~$\fo_{\fp;\bh_\R}'$ on~$\cZ_{\R}'$ can be defined by 
the first equation in~\eref{Cdecomppf_e8} if $l_0^*\!=\!0$.
By the second equation in~\eref{Cdecomppf_e10} and the assumption that 
$\phi_*[h_i]\!=\!-[h_i]$ for every $i\!\in\![l^*]$,
\BE{CdivRel_e} \big|\ev_{\C}'^{-1}(\bp_\C)\big|^\pm_{\fo_{\bh_\C}}
=\big|(\M'_\C)_{\ev^I}\!\!\times_{(\bh_\C,\bp_\C)}\!
M_{\bh_\C}\big|^\pm_{\fo_{\C;I},\fo_{\bh_{\C}}}
=\prod_{i\in L^*_\C}\!\!(h_i\cdot_X\!B_\C)N^X_{B_\C}\,;\EE
see \cite[Prop.~4.3]{RealEnum}. 

\begin{figure}
$$\xymatrix{&&&&0\ar[d]\\
&&&&T_{\nod}\P^1_0\!\oplus\!T_{\nod}\P_+^1\ar[d]\ar[rd]|{\nd_{\nod}u_+-\nd_{\nod}u_0}\\ 
0\ar[r]&T_{\wt\u}\cS_\bh\ar[d]_{\nd_{\wt\u}\ev_{\Ga;\bh}}\ar[rrr]^{\nd_{\wt\u}\io_{\cS}}&&&
T_{\wt\u_0}\cZ_\R\!\oplus\!T_{\wt\u_+}\cZ_\C \ar[r]
\ar[d]^{\nd_{\wt\u_0}\ff_{\R}\oplus\nd_{\wt\u_+}\ff_{\C}} &  T_{u(\nod)}X\ar[r] & 0 \\
& T_{\ev(\wt\u)}X_{k,l-l^*} &&&
T_{\wt\u'_0}\cZ'_\R\!\oplus\!T_{\wt\u'_+}\cZ'_\C\ar[d]
\ar[lll]_{(\nd_{\wt\u'_0}\ev_{\R}',\nd_{\wt\u'_+}\ev_\C')}  \\
&&&& 0}$$
\caption{Computation of~\eref{Cdecomppf_e11} in the $B_\R,B_\C\!\neq\!0$ case}
\label{Ga1_fig}
\end{figure}

Since the real dimension of~$X$ is~4, the homomorphism 
$$\nd_\nod u_+\!-\!\nd_\nod u_0\!:
T_{\nod}\P^1_0\!\oplus\!T_{\nod}\P_+^1\lra T_{u(\nod)}X\!=\!\cN_{\wt\u}\cS_\bh$$
in the commutative diagram of Figure~\ref{Ga1_fig} is an isomorphism 
for a generic element $\u\!\in\!\cS_\bh$.
So is the bottom row in this diagram.
The number of the preimages $(\wt\u_0,\wt\u_+)$ of a generic point~$(\u_0,\u_+)$
of \hbox{$\{\cZ_{\R}'\!\times\!\cZ_{\C}'\}\!\circ\!\io_{\cS}$}
under $\ff_{\R}\!\times\!\ff_{\C}$ is finite.
Since the dimensions of $T_{\nod}\P^1_0$ and $T_{\nod}\P_+^1$ are even,
$$\fo^v_{\R\C}\equiv \fo_{\R}^+\!\oplus\!\fo_{\C}^+$$ 
is the vertical orientation~of the fibration $\ff_{\R}\!\times\!\ff_{\C}$.
By Lemma~\ref{fibrasign_lmm1a}\ref{fibisom_it} and the reasoning 
in the proof of~\eref{RdivRel_e} at the end of Section~\ref{orient_subs},
$$\fs_{(\wt\u_0,\wt\u_+)}(\fo_{\cS}^c\fo_{\R\C}^v)\in\big\{\pm1\big\}$$ 
is the sign of the intersection
of $u_0$ and~$u_1$ at~$\nod$ and the number of such preimages counted with sign is
$B_\R\!\cdot_X\!B_\C$.
Along with the commutativity of the square in Figure~\ref{Ga1_fig}
and \eref{fsfoprod_e}, this implies~that  
$$\big|\ev_{\Ga;\bh}^{-1}(\bp)\!\cap\!\cS_\bh\big|_{\fo_{\fp;\bh}^{\Ga}}^{\pm}
=\big(B_\R\!\cdot_X\!\!B_\C\big) \big|\ev_{\R}'^{-1}(\bp_\R)\big|^\pm_{\fo_{\fp;\bh_\R}'}
\big|\ev_{\C}'^{-1}(\bp_\C)\big|^\pm_{\fo_{\bh_\C;I'}'}\,.$$
Combining this statement with~\eref{Cdecomppf_e11}, \eref{RdivRel_e},
and~\eref{CdivRel_e}, we conclude~that
$$\big|\ev_{\Ga;\bh}^{-1}(\bp)\!\cap\!\cS_\bh\big|_{\fo_{\Ga;\fp;\bh}}^{\pm}
=\big(B_\R\!\cdot_X\!\!B_\C\big) \!\bigg(\prod_{i\in L^*_0}\!\!h_i\!\cdot_X\!\!B_\R\!\!\bigg)
\!\bigg(\prod_{i\in L^*_\C}\!\!h_i\!\cdot_X\!\!B_\C\!\!\bigg)
N^{\phi;\fp}_{B_\R,l_0-l_0^*}N^X_{B_\C}\,.$$
Summing this over all possibilities for~$\cS$ with $B_{\R},B_{\C}\!\neq\!0$ that 
satisfy~\eref{Cdecomppf_e10}, we obtain the $(B_0,B')$ sum in~\eref{Cdecomp_e}.

Suppose $B_{\C}\!=\!0$ and thus $B_{\R}\!=\!B$.
We can assume that $l_{\C}^*\!=\!l_{\C}\!=\!2$; otherwise, $\cZ_{\C}\!=\!\eset$ 
for generic~$\bh$ and~$\bp$.
Thus, $\cZ_{\C}$ is a finite collection of points, while 
the dimensions of $\cZ_{\R}$ and \hbox{$X\!\times\!X_{k,l-l^*}$} are the same
by~\eref{Cdecomp_e3} and~\eref{Cdecomp_e0}.
By~\eref{Cdecomppf_e6} and Lemma~\ref{dropfactor_lmm3} with $\Ups,\cM\!=\!\pt$ 
applied to the left diagram in Figure~\ref{Ga2_fig},
\BE{Cdecomppf_e21}
\big|\ev_{\Ga;\bh}^{-1}(\bp)\!\cap\!\cS_\bh\big|_{\fo_{\fp;\bh}^{\Ga}}^{\pm}
=\big|\cZ_{\C}\big|_{\fo_{\C;I}}^{\pm}
\big|\big\{\ev_{\nod}\!\times\!\ev_{\R}\!\big\}^{-1}(\pt,\bp)\big|^\pm_{\fo_{\fp;\bh_\R}\fo_{\om}}.
\EE
By the proof of~\eref{RdivRel_e},
$$\big|\big\{\ev_{\nod}\!\times\!\ev_{\R}\!\big\}^{-1}(\pt,\bp)\big|^\pm_{\fo_{\fp;\bh_\R}\fo_{\om}}
=\bigg(\prod_{i\in L_0^*}\!\!\!h_i\!\cdot_X\!\!B\!\!\bigg)
N_{B,l_0-l_0^*+1}^{\phi;\fp}$$
Combining this statement with~\eref{Cdecomppf_e11} and \eref{Cdecomppf_e21},
we obtain the second term on the right-hand side of~\eref{Cdecomp_e}.

\begin{figure}
$$\xymatrix{ &\cZ_\C\ar[rr]^{\ev_{\C}} \ar[dl]_{\ev_\nod} && \pt &&
&\cZ_\R\ar[rr]^{\ev_{\R}} \ar[dl]_{\ev_\nod} && X^{\phi}\\
X & \cS_\bh \ar[u]\ar[d]\ar[rr]^{\ev_{\Ga;\bh}}&& X_{k,l-l^*} \ar[u]\ar@{=}[d] &&
X & \cS_\bh \ar[u]\ar[d]\ar[rr]^{\ev_{\Ga;\bh}}&& X_{1,l-l^*} \ar[u]\ar[d]\\
&\cZ_\R\ar[rr] \ar[ul]^{\ev_\nod}\ar[rr]^{\ev_\R} && X_{k,l-l^*} &&
&\cZ_\C\ar[rr] \ar[ul]^{\ev_\nod}\ar[rr]^{\ev_\C} && X^{l-l^*}}$$
\caption{Computation of~\eref{Cdecomppf_e11} in the $B_\C\!=\!0$ and $B_\R\!=\!0$ cases}
\label{Ga2_fig}
\end{figure}

Suppose $B_{\R}\!=\!0$ and thus $\fd(B_{\C})\!=\!B$.
We can assume that $k',k\!=\!1$ and $l_0\!=\!0$;
otherwise, $\cZ_{\R}\!=\!\eset$ for generic~$\bh$ and~$\bp$.
Thus, $\lr{B;k}_{\Ga}\!=\!1$, the dimensions of $\cZ_{\R}$ and~$X^{\phi}$ are the same,
and so are the dimensions of $\cZ_{\C}$ and \hbox{$X\!\times\!X_{l-l^*}$}
by~\eref{Cdecomp_e3} and~\eref{Cdecomp_e0}.
By~\eref{Cdecomppf_e6} and Lemma~\ref{dropfactor_lmm3} with $\Ups,\cM\!=\!\pt$ 
applied to the right diagram in Figure~\ref{Ga2_fig},
$$\big|\ev_{\Ga;\bh}^{-1}(\bp)\!\cap\!\cS_\bh\big|_{\fo_{\fp;\bh}^{\Ga}}^{\pm}
=\big|\ev_{\R}^{-1}(p_1^{\R})\big|^{\pm}_{\fo_{\fp;1}}
\big|\big\{\ev_{\nod}\!\times\!\ev_{\C}\!\big\}^{-1}(\pt,\bp_{\C})\big|^\pm_{\fo_{\C;I}\fo_{\om}}.$$
Combining this statement with~\eref{Cdecomppf_e11}, Lemma~\ref{orient_lmm}\ref{orient0_it}, 
and~\eref{CdivRel_e}, we conclude~that
$$\big|\ev_{\Ga;\bh}^{-1}(\bp)\!\cap\!\cS_\bh\big|_{\fo_{\Ga;\fp;\bh}}^{\pm}
=\bigg(\prod_{i\in[l^*]}\!\!\!h_i\!\cdot_X\!\!B_{\C}\!\!\bigg)\!N^X_{B_{\C}}
=2^{-l^*}\!\bigg(\prod_{i\in[l^*]}\!\!\!h_i\!\cdot_X\!\!B\!\!\bigg)\!N^X_{B_{\C}}\,.$$
Summing this over all possibilities for~$\cS$ with $B_{\R}\!=\!0$, 
we obtain the first term on the right-hand side of~\eref{Cdecomp_e}.


\noindent
{\it Department of Mathematics, Stony Brook University, Stony Brook, NY 11794\\
xujia@math.stonybrook.edu}


\begin{thebibliography}{99}

\bibitem{Adam} A.~Alcolado
{\it Extended Frobenius manifolds and the open WDVV equations},
PhD thesis, McGill,~2017 

\bibitem{ABL} A.~Arroyo, E.~Brugall\'e, and L.~L\'opez de Medrano, 
{\it Recursive formulas for Welschinger invariants of the projective plane}, 
Int.~Math.~Res.~Not.~2011, no.~5, 1107–-1134

\bibitem{B15} E.~Brugall\'e,
{\it Floor diagrams relative to a conic, and GW-W invariants of del Pezzo surfaces}, 
Adv.~Math.~279 (2015), 438–-500

\bibitem{B18} E.~Brugall\'e,
{\it On the invariance of Welschinger invariants}, math/1811.06891

\bibitem{BP13} E.~Brugall\'e and N.~Puignau,
{\it Behavior of Welschinger invariants under Morse simplifications}, 
Rend.~Semin.~Mat.~Univ.~Padova 130 (2013), 147–-153

\bibitem{RealWDVVapp} X.~Chen, {\it Solomon's relations for Welschinger's invariants: examples}, 
math/1809.08938

\bibitem{RealWDVV3} X.~Chen and A.~Zinger, 
{\it WDVV-type relations for disk Gromov-Witten invariants in dimension~6}, 
math/1904.04254

\bibitem{SpinPin} X.~Chen and A.~Zinger, {\it Spin/Pin Structures and Real Enumerative Geometry}, 
math/1905.11316

\bibitem{Penka2} P.~Georgieva,
{\it Open Gromov-Witten invariants in the presence of an anti-symplectic involution},
 Adv.~Math.~301 (2016), 116-–160

\bibitem{RealEnum} P.~Georgieva and A.~Zinger,
{\it Enumeration of real curves in $\C\P^{2n-1}$ and a WDVV relation for 
real Gromov-Witten invariants}, Duke Math.~J.~166 (2017), no.~17, 3291-–3347

\bibitem{RealGWsI} P.~Georgieva and A.~Zinger,
{\it Real Gromov-Witten theory in all genera and real enumerative geometry: construction},
Ann.~of Math.~188 (2018), no.~3, 685-–752

\bibitem{RealGWsII} P.~Georgieva and A.~Zinger,
{\it Real Gromov-Witten theory in all genera and real enumerative geometry: properties},
math/1507.06633v2

\bibitem{Gi01} A.~Givental, {\it Semisimple Frobenius structures at higher genus}, 
IMNR 2001, no.~23, 1265–-1286

\bibitem{HorevSol} A.~Horev and J.~Solomon,
{\it The open Gromov-Witten-Welschinger theory of blowups of the projective plane}, 
math/1210.4034

\bibitem{IKS09} I.~Itenberg, V.~Kharlamov, and E.~Shustin, 
{\it A Caporaso-Harris type formula for Welschinger invariants of real toric del Pezzo surfaces}, 
Comment.~Math.~Helv.~84 (2009), no.~1, 87–-126

\bibitem{IKS13} I.~Itenberg, V.~Kharlamov, and E.~Shustin, 
{\it Welschinger invariants of small non-toric Del Pezzo surfaces}, 
JEMS~15 (2013), no.~2, 539–-594

\bibitem{IKS13a} I.~Itenberg, V.~Kharlamov, and E.~Shustin, 
{\it Welschinger invariants of real Del Pezzo surfaces of degree $\ge\!3$}, 
Math.~Ann.~355 (2013), no.~3, 849--878

\bibitem{IKS13b} I.~Itenberg, V.~Kharlamov, and E.~Shustin,
{\it Welschinger invariants of real del Pezzo surfaces of degree $\ge\!2$},
Internat.~J.~Math.~26 (2015), no.~8, 1550060, 63pp

\bibitem{KM} M.~Kontsevich and Y.~Manin, 
{\it Gromov-Witten classes, quantum cohomology, and enumerative geometry}, 
Comm.~Math.~Phys.~164 (1994), no.~3, 525--562

\bibitem{LT}  J.~Li and G.~Tian, 
\emph{Virtual moduli cycles and Gromov-Witten invariants of general symplectic manifolds}, 
Topics in Symplectic \hbox{$4$-Manifolds},
47-83, First Int.~Press Lect.~Ser., I, Internat.~Press, 1998

\bibitem{MS} D.~McDuff and D.~Salamon, 
{\it $J$-holomorphic Curves and Symplectic Topology},
Colloquium Publications~52, AMS, 2012

\bibitem{RT} Y.~Ruan and G.~Tian, 
{\it A mathematical theory of quantum cohomology}, 
J.~Differential Geom.~42 (1995), no.~2, 259--367

\bibitem{RT2} Y.~Ruan and G.~Tian,
\emph{Higher genus symplectic invariants and sigma models coupled with gravity},
Invent.~Math.~130 (1997), no.~3, 455--516

\bibitem{Jake} J.~Solomon,  
{\it Intersection theory on the moduli space of holomorphic curves with 
Lagrangian boundary conditions}, math/0606429

\bibitem{Jake2} J.~Solomon, 
{\it A differential equation for the open Gromov-Witten potential},
note 2007

\bibitem{Steenrod} N.~Steenrod, {\it Homology with local coefficients}, 
Ann.~of Math.~44 (1943), no.~4, 610-–627

\bibitem{Teleman} C.~Teleman,
{\it The structure of 2D semi-simple field theories}, 
Invent.~Math.~188 (2012), no.~3, 525–-588

\bibitem{Warner} F.~Warner,
{\it Foundations of Differentiable Manifolds and Lie Groups}, 
GTM~94, Springer-Verlag,~1983

\bibitem{Wel4} J.-Y.~Welschinger, 
{\it Invariants of real symplectic 4-manifolds and lower bounds in real enumerative geometry}, 
Invent.~Math.~162 (2005), no.~1, 195–-234

\bibitem{Wel6} J.-Y.~Welschinger,  
{\it Spinor states of real rational curves in real algebraic convex 3-manifolds 
and enumerative invariants},
Duke Math.~J.~127 (2005), no.~1, 89–-121

\bibitem{Wel07} J.-Y.~Welschinger, 
{\it Optimalit\'e, congruences et calculs d'invariants des vari\'ert\'es symplectiques r\'eelles 
de dimension quatre}, math/0707.4317

\bibitem{pseudo} A.~Zinger, {\it Pseudocycles and integral homology},
Trans.~Amer.~Math.~Soc.~360 (2008), no.~5, 2741--2765

\bibitem{RealRT} A.~Zinger,
{\it Real Ruan-Tian perturbations}, math/1701.01420


\end{thebibliography}
\end{document}